\documentclass[10pt]{amsart}
\usepackage{graphicx,amssymb,amsfonts,amsmath,amsthm,newlfont}
\usepackage{epsfig,url}
\usepackage{color}
\usepackage{enumerate}
\usepackage[colorlinks=true,linkcolor=red,citecolor=blue]{hyperref}

\usepackage[all,2cell]{xy} \UseAllTwocells \SilentMatrices
\usepackage{bbm}

\newtheorem{thm}{Theorem}[section]
\newtheorem{cor}[thm]{Corollary}
\newtheorem{lem}[thm]{Lemma}
\newtheorem{prop}[thm]{Proposition}
\theoremstyle{definition}
\newtheorem{defn}[thm]{Definition}

\newtheorem{example}[thm]{Example}

\theoremstyle{remark}
\newtheorem{rem}[thm]{Remark}

\numberwithin{equation}{section}

\newcommand{\Z}{\mathbb Z}
\newcommand{\C}{\mathbb C}

\newcommand{\R}{\mathbb R}

\newcommand{\Pro}{\mathbb P}

\newcommand{\Lef}{\mathbb{L} }

\font \rus= wncyr10
\newcommand{\sha}{\, \hbox{\rus x} \,}

\newcommand{\MT}{\mathcal{MT}}

\newcommand{\ZZ}{\mathcal{Z}}

\newcommand{\Real}{\mathrm{Re}}

\newcommand{\zetam}{\zeta^{ \mathfrak{m}}}

\newcommand{\lf}{\mathrm{lf}}

\newcommand{\Q}{\mathbb Q}

\newcommand{\s}{\mathbf{s}}

\newcommand{\To}{\longrightarrow}

\newcommand{\A}{\mathbb{A}}
\newcommand{\G}{\mathbb{G}}

\newcommand{\dR}{\mathfrak{dr}}

\newcommand{\Or}{\mathcal{O}}

\newcommand{\Res}{\mathrm{Res}}

\newcommand{\mm}{\mathfrak{m} }

\newcommand{\id}{\mathrm{id} }

\newcommand{\ren}{\mathrm{ren}}

\newcommand{\per}{\mathrm{per}}

\newcommand{\Pe}{\mathcal{P}}

\newcommand{\comp}{\mathrm{comp}}

\newcommand{\Zim} {\ZZ^{i, \mm}}
\newcommand{\Zio} {\ZZ^{i,\varpi}}

\newcommand{\Zjo} {\ZZ^{j,\varpi}}

\newcommand{\dch}{\mathrm{dch}}

\newcommand{\LP}{L}
\newcommand{\FL}{F\!L}

\usepackage{geometry}
\begin{document}
\author{Francis Brown}
\address{All Souls College, Oxford, Oxford OX1 4AL, UK}
\email{francis.brown@all-souls.ox.ac.uk}
\author{Cl\'{e}ment Dupont}
\address{Institut Montpelli\'{e}rain Alexander Grothendieck, Universit\'{e} de Montpellier, CNRS, Montpellier, France}
\email{clement.dupont@umontpellier.fr}
\begin{title}[Motivic  Galois actions on  Lauricella hypergeometric functions ]{Lauricella hypergeometric functions,  unipotent fundamental groups  of the punctured Riemann sphere, and their motivic coactions}\end{title}
\maketitle

\begin{abstract}
The goal of this paper is to raise the possibility that there exists a meaningful theory of `motives' associated to certain hypergeometric integrals, viewed as functions of their parameters.  It goes beyond the classical theory of motives, but should be compatible with it.   Such a theory would explain a recent and surprising conjecture
arising in the context of scattering amplitudes  for a  motivic   Galois group action on Gauss' ${}_2F_1$ hypergeometric function, which we prove in this paper by direct means.  More generally, we  consider Lauricella hypergeometric functions and show on the one hand how the coefficients in their Taylor expansions  can be promoted, via the theory of motivic fundamental groups, to motivic multiple polylogarithms.  The latter are periods of ordinary motives and admit an  action of the usual motivic Galois group, which we call the `local' action. On the other hand, we define lifts of the full Lauricella functions as matrix coefficients in a Tannakian category of  twisted cohomology,  which  inherit an  action of the corresponding   Tannaka group.  We call this the `global' action. We prove that these two actions, local and global, are compatible with each other, even though they are defined in completely different ways.
The main technical tool is to prove that metabelian quotients of generalised Drinfeld associators on the punctured Riemann sphere are  hypergeometric functions. 
We also  study  single-valued versions of these hypergeometric functions, which may be of independent interest.
\end{abstract}

\section{Introduction}
Let $\Sigma=\{ \sigma_0, \sigma_1,\ldots, \sigma_n\}$ be distinct points in $\C$, where $\sigma_0=0$. 
In this paper, we study the Lauricella hypergeometric functions  with singularities in $\Sigma$ which are defined by
\begin{equation} \label{IntroLSigma} (L_{\Sigma})_{ij} \ =  \  -   s_j  \int_{0}^{\sigma_i} x^{s_0} \prod_{k=1}^n (1- x \sigma_k^{-1})^{s_k}  \, \frac{dx}{x-\sigma_j}  \ ,\quad \qquad \hbox{ for } 1\leq i, j \leq n  \ .
\end{equation}
The most familiar  examples are Euler's beta function $\beta(a,b)= \frac{\Gamma(a) \Gamma(b)}{\Gamma(a+b)}$ (see \S\ref{sect: introBeta}) and Gauss' hypergeometric function $F={}_2F_1$ (see \S\ref{section: 2F1}), which satisfies the integral formula
\begin{equation}\label{eq: intro 2F1 integral}
\beta(b,c-b) \, F( a,b,c;y ) =    \int_{0}^{1}  x^{b-1} (1-x)^{c-b-1} (1-yx)^{-a} dx
\end{equation}
whenever it converges. It can be expressed in terms of \eqref{IntroLSigma} for $\Sigma = \{0,1,y^{-1}\}$. \medskip

There are two possible  ways in which one might try to define  a `motivic' Galois group acting on these functions using Tannakian theory, by viewing them in one of the following ways.  \begin{enumerate}
\item[(G)] Globally,  for generic values of the  exponents $s_k \in \C$, where genericity means that the $s_k \notin \Z$ for each $k$ and $s_0+\cdots +s_n \notin \Z.$
For such generic $s_k$,  it is known \cite{aomotovanishing, DeligneMostow} how to interpret \eqref{IntroLSigma} as `periods'\footnote{These numbers are not periods in the sense of Kontsevich--Zagier \cite{kontsevichzagier} unless the exponents $s_i$ are rational.} of the cohomology of $X_{\Sigma} = \mathbb{A}^1\backslash \Sigma$ with coefficients in a  rank one algebraic vector bundle with integrable connection (or local system).   One can interpret these twisted cohomology groups as realisations of objects in a suitable Tannakian category, and hence interpret  them as representations of a `global' Tannaka group.

 \vspace{0.05in}
\item[(L)]  Locally, as formal power series in the $s_k$ around the non-generic point $s_0=\cdots =s_n=0$. Even though the integral in \eqref{IntroLSigma} is divergent at that point, the prefactor $s_j$ compensates the pole and \eqref{IntroLSigma} has a Taylor expansion in the $s_k$ at the origin. Its coefficients are generalised polylogarithms, which can be lifted  to periods of mixed Tate motives. They admit an action of the usual motivic Galois group, which acts term by term on coefficients in the series expansion. 
\end{enumerate}

As is customary in the Tannakian formalism,  it is easier to compute the \emph{motivic coaction} of the Hopf algebra of functions on the motivic Galois group, which is dual to the group action.    In this article we compute the global (G) and local (L) motivic coactions and prove that they are formally identical. This strongly suggests that there exists an underlying `Lauricella motive', \emph{viewed as a function of the} $s_k$, of which (G) and (L) are two different incarnations, or realisations. The realisation (G) is obtained by specialising $s_k$ to generic  complex numbers; the realisation (L) is obtained by expanding locally in the $s_k$. Such a theory  has a consequence for ordinary motives: it implies that the ordinary motivic Galois group acts on the coefficients of the power series expansion (L) in a uniform way. This is surprising since in general these coefficients are periods of unrelated  motives.
For example, in the case of the Euler beta function, it implies that the motivic Galois group acts  uniformly on all (motivic) zeta values of all weights (see \S \ref{sect: introBeta}).  For the general Lauricella function \eqref{IntroLSigma}, the main geometric object in the local framework (L) is the mixed Tate motivic fundamental groupoid of the punctured Riemann sphere $X_\Sigma$ with respect to suitable tangential basepoints. \medskip

The impetus for this work came from a remarkable conjecture \cite{abreuetaldiagrammatictwoloop, abreuetalpositive} arising in the study of dimensionally-regularised one-loop and two-loop  Feynman amplitudes in $4-2 \varepsilon$ spacetime dimensions, which  can be expressed in terms of hypergeometric functions \cite{ abreuetalalgebraicstructure, abreuetaldiagrammatichopf}. It  was observed experimentally that the  motivic coaction (L),  computed order-by-order in an $\varepsilon$-expansion of $F(n_1+ a_1 \varepsilon, n_2+ a_2 \varepsilon, n_3+ a_3 \varepsilon ;y )$,  where $n_1, n_2, n_3, a_1, a_2, a_3$ are integers, could, at least to low orders in $\varepsilon$, be succinctly  packaged into a coaction formula on the hypergeometric function itself with only two terms. We give a rigorous sense to these statements and prove a complete (global and local) coaction formula for the hypergeometric function.\medskip

A large part of this paper is also devoted to studying the single-valued versions of the integrals $(L_{\Sigma})_{ij}$, defined by the following complex integrals:
\begin{equation}\label{IntroLSigmasv}
(L_\Sigma^\s)_{ij} = \frac{s_j}{2\pi i}\iint_{\mathbb{C}}|z|^{2s_0}\prod_{k=1}^n|1-z\sigma_k^{-1}|^{2s_k}\left(\frac{d\overline{z}}{\overline{z}-\overline{\sigma_i}}-\frac{d\overline{z}}{\overline{z}}\right)\wedge\frac{dz}{z-\sigma_j} \; ,\qquad \hbox{ for } 1\leq i, j \leq n  \ .
\end{equation}
The prototype for such functions is the single-valued (or `complex') beta function
$$\beta^{\,\s}(a,b)=\frac{1}{2\pi i}\iint_{\mathbb{C}}|z|^{2a}|1-z|^{2b}\frac{d\overline{z}}{\overline{z}(1-\overline{z})}\wedge\frac{dz}{z(1-z)}  = \frac{\Gamma(a)\,\Gamma(b)\,\Gamma(1-a-b)}{\Gamma(a+b)\,\Gamma(1-a)\,\Gamma(1-b)}\ \cdot$$
They are closely related to constructions in conformal field theory (see \cite[(E)]{BPZ}, \cite[\S4]{KZ}, \cite{DotsenkoFateev}) and have been studied from the global point of view (G) in \cite{aomotoselberg, mimachiyoshida, mimachicomplex}. We recast them and their double copy formulae in the general framework of single-valued integration developed in \cite{BD1, BD2}. In particular we prove that they can also be interpreted as single-valued versions of Lauricella hypergeometric functions in the local sense (L), i.e., by applying the single-valued period homomorphism term by term to the coefficients in the series expansion. As a special case, we define and study two single-valued versions of the hypergeometric function \eqref{eq: intro 2F1 integral}, one of which may be new.\medskip

In summary, we prove that both the Galois coaction and single-valued period map coincide for (G) and (L), in other words, they `commute' with Taylor expanding at $s_0=\cdots = s_n=0$.

\subsection{Contents}
This   paper is in two parts, corresponding to the two points of view (G) and (L).
In the first part (G), we interpret the Lauricella function as a matrix coefficient in a Tannakian category  of Betti and de Rham realisations of cohomology with coefficients. 
To make this a little more precise, 
 consider the trivial  algebraic vector bundle  of rank one  on  $X_{\Sigma}$     with the integrable connection  
$$\nabla_{\underline{s}} = d + \sum_{k=0}^n  s_k \frac{dx}{x-\sigma_k} \ \cdot$$  
Let $\mathcal{L}_{\underline{s}}$ be the rank one local system 
 generated by $x^{-s_0} \prod_{k=1}^n (1-x \sigma_k^{-1})^{-s_k}$, which is a flat section of $\nabla_{\underline{s}}$. For generic $s_0,\ldots, s_n$,  integration defines  a canonical pairing between algebraic de Rham cohomology  and locally finite homology groups
$$ H^1_{\mathrm{dR}}(X_{\Sigma}, \nabla_{\underline{s}})  \qquad \hbox{ and }   \qquad H^{\lf}_1(X_{\Sigma}(\C), \mathcal{L}^{\vee}_{\underline{s}}) $$  
which are both of rank $n$. The period matrix, with respect to suitable bases, is exactly the $(n \times n)$ matrix  \eqref{IntroLSigma}. Its entries can be promoted to 
equivalence classes 
$$\left(\LP_{\Sigma}^{\mm}\right)_{ij} =  \Big[ M_{\Sigma} \ ,  \  \delta_i \otimes x^{s_0} \prod_{k=1}^n (1-x \sigma_k^{-1})^{s_k} \ ,  \  - s_j \,  d\log (x- \sigma_j) \Big]^{\mm}$$
of Betti--de Rham matrix coefficients (a.k.a. motivic periods), where $\delta_i$ is a path from $0$ to $\sigma_i$, and  $M_{\Sigma} $ is an object of a Tannakian category encoding the data of the Betti and de Rham cohomology together with the integration pairing.  They map to \eqref{IntroLSigma} under the period homomorphism:
 $$  \mathrm{per}  ( L^{\mm}_{\Sigma} )    = L_{\Sigma} \ .$$
We also define  de Rham versions of    $\left(\LP_{\Sigma}^{\mm}\right)_{ij}$ as equivalence classes  of  matrix coefficients as follows. 
Consider the (classes of the) logarithmic $1$-forms:
$$\nu_i = \frac{dx}{x-\sigma_i} - \frac{dx}{x}  \qquad \hbox{ for } 1\leq i \leq n\ ,    $$
with residues at $0$ and $\sigma_i$ only (see Remark \ref{rem: c zero} for an interpretation of these forms as de Rham `versions' of the paths $\delta_i$).
They define de Rham cohomology classes in the space $H^1_{\mathrm{dR}}(X_{\Sigma}, \nabla_{-\underline{s}})$, which is isomorphic, via  the de Rham intersection pairing, to the dual $H^1_{\mathrm{dR}}(X_{\Sigma}, \nabla_{\underline{s}})^{\vee}$.  Consider the de Rham--de Rham matrix coefficients (a.k.a. de Rham periods) 
$$\left(\LP_{\Sigma}^{\dR}\right)_{ij} =  \left[ M_{\Sigma} \ ,  \     \nu_i ,  \  - s_j \,  d\log (x- \sigma_j) \right]^{\dR}\ .$$
Comultiplication of matrix coefficients immediately implies a global coaction  formula  which takes the very simple matrix form:
\begin{equation} \label{introCoact}  \Delta\LP^{\mm}_{\Sigma} = \LP^{\mm}_{\Sigma} \otimes \LP^{\dR}_{\Sigma} \ .
\end{equation}
 The period map cannot be applied to the de Rham Lauricella functions, but one can replace them with variants $\widetilde{\LP}^{\dR}_{\Sigma}$ by passing to a slightly different Tannakian category (where a coaction formula similar to \eqref{introCoact} holds) which takes into account the real Frobenius, i.e., complex conjugation. This added feature yields a `single-valued period' homomorphism $\s$ which can be applied to de Rham periods, whenever all the $s_i$ are real\footnote{This is an unnatural technical assumption that is forced upon us by the fact that the parameters $s_i$ are not treated as formal variables (see \S \ref{par:variant formal s}), but formulae such as \eqref{introDC} and \eqref{introDCbis} remain valid for complex $s_i$.}. 
  We recover in this way the single-valued Lauricella hypergeometric functions \eqref{IntroLSigmasv}:
 $$ \s ( \widetilde{\LP}^{\dR}_{\Sigma}) =\LP^{\s}_{\Sigma}\ .$$
 This implies, by the definition of the single-valued period homomorphism, the identity
\begin{equation} 
\label{introDC}   L^{\s}_{\Sigma}(\underline{s})   =  L_{\overline{\Sigma}}(-\underline{s})^{-1}  L_{\Sigma}(\underline{s}) \end{equation}
where $\overline{\Sigma}=\{\overline{\sigma_0},\overline{\sigma_1},\ldots,\overline{\sigma_n}\}$, and $L_{\Sigma}(\underline{s})$ denotes the matrix \eqref{IntroLSigma} with  dependence on $\underline{s}=(s_0,\ldots,s_n)$ made explicit. After applying twisted period relations \cite{ChoMatsumoto}, \eqref{introDC} can be rewritten as a quadratic formula for single-valued Lauricella functions, which we call a \emph{double copy formula}:
\begin{equation}\label{introDCbis}
L^\s_\Sigma(\underline{s}) = \frac{1}{2\pi i}\, {}^tL_{\overline{\Sigma}}(\underline{s})\, I^{\mathrm{B}}_{\overline{\Sigma}}(-\underline{s})\, L_\Sigma(\underline{s})\ ,
\end{equation}
where $I^{\mathrm{B}}_{\overline{\Sigma}}(-\underline{s})$ is a `Betti intersection matrix' computed in terms of intersection numbers on $X_\Sigma(\mathbb{C})$.\medskip

In \S\ref{sect: Renorm}  we make the transition from the global to the local picture   and explain how to renormalise the integrals  \eqref{IntroLSigma}, following a similar procedure to \cite{BD2}, to  expose their poles in the $s_0,\ldots, s_n$ in a neighbourhood of the origin.   These poles are compensated by the prefactors $s_j$ in the definition \eqref{IntroLSigma}, yielding   Taylor  expansions  for both the functions $(L_{\Sigma})_{ij}$  (Proposition \ref{proprenorm}) and  their single-valued versions $(L^{\s}_{\Sigma})_{ij}$ (Proposition \ref{proprenormsv}). 

In the remaining, `local', part of the paper, we compute the periods, single-valued periods and motivic coaction order-by-order in the Taylor expansion with respect to the $s_i$.  For this, we  assume that the $\sigma_i $ lie in a number field $k\subset \C$ and work in the Tannakian category $\MT(k)$ of mixed Tate motives over $k$ \cite{delignegoncharov}. It has a canonical fiber functor $\varpi$.  Alternatively, one can avoid fixing the  $\sigma_i$ by working in the category of mixed Tate motives over the space of configurations $\Sigma$, which leads to identical formulae. 
The motivic torsor of paths
$\pi_1^{\mathrm{mot}}(X_{\Sigma}, t_0, -t_{i})$
where $t_0$ is the tangent vector $1$ at $0$, and $t_{i}$ is the tangent vector $\sigma_i$ at $\sigma_i$, is dual to an ind-object of $\MT(k)$ whose periods are regularised iterated integrals of logarithmic $1$-forms. Since it is a torsor over the motivic fundamental group based at $t_0$, one can define its metabelian (or double-commutator) quotient.  
It turns out that the periods of the latter are very closely related to generalised beta-integrals of the form \eqref{IntroLSigma}. To make this more precise, for any formal power series
$ F   \in R\langle\langle e_0,\ldots, e_n \rangle \rangle $
 in non-commuting variables $e_0,\ldots, e_n$ with coefficients in a commutative ring $R$,
consider its \emph{abelianisation} and  $j^{th}$ \emph{beta-quotient}
$$\overline{F} \quad \hbox{ and } \quad \overline{F_j}  \quad \in \quad  R[[s_0,\ldots, s_n]]$$
where $\overline{S}$ denotes 
 the  image  of a formal power series $S$ under the abelianisation  map  $e_k \mapsto s_k$, where the $s_k$ are commuting variables, and the $F_j$ are the unique power series such that
 $$F = F_{\varnothing} + F_0 e_0 + \cdots + F_n e_n \ ,$$
 where $F_{\varnothing} \in R$ is simply the constant term of $F$. 
The motivic torsor of paths from $t_0$ to $-t_{i}$ defines formal power series $\ZZ^{\mm, i}$, $\ZZ^{\varpi, i} $  in the $e_0,\ldots, e_n$
whose coefficients are motivic periods and (canonical) de Rham periods of $\MT(k)$ respectively\footnote{The reader should be warned that the superscripts $\mm$, $\varpi$, $\dR$, corresponding to different types of matrix coefficients in Tannakian categories, have two different interpretations in this paper: one in the `global' setting, and one is the `local' setting. This should not create any ambiguity.}, and whose abelianisations and beta-quotients are of interest. 
We assemble them into a matrix of formal power series in $s_0,\ldots,s_n$:
$$(\FL^{\mm}_{\Sigma})_{ij} =  \mathbf{1}_{i=j} \overline{\ZZ^{\mm, i}} -s_j  \overline{\ZZ_j^{\mm, i}}$$
for $1\leq i, j \leq n$,  (respectively   with $\mm$ replaced with $\varpi$). The following' result (Theorem \ref{thm: FL equals L} and Theorem \ref{thm: FL equals L sv}) states that these are indeed `local motivic lifts' of the expanded Lauricella functions (viewed as  power series in the $s_k$). This justifies the notation $FL_\Sigma$ where the letter $F$ stands for `formal'.

\begin{thm}  \label{introthmperandsv} 
We have the equalities of formal power series in $s_0,\ldots,s_n$:
\begin{eqnarray}
  (i) \quad  \mathrm{per} \, (\FL^{\mm}_{\Sigma}) & = &   \LP_{\Sigma}   \nonumber  \\
  (ii)  \qquad   \s \, (\FL^{\varpi}_{\Sigma})  &= &   \LP^{\s}_{\Sigma}  \ ,   \nonumber 
    \end{eqnarray} 
  where  $\per$ and  $\s$  are the period map and single-valued period map of $\MT(k)$ applied order by order to the coefficients in the expansion in the variables $s_k$.  
\end{thm}
In the special case  $\Sigma=\{0,1\}$, part $(i)$ of the theorem reduces to Drinfeld's well-known computation of the metabelian quotient of the Drinfeld associator in terms of Euler's beta function. 

Having therefore established that $FL^\mm_\Sigma$ is a local motivic lift of the Lauricella functions viewed as power series, we then prove that its local motivic coaction is given by the same formula as the global motivic coaction \eqref{introCoact} of the global motivic lift $L^\mm_\Sigma$.

\begin{thm}\label{thm: intro thm coaction local}  The (local) motivic coaction on the formal power series  $\FL^{\mm}_{\Sigma}$ reads: 
\begin{equation}\label{introDeltaFL}   \Delta \FL^{\mm}_{\Sigma}(s_0,\ldots, s_n)= \FL^{\mm}_{\Sigma}( s_0,\ldots,  s_n) \otimes \FL^{\varpi}_{\Sigma} ( s_0(\Lef^{\varpi})^{-1} ,\ldots, s_n(\Lef^{\varpi})^{-1}) \ ,
\end{equation}
 where $\Lef^{\varpi}$ is the (canonical) de Rham Lefschetz motivic period (motivic version of $2\pi i$). In this formula, the variables $s_k$ should be viewed as having weight $-2$ and as such admit a non-trivial motivic coaction: $\Delta(s_k)=s_k\, (1\otimes (\Lef^{\varpi})^{-1})$. \end{thm}

In the last two sections we apply our results to Gauss' hypergeometric function $F={}_2F_1$. Via the integral formula \eqref{eq: intro 2F1 integral} we define global and local lifts of the function $F(a,b,c;y)$. In the global setting, $a,b,c$ are generic complex numbers, i.e.,  $a,b,c,c-a,c-b\notin\mathbb{Z}$,  whereas in the local setting they are formal variables around $a=b=c=0$. The global motivic coaction formula reads:
$$\Delta F^\mm(a,b,c;y) = F^\mm(a,b,c;y) \otimes F^\dR(a,b,c;y) - \frac{y}{1+c}\, F^\mm(a+1,b+1,c+2;y)\otimes G^\dR(a,b,c;y)\ ,$$
and the local formula (Theorem \ref{thm: coaction local 2F1}) is formally similar, with $\dR$ replaced with $\varpi$ and extra $(\Lef^\varpi)^{-1}$ factors inserted as in Theorem \ref{thm: intro thm coaction local}. The term $G^\dR$ appearing in the formula is a de Rham version of the function
$$G(a,b,c;y) = \frac{\sin(\pi a)\sin(\pi (c-a))}{\pi \sin(\pi c)} \,\beta(b,c-b)^{-1}\int_\infty^{y^{-1}}x^{b-1}(1-x)^{c-b-1}(1-yx)^{-a}\,dx\ ,$$
which equals $y^{1-c}\,F(1+b-c,1+a-c,2-c;y)$ up to a prefactor. We also study single-valued versions of the functions $F$ and $G$, both in the global and local settings, and prove double copy formulae relating them to $F$ and $G$ (see Proposition \ref{prop: doublecopyfor2F1}).

\subsection{Example}  \label{sect: introBeta}
Let $n=1$,  $\Sigma = \{0,1\}$,   $k=\Q$ and $X_{\Sigma} = \Pro^1\backslash \{0,1,\infty\}$. The canonical fiber functor $\varpi$ on $\MT(\mathbb{Q})$ is simply the de Rham fiber functor. 
\subsubsection{Cohomology with coefficients}
Let $s_0, s_1\in\mathbb{C}$ be generic, i.e.,  $\{s_0,s_1, s_0+s_1\} \cap \Z = \varnothing$. The algebraic de Rham cohomology $H^1_{\mathrm{dR}}(X, \nabla_{\underline{s}})$  has rank one over $\Q(s_0,s_1)$ and is spanned by the class of  $s_1 \frac{dx}{1-x}$. The locally finite homology $H^{\lf}_1(X(\C), \mathcal{L}_{\underline{s}}^\vee)$  also has rank one over $\Q(e^{2\pi is_0},e^{2\pi is_1})$, and is spanned by the class of  $(0,1)\otimes x^{s_0} (1-x)^{s_1}$.  The corresponding period matrix is the $(1 \times 1)$ matrix
$$ \LP_{\{0,1\}} =  \left( s_1\int_0^1  x^{s_0}(1-x)^{s_1} \frac{dx}{1-x}\right) = \left(  \frac{s_0s_1}{s_0+s_1} \beta(s_0,s_1) \right) \ .$$
Note that $\LP_{\{0,1\}}$ is \emph{a priori} only defined for generic $s_0,s_1$. It turns out \emph{a posteriori} that it admits a Taylor expansion at the point $(s_0,s_1)=(0,0)$, which is not generic. The lifted (global) period matrix  $\LP^{\mm}_{\{0,1\}}$ also has a single entry and 
satisfies the (global) coaction formula
\begin{equation} \label{introdeltaLbeta}  \Delta  \LP_{\{0,1\}}^{\mm} = \LP_{\{0,1\}}^{\mm} \otimes \LP_{\{0,1\}}^{\dR}   \ . \end{equation}
This is immediate from the fact that the matrix has rank one. 
\subsubsection{Formal series expansion} Consider the Drinfeld associator
$$\ZZ= \sum_{w\in \{e_0,e_1\}^{\times}} \zeta(w) \, w  =  1+ \zeta(2) (e_0e_1-e_1e_0) + \ldots $$
where $\zeta(w)$ are shuffle regularised multiple zeta values. Its  abelianisation   satisfies $\overline{\ZZ}=1$. The  $(1\times 1)$ matrix of formal expansions of Lauricella functions $\FL_{\{0,1\}}$ is  therefore $$\FL_{\{0,1\}}= \Big( 1 - s_1 \overline{\ZZ_1} \Big) \ .$$
Its entry is  the formal power series
$$
 1 - s_1\int_{\dch}  x^{s_0} (1-x)^{s_1} \frac{dx}{x-1}  = \frac{s_0s_1}{s_0+s_1} \beta(s_0,s_1)$$
where $\dch$ is the straight  line path between tangential base points at  $0$ and $1$, and the second equality follows from 
 Proposition   \ref{lem: ConvergentZi}. It is well-known that
 \begin{equation} \label{introbetaexpand} \frac{s_0s_1}{s_0+s_1} \beta(s_0,s_1) =  \exp \left( \sum_{n\geq 2}  \frac{(-1)^{n-1} \zeta(n)}{n} \left( (s_0+s_1)^n  - s_0^n -s_1^n\right)\right)\ .
 \end{equation} 
The above objects have motivic and de Rham  versions $\ZZ^{\mm} , \FL^{\mm}_{\{0,1\}}$ (respectively $\mathcal{Z}^\dR$, $FL^\dR_{\{0,1\}}$),  formally denoted by adding superscripts in the appropriate places,  which are formal power series in $s_0,s_1$ whose coefficients are motivic (respectively de Rham) periods of $\MT(\mathbb{Q})$\footnote{They are actually motivic and de Rham periods of the Tannakian subcategory $\MT(\mathbb{Z})$.}.
For example,  the entry of the matrix  $\FL^{\mm}_{\{0,1\}}$ is  exactly  the right-hand side of  \eqref{introbetaexpand}, in which $\zeta(n)$ is replaced by $\zetam(n)$. The (local) motivic coaction satisfies\footnote{If one writes this in terms of the motivic beta function $\beta^{\mm}(s_0,s_1) $ defined by  $(s_0^{-1} + s_1^{-1})$ times the entry of $ \FL^{\mm}_{\{0,1\}}$, then it takes the form 
 $$ \Delta \beta^{\mm}(s_0,s_1)   =  \frac{s_0s_1}{s_0+s_1}  \beta^{\mm} (s_0, s_1) \otimes \beta^{\dR}((\Lef^\dR)^{-1}s_0,(\Lef^\dR)^{-1}s_1) \ .$$}
 \begin{equation}\label{eq: coaction FL beta intro}
 \Delta \, \FL^{\mm}_{\{0,1\}}(s_0,s_1) = \FL^{\mm}_{\{0,1\}}(s_0,s_1) \otimes \FL_{\{0,1\}}^{\dR}((\Lef^\dR)^{-1}s_0,(\Lef^\dR)^{-1}s_1) \ , 
 \end{equation}
 where $\Delta$ acts term by term and on the formal variables $s_k$ via $\Delta(s_k)=s_k(1\otimes (\Lef^\dR)^{-1})$. The formula \eqref{eq: coaction FL beta intro} is equivalent to the equations:
 \begin{equation} \label{intro: Deltazetam} \Delta \, \zetam(n) = \zetam(n) \otimes (\Lef^{\dR})^n + 1 \otimes \zeta^{\dR}(n)\end{equation} 
 for all $n\geq 2$, using a variant of the well-known fact that in a complete Hopf algebra, 
  an element is group-like if and only if it is the exponential of a primitive element. Since  $\zeta^{\dR}(2n)=0$ for  $n\geq 1$, we retrieve the known coaction formulae on motivic zeta values \cite{brownICM}.

 \begin{rem} 
 Equation \eqref{intro: Deltazetam} is equivalent to the fact that  zeta values are periods of simple extensions in the category $\MT(\Q)$.  It is very interesting that  this non-trivial statement  shows up as the  apparently simpler  fact \eqref{introdeltaLbeta} that the cohomology group underlying $L_{\{0,1\}}$ has rank one. It also explains why the formula is uniform for all values of $n$.
\end{rem} 

\subsubsection{Single-valued versions} The single-valued beta integral is 
\begin{equation}\label{eq: intro sv beta}
- \frac{ s_1}{2\pi i} \iint_{\C}   |z|^{2s_0} |1-z|^{2s_1}  \left(\frac{d\overline{z}}{\overline{z}-1}  -  \frac{d\overline{z}}{\overline{z}} \right) \wedge \frac{dz}{1-z}  \   = \      \frac{s_0s_1}{s_0+s_1} \beta^{\,\s}(s_0,s_1)  \ , 
\end{equation}
where the single-valued (or `complex') beta function $\beta^{\,\s}(s_0,s_1)$ satisfies the  formula 
$$  \frac{s_0s_1}{s_0+s_1} \beta^{\,\s}(s_0,s_1)   =   - \frac{\beta(s_0,s_1)}{\beta(-s_0,-s_1)} \ .$$
The expansion of \eqref{eq: intro sv beta} can be expressed in the form 
$$
1- s_1\, \s( \overline{\ZZ^{\dR}_1}) =  \exp \left( \sum_{n\geq 2}  \frac{(-1)^{n-1} \zeta^{\mathbf{sv}}(n)}{n} \left( (s_0+s_1)^n  - s_0^n -s_1^n\right)\right)\ .
$$
where the single-valued zeta $\zeta^{\mathbf{sv}}(n)$ equals $2\, \zeta(n)$ for $n$ odd $\geq 3$ and vanishes for even $n$. \medskip

This discussion of the beta function generalises to the case of  the moduli spaces of curves of genus zero with marked points. It has  been studied in \cite{BD2, vanhovezerbini, schlottererschnetz}.

\subsection{Comments} 

\subsubsection{Comparing local expansions} 
It is natural to ask what can be said about Taylor expansions of Lauricella functions around different (rational) values of the parameters $s_k$. Clearly, the expansions around different values of $s_k$ are independent from each other: trying to compare them quickly leads to identities involving infinite sums of the kind
$$\sum_{n=2}^{\infty} \left(\zeta(n)-1 \right)=1\ , $$
for which there is no motivic interpretation.
In fact, the coefficients of  expansions  at different rational points are very different from the motivic point of view. For example, the value of $\beta(s_0,s_1)$ at non-integer rational values of $s_0,s_1$ is not a mixed Tate period in general. However, we believe that there should be a general tannakian framework which, when correctly interpreted, controls the motivic Galois theory of all the expansions of Lauricella functions (and related hypergeometric-type integrals) around rational values of the parameters. This is beyond the scope of this article.

\subsubsection{Divergences}
The main technical point in this paper is, as usual,  dealing with divergences. For cohomology with coefficients, this appears as non-genericity of  the parameters $s_k$. For motivic fundamental groups, it takes the form of tangential base points. The following key example illustrates the point:

\begin{example} Suppose that $\mathrm{Re}(s)>0$. Then, viewed as a function of $s$, 
$$I(s)= \int_{0}^1  x^s \,\frac{dx}{x} = \frac{1}{s}\ \cdot  $$
The renormalised version $I^{\ren}(s)$  of this integral (defined in \S\ref{sect: Renorm})  removes the pole in $s$, hence  $I^{\ren}(s)=0$. 
Now consider the integral as a formal power series in $s$. We perform a Taylor expansion of the integrand  and integrate term by term. Since the integrals diverge, they are regularised with respect to a tangent vector of length $1$ at the origin, which is equivalent to integrating along the straight line path $\dch$ (for `droit chemin'):
$$I^{\mathrm{local}}(s)  =  \sum_{n\geq 0}  \frac{s^n}{n!} \int_{\dch}  \log^n(x) \, \frac{dx}{x}   = 0  \ .$$
Thus $I^{\mathrm{local}}(s)$ is indeed  the Taylor expansion of $I^{\ren}(s)$, which is not true if one were to regularise with respect to a different tangential base point. In general, our tangential basepoints are chosen to be consistent with the renormalisation of divergent integrals. 
\end{example}

\subsubsection{Higher dimensional generalisations}
There are precursors in the physics literature  to coaction formulae on generating series of motivic periods.  Indeed, in  \cite{schlottererstiebergermotivic}  open string amplitudes in genus $0$ (which can be computed in terms of associators \cite{BroedelSchlottererStiebergerTerasoma}) were recast in terms of series of motivic multiple zeta values, and some conjectures were formulated about their $f$-alphabet decomposition to all orders.  For four particles, this  is equivalent to Example \ref{sect: introBeta}.

We can make these conjectures precise as a simple application of the framework described here. Let $\mathcal{M}_{0,S}$ denote the moduli space of curves with points marked by a set $S$ with $n+3$ elements, and  let $\nabla_{\underline{s}}$, $\mathcal{L}_{\underline{s}}$ be the Koba--Nielsen connection and local system considered in \cite[\S 6]{BD2}. The periods and single-valued periods  studied in that paper can be easily formalised using the Tannakian  categories defined in this paper: in brief, the triple
$$ M^S \ = \ \left( H_{\mathrm{B}}^n(\mathcal{M}_{0,S}, \mathcal{L}_{\underline{s}}) \ , \   H_{\mathrm{dR}}^n(\mathcal{M}_{0,S}, \nabla_{\underline{s}}) \ , \  \comp_{\mathrm{B}, \mathrm{dR}} \right) $$
defines an object of any of the `global' Tannakian categories considered in \S\ref{sec: tannakianglobal}.  From this one can define `global' Tannakian lifts of the closed and open superstring amplitudes in genus $0$. `Local' lifts (after expanding in the variables $s_{ij}$) were worked out in \cite{BD2}.   We can define global matrix coefficients 
$$I^{\mm}(\omega) = [  M^S,  \gamma, \omega ]^{\mm} \qquad \hbox{ and } \qquad I^{\dR}(\nu, \omega) = [  M^S,  \nu , \omega ]^{\dR} $$
for suitable Betti homology classes $\gamma$ and de Rham (resp. dual de Rham) classes $\omega$ (resp.  $\nu$). The period of $I^{\mm}(\omega)$ is an open string amplitude, and the single-valued period of (a slight variant of) $I^{\dR}(\nu, \omega)$ is a closed string amplitude. The general coaction formalism \eqref{Deltaformula}  or \cite{brownnotesmot} yields 
$$ \Delta I^{\mm}(\omega) = \sum_{\eta}  I^{\mm}(\eta) \otimes  I^{\dR}(\eta^{\vee}, \omega)$$ 
where $\eta$ ranges over a basis of  $H_{\mathrm{dR}}^n(\mathcal{M}_{0,S}, \nabla_{\underline{s}})$ and $\eta^{\vee}$ is the dual basis. 
The objects $I^{\mm}(\omega)$ can be viewed as versions of open string amplitudes, the objects   $I^{\dR}(\eta^{\vee}, \omega)$ as  versions of closed string amplitudes. 
In terms of the de Rham intersection pairing, this can equivalently  be written 
$$ \Delta I^{\mm}(\omega) = \sum_{\eta ,\eta'}  \langle \eta, \eta' \rangle^{\mathrm{dR}}  \, I^{\mm}(\eta) \otimes  I^{\dR}(\eta', \omega)$$ 
where $\eta, \eta'$ range over bases of   $H_{\mathrm{dR}}^n(\mathcal{M}_{0,S}, \nabla_{\underline{s}})$ and $H_{\mathrm{dR}}^n(\mathcal{M}_{0,S}, \nabla_{-\underline{s}})$ respectively.  We proved in \cite{BD2} that the Laurent expansions of open and closed string amplitudes admit (non-canonical)  motivic lifts. In the light of the present paper, it is natural to expect that their coactions are compatible with the global formula written above. It would be interesting to see if this is equivalent to the conjectures of Stieberger and Schlotterer mentioned above. 

\subsubsection{Generalisations to other settings}
There should be interesting possible  generalisations of our results to the elliptic \cite{Matthes} and $\ell$-adic \cite{Nakamura, UniversalJacobi} settings. \medskip

\subsubsection*{Convention}
Throughout this paper we use the following convention: we denote the coordinate on the affine line $\mathbb{C}$ by $x$ when dealing with line integrals, and by $z$ when dealing with double (single-valued) integrals. 
\vspace{0.1in}

\subsubsection*{Acknowledgements} This paper was begun during a visit of the  first named author to Trinity College Dublin as a Simons visiting Professor.  He is grateful to that institution for hospitality and also to Ruth Britto and Ricardo Gonzo for conversations during that time. Many thanks  to Samuel Abreu, Claude Duhr, and Einan Gardi for discussions on their  work, to Lance Dixon and Paul Fendley for references to  conformal field theory, and to Matija Tapu\v{s}kovi\'{c} for comments on the first version of this paper. 
We thank the anonymous referees for many valuable comments and suggestions.
This project has received funding from the European Research Council (ERC) under the European Union's Horizon 2020 research and innovation programme (grant agreement No. 724638).
The second named author was partially supported by ANR grant PERGAMO (ANR-18-CE40-0017).

\section{Cohomology with coefficients of a punctured Riemann sphere}\label{sec: cohomology coeff}
We  first recall the interpretation of  the integrals \eqref{IntroLSigma} as periods of  the cohomology of the punctured Riemann sphere  with coefficients in a rank one algebraic vector bundle or local system. Most of the results of this section can be found in the classical literature \cite{hattorikimura, aomotovanishing, kitanoumi, kitayoshida, kitayoshida2, ChoMatsumoto}. However, we provide an original treatment of the single-valued period homomorphism for cohomology with coefficients. It is based on the Frobenius at infinity, as in the general setting of \cite{BD1, BD2} and not on the complex conjugation of coefficients, as in  \cite{hanamurayoshida, mimachiyoshida}.

\subsection{Periods of cohomology with coefficients}
 Let  $k\subset \C$ be a field and let $\Sigma=\{\sigma_0,\ldots, \sigma_n\}$ be  distinct points  in $k$ with $\sigma_0=0$. Write $$X_{\Sigma}= \A_k^1 \backslash \Sigma \ .$$
We consider a tuple $\underline{s}=(s_0,\ldots,s_n)$ of complex numbers that we shall often assume to be \emph{generic}, meaning that we have:
\begin{equation}\label{generic} 
\{s_0, s_1, \ldots, s_n,  s_0+s_1+\cdots+s_n\}  \cap  \Z = \varnothing \ .
\end{equation}
An alternative point of view, that we will not develop here, would be to treat the $s_i$ as formal variables (see \S \ref{par:variant formal s}). 

\subsubsection{Algebraic de Rham cohomology} For any subfield $k\subset \C$, denote by 
$$k^{\mathrm{dR}}_{\underline{s}} = k(s_0,\ldots, s_n)\ .$$
The algebraic de Rham cohomology groups that we will consider are $k^{\mathrm{dR}}_{\underline{s}}$-vector spaces with a natural $\Q^{\mathrm{dR}}_{\underline{s}}$-structure. 
Define the following  logarithmic 1-forms on $\Pro^1_k$  
\begin{equation} \label{omegajdef}   \omega_i= \frac{dx}{x- \sigma_i} \qquad \hbox{ for } \quad  i = 0,\ldots, n \ ,
\end{equation} 
 which have residue $0$ or $1$ at points of $\Sigma$, and $-1$ at $\infty$. They form a basis of the space of global logarithmic forms $\Gamma(\Pro^1_k, \Omega^1_{\Pro^1_k} ( \log \Sigma \cup \{\infty\}))$, which maps isomorphically to $H_{\mathrm{dR}}^1(X_{\Sigma}/k) $.

\begin{defn} Let $\Or_{X_\Sigma}$ denote the trivial rank one bundle on $X_{\Sigma} \times_{k} k_{\underline{s}}^{\mathrm{dR}}$, and consider the following logarithmic connection upon it
$$\nabla_{\underline{s}} : \Or_{X_\Sigma} \To \Omega^1_{X_{\Sigma}}
\qquad \hbox{ given by } \qquad 
\nabla_{\underline{s}} = d + \sum_{i=0}^n  s_i\, \omega_i \ .$$
\end{defn} 
It is automatically integrable since $X_{\Sigma}$ has dimension one, and is in fact closely related to the abelianisation of the canonical connection on the de Rham unipotent fundamental group of $X_\Sigma$. Consider the algebraic de Rham cohomology groups 
$$H^r_{\mathrm{dR}}(X_{\Sigma}, \nabla_{\underline{s}}) = H^r_{\mathrm{dR}}(X_{\Sigma}, (\Or_{X_{\Sigma}}, \nabla_{\underline{s}})) $$
which are finite-dimensional $k_{\underline{s}}^{\mathrm{dR}}$-vector spaces.
The fact that the $s_i$ are generic implies, by \cite[Proposition II.3.13]{deligneEqDiff}, that one has a logarithmic comparison theorem for $(\Or_{X_\Sigma},\nabla_{\underline{s}})$. Since the cohomology of $X_\Sigma$ is spanned by global logarithmic forms, this implies (see \cite{esnaultschechtmanviehweg}) that the cohomology groups $H^r_{\mathrm{dR}}(X_\Sigma,\nabla_{\underline{s}})$ are computed by the complex of global logarithmic forms
$$0\To k_{\underline{s}}^{\mathrm{dR}}\stackrel{\nabla_{\underline{s}}}{\To} k_{\underline{s}}^{\mathrm{dR}}\,\omega_0\oplus \cdots \oplus k_{\underline{s}}^{\mathrm{dR}}\,\omega_n \To 0\ ,$$
where  $\nabla_{\underline{s}}(1)=\sum_{i=0}^ns_i\omega_i$. Again by genericity of the $s_i$,  $H^r_{\mathrm{dR}}(X_{\Sigma}, \nabla_{\underline{s}})$ vanishes for  $r\neq 1$, and $H^1_{\mathrm{dR}}(X_{\Sigma}, \nabla_{\underline{s}})$ has dimension $n$. It is generated by the $\omega_i$ 
subject to  the single relation 
\begin{equation}\sum_{i=0}^n s_i \, \omega_i = 0\ .
\end{equation}
 Since the forms $\omega_i$ have rational residues, they in fact define a natural $\Q^{\mathrm{dR}}_{\underline{s}}$-structure on $H^1_{\mathrm{dR}}(X_{\Sigma}, \nabla_{\underline{s}})$ which we shall denote by $H^1_{\varpi}(X_{\Sigma}, \nabla_{\underline{s}})$. We therefore have
\begin{equation}  \label{H1dRwithcoeff} H^1_{\varpi}(X_{\Sigma}, \nabla_{\underline{s}}) \cong \left(\bigoplus_{i=0}^n \Q^{\mathrm{dR}}_{\underline{s}}  \omega_i\right) / \ \, \Q^{\mathrm{dR}}_{\underline{s}}\,\sum_{i=0}^n s_i\omega_i\ .
\end{equation}
We shall use the following  basis  for \eqref{H1dRwithcoeff}:
\begin{equation} \label{dRcoeffBasis} 
\{  - s_i \omega_i, \hbox{ for }  i= 1,\ldots, n\}\ .
\end{equation}

\subsubsection{Betti (co)homology} 

We introduce the subfield of $\C$ defined by
$$\Q_{\underline{s}}^{\mathrm{B}}=\Q(e^{2\pi is_0},\ldots,e^{2\pi is_n})\ .$$
\begin{defn} Let $\mathcal{L}_{\underline{s}}$ denote the rank one local system of $\Q^{\mathrm{B}}_{\underline{s}}$-vector spaces on the complex points $X_{\Sigma}(\C)= \C \backslash \Sigma$ defined by 
$$\mathcal{L}_{\underline{s}} = \Q^{\mathrm{B}}_{\underline{s}}\; x^{-s_0} \prod_{k=1}^n (1-x \sigma_k^{-1})^{-s_k} \ . $$
\end{defn} 
The local system $\mathcal{L}_{\underline{s}}$ has monodromy $e^{-2\pi is_k}$ around the point $\sigma_k$. After extending scalars to $\C$, it is identified with the horizontal sections of the (analytified) connexion $\nabla_{\underline{s}}$ on the trivial vector bundle of rank one on $X_\Sigma(\C)$:
$$\mathcal{L}_{\underline{s}}\otimes_{\Q_{\underline{s}}^{\mathrm{B}}} \C \cong \left( \Or^{\mathrm{an}}_{X_\Sigma}\right)^{\nabla_{\underline{s}}}\ .$$
We will be interested in its cohomology 
$$H_{\mathrm{B}}^r(X_{\Sigma}, \mathcal{L}_{\underline{s}})= H^r(\C\backslash \Sigma, \mathcal{L}_{\underline{s}})   \  \cong  \ H_r(\C \backslash \Sigma, \mathcal{L}_{\underline{s}}^\vee)^{\vee}\ ,$$
where $\mathcal{L}_{\underline{s}}^\vee$ is the dual local system
$$\mathcal{L}_{\underline{s}}^\vee=\Q^{\mathrm{B}}_{\underline{s}}\; x^{s_0}\prod_{k=1}^n (1-x\sigma_k^{-1})^{s_k}\ .$$
The genericity assumption \eqref{generic} implies that $\mathcal{L}_{\underline{s}}$ and $\mathcal{L}_{\underline{s}}^\vee$ have non-trivial monodromy around every point of $\Sigma$ and around $\infty$, which implies that the natural map 
$$ H_r( \C \backslash \Sigma, \mathcal{L}_{\underline{s}}^\vee) \To H^{\lf}_r( \C \backslash \Sigma, \mathcal{L}_{\underline{s}}^\vee)  $$
from ordinary homology to locally finite homology  is an isomorphism.  Its inverse is sometimes called regularisation. 
 An easy computation shows that all  homology  is concentrated in degree one and $H_1^{\lf}(\C \backslash \Sigma, \mathcal{L}_{\underline{s}}^\vee)$ has rank $n$. 
 It has a basis consisting of the (classes of the) locally finite chains
 \begin{equation} \label{BcoeffBasis} 
  \delta_i \otimes  x^{s_0} \prod_{k=1}^n (1-x \sigma_k^{-1})^{s_k} 
\end{equation} 
for $i=1,\ldots,n$, where $\delta_i: (0,1) \rightarrow \C \backslash \Sigma$ is a smooth  
path from $0$ to $\sigma_i$ that can be extended to a smooth path $\overline{\delta_i}:[0,1]\to \C$, and $x^{s_0}\prod_{k=1}^n(1-x\sigma_k^{-1})^{s_k}$ denotes some choice of section of $\mathcal{L}_{\underline{s}}^\vee$ on $\delta_i$.

\begin{rem}\label{rem: choice delta i winding}
In all that follows, we will assume that the choices of representatives are such that $\delta_i$ does not wind infinitely around $0$ and $\sigma_i$, i.e., that the argument of $\delta_i(t)$ is bounded as $t$ approaches $0$ and that the argument of $\delta_i(t)-\sigma_i$ is bounded as $t$ approaches $1$. This assumption will ensure the convergence of the integrals considered below.
\end{rem}

\subsubsection{Comparison isomorphism.} 
 There is a canonical isomorphism \cite[\S6]{deligneEqDiff}
 \begin{equation}\label{eq:compBdRcoefficients}
 \comp_{\mathrm{B}, \mathrm{dR}}(\underline{s}):  H^1_{\mathrm{dR}}(X_{\Sigma}, \nabla_{\underline{s}}) \otimes_{k^{\mathrm{dR}}_{\underline{s}}} \C \overset{\sim}{\To}   H^1_{\mathrm{B}}(X_{\Sigma}, \mathcal{L}_{\underline{s}}) \otimes_{\Q^{\mathrm{B}}_{\underline{s}}} \C\ .
 \end{equation}
 whose restriction to the $\Q_{\underline{s}}^{\mathrm{dR}}$-structure we shall denote by
 \begin{equation}\label{eq:compBvarpi}
 \comp_{\mathrm{B}, \varpi}(\underline{s}):  H^1_{\varpi}(X_{\Sigma}, \nabla_{\underline{s}}) \otimes_{\Q^{\mathrm{dR}}_{\underline{s}}} \C \overset{\sim}{\To}   H^1_{\mathrm{B}}(X_{\Sigma}, \mathcal{L}_{\underline{s}}) \otimes_{\Q^{\mathrm{B}}_{\underline{s}}} \C\ .
 \end{equation}
 Assuming  \eqref{generic}, we can identify
 Betti cohomology  $H_{\mathrm{B}}^1(X_{\Sigma}, \mathcal{L}_{\underline{s}})$ with the dual of locally finite homology, which leads to  a bilinear pairing
 $$H_1^{\lf}(\C\backslash {\Sigma}, \mathcal{L}_{\underline{s}}^\vee) \times  H^1_{\varpi}(X_{\Sigma}, \nabla_{\underline{s}})   \To \C\ .  $$
 It is well-known that the comparison isomorphism is computed by integration when it makes sense.
 
 \begin{lem}\label{lem:comparison equals Lauricella coefficients}   Assuming  \eqref{generic}, a matrix representative for the comparison isomorphism in the bases \eqref{dRcoeffBasis}, \eqref{BcoeffBasis}   is  the $(n\times n)$  period  matrix $\LP_{\Sigma}$ with entries
$$(\LP_{\Sigma})_{ij} = -s_j   \int_{\delta_i}   x^{s_0} \prod_{k=1}^n (1-x \sigma_k^{-1})^{s_k}  \frac{dx}{x-\sigma_j}\ ,$$
provided that $\mathrm{Re}\, s_0>-1$ and 
$\mathrm{Re}\, s_i>
\begin{cases} -1 & \hbox{ if }\; i\neq j\ ; \\
0 & \hbox{ if }\;  i=j\ .
\end{cases}
$
\end{lem}

\begin{proof}
By definition, the pairing that we wish to compute is 
$$\langle\, \delta_i \otimes  x^{s_0} \prod_{k=1}^n (1-x \sigma_k^{-1})^{s_k}  , \, \mathrm{comp}_{\mathrm{B},\mathrm{dR}}(\underline{s})(-s_j\omega_j)\,\rangle = -s_j \int_{\delta_i}x^{s_0}\prod_{k=1}^n(1-x\sigma_k^{-1})^{s_k}\widetilde{\omega}_j\ ,$$
where $\widetilde{\omega}_j$ is a smooth form on $\C\setminus\Sigma$ with compact support, representing the cohomology class of $\omega_j=dz/(z-\sigma_j)$.  In other words, we have
$$\widetilde{\omega}_j-\omega_j = \nabla_{\underline{s}}\phi = d\phi + \sum_{k=0}^ns_k\, \phi \, d\log(z-\sigma_k)\ ,$$
where $\phi$ is a smooth function on $\mathbb{P}^1(\C)$. Since $\widetilde{\omega}_j$ vanishes in the neighbourhood of $\Sigma\cup\infty$, taking residues in the previous equation along points of $\Sigma\cup\infty$ implies that $\phi$ vanishes at every $\sigma_k$, $k\neq j$, including at $\sigma_0=0$. Now we only need to prove that the integral
$$\int_{\delta_i}x^{s_0} \prod_{k=1}^n (1-x \sigma_k^{-1})^{s_k} \,\nabla_{\underline{s}}\phi $$
vanishes. Since $x^{s_0}\prod_{k=1}^n (1-x \sigma_k^{-1})^{s_k} \,\nabla_{\underline{s}}\phi = d(x^{s_0}\prod_{k=1}^n (1-x \sigma_k^{-1})^{s_k} \,\phi)$, this integral is computed by Stokes' theorem and we need to prove that $x^{s_0}\prod_{k=1}^n(1-x\sigma_k^{-1})^{s_k}\phi$ vanishes at $0$ and at $\sigma_i$.
\begin{enumerate}[--]
\item At $0$ this amounts to proving that $\delta(t)^{s_0}\phi(\delta(t))$ goes to $0$ when $t\to 0$. Since $\phi$ is smooth and vanishes at $0$ we have 
$ |\phi(\delta(t)) | < C |\delta(t)| $ for some constant $C$, and so
$$|\delta(t)^{s_0}\phi(\delta(t))| < C |\delta(t)^{s_0+1}|=C |\delta(t)|^{\mathrm{Re}(s_0)+1}\exp(-\mathrm{Im}(s_0)\mathrm{arg}(\delta(t)))\ .$$ 
By assumption, $\mathrm{Re}(s_0)+1>0$ and $\mathrm{arg}(\delta(t))$ is bounded as $t$ approaches zero (see Remark \ref{rem: choice delta i winding}), which implies that the limit of $\delta(t)^{s_0+1}$, and also    $\delta(t)^{s_0}\phi(\delta(t))$,   is zero when $t\rightarrow 0$.
\item At $\sigma_i$, the same argument gives the desired vanishing, by using the fact that $\phi$ vanishes at $\sigma_i$ if $i\neq j$.
\end{enumerate}
The result follows.
\end{proof}

\subsection{Intersection pairings} For   generic   $s_i$ \eqref{generic},  the natural map 
    $H_1(\C\backslash \Sigma, \mathcal{L}_{\underline{s}}^\vee)\rightarrow H_1^\lf(\C\backslash \Sigma, \mathcal{L}_{\underline{s}}^\vee)$
    is an isomorphism.  Poincar\'{e} duality gives an isomorphism
    $$H_1^\lf(\C\backslash \Sigma, \mathcal{L}_{\underline{s}}^\vee) \simeq  H_1(\C\backslash \Sigma, \mathcal{L}_{\underline{s}})^\vee \simeq H_1(\C\backslash \Sigma, \mathcal{L}_{-\underline{s}}^\vee)^\vee\ ,$$
    where we set $-\underline{s}=(-s_0,\ldots, -s_n)$. By combining these two isomorphisms we get a perfect pairing, called the \emph{Betti intersection pairing} \cite[\S 2]{kitayoshida}, \cite{ChoMatsumoto}, \cite[\S 2]{mimachiyoshida}:
    $$\langle  \ ,  \, \rangle_{\mathrm{B}} : H_1(\C\backslash \Sigma, \mathcal{L}_{-\underline{s}}^\vee) \otimes_{\Q_{\underline{s}}^{\mathrm{B}}} H_1(\C\backslash \Sigma, \mathcal{L}_{\underline{s}}^\vee) \To \Q_{\underline{s}}^{\mathrm{B}}\ ,$$
    or dually in cohomology:
    \begin{equation}\label{eq:bettipairing}
    \langle  \ ,  \, \rangle^{\mathrm{B}} : H^1_{\mathrm{B}}(X_\Sigma, \mathcal{L}_{-\underline{s}}) \otimes_{\Q_{\underline{s}}^{\mathrm{B}}} H^1_{\mathrm{B}}(X_\Sigma, \mathcal{L}_{\underline{s}}) \To \Q_{\underline{s}}^{\mathrm{B}}\ .
    \end{equation}
    The de Rham counterpart is the \emph{de Rham intersection pairing} \cite{ChoMatsumoto, KMatsumoto}:
    \begin{equation}\label{eq:derhampairing}
    \langle  \ ,  \, \rangle^{\mathrm{dR}} : H^1_{\mathrm{dR}}(X_\Sigma,\nabla_{-\underline{s}}) \otimes_{k_{\underline{s}}^{\mathrm{dR}}} H^1_{\mathrm{dR}}(X_\Sigma,\nabla_{\underline{s}}) \To k_{\underline{s}}^{\mathrm{dR}}\ ,
    \end{equation}
    which comes from Poincar\'{e} duality and the fact that the natural map
    \begin{equation}\label{eq:canonicalderhampairing}
    H^1_{\mathrm{dR},\mathrm{c}}(X_\Sigma,\nabla_{\underline{s}}) \To H^1_{\mathrm{dR}}(X_\Sigma,\nabla_{\underline{s}})
    \end{equation}
    is an isomorphism if the $s_i$ are generic, where the subscript $\mathrm{c}$ denotes compactly supported cohomology. The map \eqref{eq:derhampairing} respects the natural $\Q_{\underline{s}}^{\mathrm{dR}}$-structures and induces
    $$\langle\ , \, \rangle^{\varpi}:H^1_{\varpi}(X_\Sigma,\nabla_{-\underline{s}}) \otimes_{\Q_{\underline{s}}^{\mathrm{dR}}} H^1_{\varpi}(X_\Sigma,\nabla_{\underline{s}}) \To \Q_{\underline{s}}^{\mathrm{dR}}\ .$$

    We set  
    \begin{equation}\label{eq:def nui}
    \nu_i = \frac{dx}{x-\sigma_i} - \frac{dx}{x} \qquad \hbox{ for } \quad  1\leq i \leq n
    \end{equation}
    and use the same notation for the corresponding class in $H^1_{\varpi}(X_\Sigma,\nabla_{-\underline{s}})$.
    
    \begin{lem}\label{lem: omega nu dR pairing}  For all $1\leq i, j \leq n$,  and $s_0,\ldots, s_n$ satisfying \eqref{generic},
    \begin{equation}
    \label{deRhampairingForms} \left\langle \nu_i   , \omega_j \right\rangle^{\mathrm{dR}} = -\frac{1}{s_i} \, \mathbf{1}_{i=j} \ .
    \end{equation} 
    \end{lem} 
    \begin{proof} 
    By definition, the de Rham intersection pairing that we wish to compute is
    $$\langle \nu_i,\omega_j \rangle^{\mathrm{dR}} = \frac{1}{2\pi i}\iint_{\Pro^1(\C)}\widetilde{\nu}_i\wedge\omega_j\ ,$$
    where $\widetilde{\nu}_i$ is a smooth form on $\mathbb{C}\setminus \Sigma$ with compact support, representing 
    the cohomology class of $\nu_i$. In other words we have 
    $$\widetilde{\nu}_i-\nu_i =\nabla_{-\underline{s}}\phi = d\phi -\sum_{k=0}^ns_k\,\phi\, d\log(x-\sigma_k)\ ,$$
    where $\phi$ is a smooth function on $\Pro^1(\C)$. Since $\widetilde{\nu}_i$ vanishes in the neighbourhood of $\Sigma\cup\infty$, taking residues along points of   $\Sigma\cup\infty$  implies that $\phi$ vanishes at every $\sigma_k$, $k\notin\{ 0,i\}$, and at $\infty$, and
    $$\phi(0) = -\frac{1}{s_0} \; , \; \phi(\sigma_i)=\frac{1}{s_i} \ . $$
    By noticing that $\nu_i\wedge \omega_j=0$ and  $d\log(x-\sigma_k)\wedge \omega_j=0$, we thus get
    $$\langle \nu_i,\omega_j\rangle^{\mathrm{dR}} = \frac{1}{2\pi i}\iint_{\Pro^1(\C)} d\phi\wedge \omega_j = \frac{1}{2\pi i}\iint_{\Pro^1(\C)} d(\phi\,\omega_j)\ .$$
    By Stokes, this last integral can be computed as the limit when $\varepsilon$ goes to zero of
    $$-\frac{1}{2\pi i}\int_{\partial P_\varepsilon} \phi\,\omega_j$$
    where $P_\varepsilon$ is the complement in $\Pro^1(\C)$ of $\varepsilon$-disks around the points of $\Sigma\cup\infty$, and the sign comes from the orientation of $\partial P_\varepsilon$. By using the fact that $\phi(\infty)=0$ and $\omega_j$ is regular at every $\sigma_k$, $k\neq j$, a local computation (variant of Cauchy's formula) thus gives
    $$\langle \nu_i , \omega_j\rangle^{\mathrm{dR}} = -\mathrm{Res}_{\sigma_j}(\phi\,\omega_j) =-\frac{1}{s_i}\mathbf{1}_{i=j}\ .$$
    \end{proof} 
    
    This lemma implies that the dual basis to \eqref{dRcoeffBasis} is given by the classes
    $$ \nu_i  \;\;\in \;\; H^1_{\varpi}(X_\Sigma,\nabla_{-\underline{s}}) \qquad 
    \hbox{ for } i=1,\ldots,n \ .$$

    \begin{rem}\label{rem: twisted period relations}
    The Betti and de Rham intersection pairings are compatible with the comparison isomorphism, which leads to quadratic relations among periods of cohomology with coefficients, known in the literature as \emph{twisted period relations} \cite{ChoMatsumoto, gotolauricella}. In matrix form they read:
    \begin{equation}\label{eq: twisted period relation}
    {}^tL_\Sigma(-\underline{s})\, I^{\mathrm{B}}_\Sigma(\underline{s})\,L_\Sigma(\underline{s}) = 2\pi i\,I^{\mathrm{dR}}_\Sigma(\underline{s})
    \end{equation}
    where $I^{\mathrm{B}}_\Sigma(\underline{s})$ and $I^{\mathrm{dR}}_\Sigma(\underline{s})$ are the matrices of the Betti and de Rham intersection pairings \eqref{eq:bettipairing} and \eqref{eq:derhampairing}, respectively.
    \end{rem}

\subsection{Single-valued periods of cohomology with coefficients}
We can define and compute a period pairing on de Rham cohomology classes by transporting complex conjugation.

    \subsubsection{Definition of the single-valued period map}\label{subsubsec: def sv}
 Let 
 $\overline{\Sigma}=\{\overline{\sigma_0},\overline{\sigma_1},\ldots,\overline{\sigma_n}\}$ denote the complex conjugates of the points in $\Sigma$. 
 We have an anti-holomorphic diffeomorphism 
 $$\mathrm{conj}:\C\backslash \overline{\Sigma} \To \C\backslash \Sigma$$ given by complex conjugation. We note that the induced map $H_1(\C\backslash\overline{\Sigma})\rightarrow H_1(\C\backslash \Sigma)$ sends the class of a positively oriented loop around $\overline{\sigma_j}$ to the class of a negatively oriented loop around $\sigma_j$. Since a rank one local system on $\C\backslash\Sigma$ (resp. $\C\backslash\overline{\Sigma}$) is equivalent to a representation of the abelian group $H_1(\C\backslash\Sigma)$ (resp. $H_1(\C\backslash\overline{\Sigma})$) we see that we have an isomorphism of local systems: \begin{equation}\label{eq: conj frob iso}
 \mathrm{conj}^*\mathcal{L}_{\underline{s}}\simeq \mathcal{L}_{-\underline{s}}\ .
 \end{equation}
 We thus get a morphism of local systems on $\C\backslash\Sigma$:  
 $$\mathcal{L}_{\underline{s}} \To \mathrm{conj}_*\mathrm{conj}^*\mathcal{L}_{\underline{s}} \simeq \mathrm{conj}_*\mathcal{L}_{-\underline{s}}\ ,$$
 which at the level of cohomology induces a morphism of $\Q_{\underline{s}}^{\mathrm{B}}$-vector spaces
 $$F_\infty: H^1_{\mathrm{B}}(\C\backslash \Sigma, \mathcal{L}_{\underline{s}}) \To H^1_{\mathrm{B}}(\C\backslash \overline{\Sigma},\mathcal{L}_{-\underline{s}})\ .$$
 We call $F_\infty$ the \emph{real Frobenius} or \emph{Frobenius at the infinite prime}. We will use the notation $F_\infty(\underline{s})$ when we want to make  the dependence on $\underline{s}$ explicit. One checks that the Frobenius is involutive: $F_{\infty}(-\underline{s}) F_{\infty}(\underline{s}) =\id$.

 \begin{rem}\label{rem: compute frobenius betti}
 The isomorphism \eqref{eq: conj frob iso} is induced by the trivialisation of the tensor product $\mathrm{conj}^*\mathcal{L}_{\underline{s}}\otimes\mathcal{L}_{\underline{s}}$ given by the section
 $$g_{\underline{s}} = |z|^{-2s_0}\prod_{k=1}^n|1-z\sigma_k^{-1}|^{-2s_k}\ .$$
 Thus, the homological real Frobenius 
 $$F_\infty:H_1^{\mathrm{B}}(\mathbb{C}\setminus \overline{\Sigma},\mathcal{L}_{\underline{s}}) \longrightarrow H_1^{\mathrm{B}}(\mathbb{C}\setminus \Sigma,\mathcal{L}_{-\underline{s}})$$
 is computed at the level of representatives by the formula
 $$\delta\otimes z^{-s_0}\prod_{k=1}^n(1-z\overline{\sigma_k}^{-1})^{-s_k} \mapsto \overline{\delta}\otimes \overline{z}^{-s_0}\prod_{k=1}^n(1-\overline{z}\,\overline{\sigma_k}^{-1})^{-s_k} \, g_{\underline{s}}^{-1} = \overline{\delta}\otimes z^{s_0}\prod_{k=1}^n(1-z\sigma_k^{-1})^{s_k}\ .$$
 \end{rem}
 
 \begin{rem}\label{re: sv is not conj of coeff}
 A morphism similar to $F_\infty$ was considered in \cite{hanamurayoshida} and leads to similar formulae but has a different definition and a different interpretation. Our definition only uses the action of complex conjugation  on the complex points of the variety $X_\Sigma$ relative to two complex conjugate embeddings of $k$ in $\mathbb{C}$, whereas the definition in [\emph{loc. cit.}] conjugates the field of coefficients of the local systems, which requires the $s_i$ to be real. Note that our definition does not require the assumption that the $s_i \in \mathbb{R}$.
 \end{rem}

In the rest of this article, however, we will often assume that the $s_i$ are real. (This is an unnatural assumption and would not be necessary if the $s_i$ were treated as formal variables, see \S\ref{par:variant formal s}.) In this way, the complex conjugate of the field $k_{\underline{s}}^{\mathrm{dR}}$ inside $\C$ is the field $\overline{k}(s_1,\ldots,s_n)$. We use the notation $(-)\otimes_{k_{\underline{s}}^{\mathrm{dR}}}\overline{\C}$ for the tensor product with $\C$, viewed as a $k_{\underline{s}}^{\mathrm{dR}}$-vector space via the complex conjugate embedding. We thus have an additional $\C$-linear comparison isomorphism:
$$\mathrm{comp}_{\overline{\mathrm{B}},\mathrm{dR}}(\underline{s}):H^1_{\mathrm{dR}}(X_\Sigma,\nabla_{\underline{s}})\otimes_{k_{\underline{s}}^{\mathrm{dR}}}\overline{\C} \To H^1_{\mathrm{B}}(\C\setminus \overline{\Sigma},\mathcal{L}_{\underline{s}})\otimes_{\Q_{\underline{s}}^{\mathrm{B}}}\C\ .$$

\begin{defn}
Assume that the $s_i$ are real. The \emph{single-valued period map} is the $\C$-linear isomorphism 
$$\s : H^1_{\mathrm{dR}}(X_\Sigma,\nabla_{\underline{s}})\otimes_{k^{\mathrm{dR}}_{\underline{s}}}\C\To H^1_{\mathrm{dR}}(X_\Sigma,\nabla_{-\underline{s}})\otimes_{k^{\mathrm{dR}}_{\underline{s}}}\overline{\C}$$
defined as the composite 
$$\s = \mathrm{comp}_{\overline{\mathrm{B}},\mathrm{dR}}^{-1}(-\underline{s}) \circ (F_\infty\otimes \mathrm{id}) \circ \mathrm{comp}_{\mathrm{B},\mathrm{dR}}(\underline{s}) \ .$$
In other words, it is defined by the commutative diagram
$$\xymatrixcolsep{5pc}\xymatrix{
 H^1_{\mathrm{dR}}(X_\Sigma,\nabla_{\underline{s}})\otimes_{k^{\mathrm{dR}}_{\underline{s}}}\C \ar[d]_{\s} \ar[r]^-{\mathrm{comp}_{\mathrm{B},\mathrm{dR}}(\underline{s})} & H^1_{\mathrm{B}}(\C\setminus \Sigma,\mathcal{L}_{\underline{s}})\otimes_{\Q_{\underline{s}}^{\mathrm{B}}}\C \ar[d]^{F_\infty\otimes \mathrm{id}}\\
H^1_{\mathrm{dR}}(X_\Sigma,\nabla_{-\underline{s}})\otimes_{k^{\mathrm{dR}}_{\underline{s}}}\overline{\C} \ar[r]_{\mathrm{comp}_{\overline{\mathrm{B}},\mathrm{dR}}(-\underline{s})}& H^1_{\mathrm{B}}(\C\setminus \overline{\Sigma},\mathcal{L}_{-\underline{s}})\otimes_{\Q_{\underline{s}}^{\mathrm{B}}}\C
}$$
We will use the notation $\s(\underline{s})$ when we want to make the dependence on $\underline{s}$ explicit.
\end{defn}

The single-valued period map is a transcendental comparison isomorphism that is naturally interpreted at the level of analytic de Rham cohomology via the isomorphisms
$$H^1_{\mathrm{dR}}(X_\Sigma,\nabla_{\underline{s}})\otimes_{k^{\mathrm{dR}}_{\underline{s}}}\C \simeq H^1_{\mathrm{dR, an}}(\C\backslash\Sigma, (\mathcal{O}_{\C\backslash \Sigma},\nabla_{\underline{s}}))$$
and
$$H^1_{\mathrm{dR}}(X_\Sigma,\nabla_{-\underline{s}})\otimes_{k^{\mathrm{dR}}_{\underline{s}}}\overline{\C} \simeq H^1_{\mathrm{dR, an}}(\C\backslash\overline{\Sigma}, (\mathcal{O}_{\C\backslash \overline{\Sigma}},\nabla_{-\underline{s}}))\ .$$
To avoid any confusion we use the coordinate $w=\overline{z}$ on $\C\backslash\overline{\Sigma}$.

\begin{lem}\label{lem:singlevalued analytic}
Assume that the $s_i$ are real. In analytic de Rham cohomology, the single-valued period map is induced by the morphism of smooth de Rham complexes
\begin{eqnarray*}
    \s_{\mathrm{an}} : (\mathcal{A}^\bullet_{\C\backslash \Sigma},\nabla_{\underline{s}}) &\To &
    \mathrm{conj}_*(\mathcal{A}^\bullet_{\C\backslash \overline{\Sigma}}, \nabla_{-\underline{s}})
\end{eqnarray*}
given on the level of sections by 
$$\mathcal{A}^\bullet_{\C\backslash\Sigma}(U)\,\ni\, \omega \quad  \mapsto \quad  |w|^{2s_0}\prod_{k=1}^n|1-w\overline{\sigma}_k^{-1}|^{2s_k}\; \mathrm{conj}^*(\omega)\,\in\, \mathcal{A}^\bullet_{\C\backslash\overline{\Sigma}}(\overline{U})\ .$$

\end{lem}

\begin{proof}
Let  $P=|w|^{2s_0}\prod_{k=1}^n|1-w\overline{\sigma}_k^{-1}|^{2s_k}$. We first check that $\s_{\mathrm{an}}$ is a morphism of complexes:
\begin{eqnarray*}
\nabla_{-\underline{s}}(\s_{\mathrm{an}}(\omega)) & = & \nabla_{-\underline{s}}(P\,\mathrm{conj}^*(\omega)) \\
& = & P\left(\left(\sum_{k=0}^ns_k\,d\log(w-\overline{\sigma_k})  + \sum_{k=0}^ns_k\,d\log(\overline{w}-\sigma_k)\right)\wedge\mathrm{conj}^*(\omega) + d(\mathrm{conj}^*(\omega))\right) \\
& & \qquad -\sum_{k=0}^ns_k\, d\log(w-\overline{\sigma_k})\wedge (P \,\mathrm{conj}^*(\omega)) \\
& = & P\left(\sum_{k=0}^n s_k\,d\log(\overline{w}-\sigma_k)\wedge\mathrm{conj}^*(\omega)+d(\mathrm{conj}^*(\omega))\right) \\
& = & P\,\mathrm{conj}^*\left(\sum_{k=0}^n s_k\, d\log(z-\sigma_k)\wedge\omega+ d\omega \right)  \\
& = & \s_{\mathrm{an}}(\nabla_{\underline{s}}(\omega))\ .
\end{eqnarray*}
On the level of horizontal sections, we compute:
\begin{eqnarray*}
\s_{\mathrm{an}}\left(z^{-s_0}\prod_{k=1}^n(1-z\sigma_k^{-1})^{-s_k}\right)  & = & |w|^{2s_0}\prod_{k=1}^n|1-w\overline{\sigma}_k^{-1}|^{2s_k}\, \overline{w}^{-s_0}\prod_{k=1}^n(1-\overline{w}\sigma_k^{-1})^{-s_k} \\
&= & w^{s_0}\prod_{k=1}^n(1-w\overline{\sigma_k}^{-1})^{s_k}\ .
\end{eqnarray*}
Thus, $\s_{\mathrm{an}}$ induces the morphism $\mathcal{L}_{\underline{s}}\rightarrow \mathrm{conj}_*\mathcal{L}_{-\underline{s}}$ and the result follows.
\end{proof}

\subsubsection{Integral formula for single-valued periods} We  derive a formula for the single-valued period map $\s$ using the de Rham intersection pairing \eqref{eq:derhampairing}, that is, for the \emph{single-valued period pairing}
\begin{equation}\label{eq: sv pairing}
H^1_{\mathrm{dR}}(X_\Sigma,\nabla_{\underline{s}}) \times H^1_{\mathrm{dR}}(X_\Sigma,\nabla_{\underline{s}}) \longrightarrow \mathbb{C} \;\;\; ,\;\;\;  (\nu ,\omega)\,\mapsto\, \langle \nu, \s\,\omega\rangle^{\mathrm{dR}} \ .
\end{equation}
Note that this pairing is $k_{\underline{s}}^{\mathrm{dR}}$-linear in each argument, where in the first slot the complex conjugate embedding of $k_{\underline{s}}^{\mathrm{dR}}$ inside $\mathbb{C}$ is understood: this means that we have, for $a,b\in k_{\underline{s}}^{\mathrm{dR}}$:
$$\langle a\nu , \s\, b\omega\rangle^{\mathrm{dR}} = \overline{a}b\,\langle \nu,\s\,\omega\rangle^{\mathrm{dR}}\ .$$

\begin{prop}  \label{propSVpairingCoeffs} 
Assume that $s_0,\ldots, s_n$ are real and generic \eqref{generic}. Let $\omega, \nu \in \Gamma(\Pro^1_k, \Omega^1_{\Pro^1_k} (\log \Sigma  \cup \infty))$ and write $\omega$ and $\nu$ for their classes in $H^1_{\mathrm{dR}} (X_{\Sigma}, \nabla_{\underline{s}})$. Assume that $\nu$ has no pole at $\infty$.  Then the single-valued pairing is
\begin{equation} \label{svGlobal}   \langle \nu , \, \s\, \omega \rangle^{\mathrm{dR}} =  - \frac{1}{2\pi i} \iint_{\C} |z|^{2s_0} \prod_{k=1}^n |1- z \sigma_k^{-1} |^{2s_k}   \, \overline{ \nu} \wedge  \omega  
\end{equation}
whenever   $s_i>0$ for all $0\leq i \leq n$, and $s_0+\cdots +s_n < 1/2$. 
\end{prop}
\begin{proof} 
Let us first note that the integral in \eqref{svGlobal} converges under the assumptions on $\omega$, $\nu$, and the $s_i$. To check this, pass to local polar coordinates $z=\rho e^{i\theta}$ in the neighbourhood of every point in $\Sigma$, and verify that $|z|^{2s}\frac{dz\,d\overline{z}}{z\overline{z}}$ is proportional to $\rho^{2s-1}d\rho\,d\theta$ and is integrable when $\mathrm{Re}(s)>0$. At $\infty$ use the fact that $\nu$ has no pole to obtain a local estimate of the form $\rho^{-2s_0-\cdots-2s_n}d\rho\,d\theta$, which is integrable for $s_0+\cdots +s_n < 1/2$. 

We use the notation $\omega=\omega_\sigma$ for the smooth form on $\C\setminus\Sigma$ induced by $\omega$, and $\omega_{\overline{\sigma}}$ for the smooth form on $\C\setminus\overline{\Sigma}$ induced by $\omega$. (It is obtained by replacing each occurrence of $\sigma_j$ in $\omega$ by a $\overline{\sigma_j}$.) Then by definition we have $\langle \nu,\s\,\omega\rangle = \langle \nu_{\overline{\sigma}},\s_{\mathrm{an}}\omega_\sigma\rangle $. By Lemma \ref{lem:singlevalued analytic} and the definition of the de Rham intersection pairing this equals the integral
\begin{equation}\label{eq:integralsvproofglobal}
\frac{1}{2\pi i}\iint_{\Pro^1(\C)}|w|^{2s_0}\prod_{k=1}^n|1-w\overline{\sigma_k}^{-1}|^{2s_k}\, \widetilde{\nu_{\overline{\sigma}}}\wedge \mathrm{conj}^*(\omega_{\sigma})\ ,
\end{equation}
where $\widetilde{\nu_{\overline{\sigma}}}$ is a smooth form on $\C\setminus \overline{\Sigma}$ with compact support,  
representing the cohomology class of $\nu_{\overline{\sigma}}$. In other words we have 
$$\widetilde{\nu_{\overline{\sigma}}} - \nu_{\overline{\sigma}} = \nabla_{\underline{s}}\phi = d\phi + \sum_{k=0}^ns_k\,\phi\, d\log(w-\overline{\sigma_k})\ ,$$
where $\phi$ is a smooth function on $\Pro^1(\C)$. The assumption that $\nu$ has no pole at $\infty$ and the fact that $s_0+\cdots+s_n\neq 0$ imply, by taking residues, that $\phi(\infty)=0$. We first prove that we may remove the tilde in \eqref{eq:integralsvproofglobal}, i.e., that the integral
\begin{equation}\label{eq:integralmustvanish}
\iint_{\Pro^1(\C)}|w|^{2s_0}\prod_{k=1}^n|1-w\overline{\sigma_k}^{-1}|^{2s_k}\, \nabla_{\underline{s}}\phi\wedge \mathrm{conj}^*(\omega_{\sigma})
\end{equation}
vanishes. Its integrand equals 
$$d\left(|w|^{2s_0}\prod_{k=1}^n|1-w\overline{\sigma_k}^{-1}|^{2s_k}\phi\; \mathrm{conj}^*(\omega)\right)$$
because $d\overline{w}\wedge\mathrm{conj}^*(\omega_\sigma)=\mathrm{conj}^*(dz\wedge\omega_\sigma)=0$. By Stokes, the integral \eqref{eq:integralmustvanish} can be computed as the limit when $\varepsilon$ goes to zero of
$$-\int_{\partial P_\varepsilon} |w|^{2s_0}\prod_{k=1}^n|1-w\overline{\sigma_k}^{-1}|^{2s_k}\phi\;\mathrm{conj}^*(\omega)$$
where $P_\varepsilon$ is the complement in $\Pro^1(\C)$ of $\varepsilon$-disks around the points of $\overline{\Sigma}\cup\infty$, and the sign comes from the orientation of $\partial P_\varepsilon$. The contribution of each point of $\overline{\Sigma}$ vanishes, as can be seen from a computation in a local coordinate, because of the assumption that $s_i>0$ for all $i$. The contribution of the point $\infty$ also vanishes because of the fact that $\phi(\infty)=0$ and the assumption that $s_0+\cdots +s_n < 1/2$. Thus, we have 
\begin{eqnarray*}
\langle \nu , \, \s \,\omega \rangle^{\mathrm{dR}} & = & \frac{1}{2\pi i}\iint_{\Pro^1(\C)}|w|^{2s_0}\prod_{k=1}^n|1-w\overline{\sigma_k}^{-1}|^{2s_k}\, \nu_{\overline{\sigma}}\wedge \mathrm{conj}^*(\omega_{\sigma}) \\
& = & - \frac{1}{2\pi i}\iint_{\Pro^1(\C)}|\overline{z}|^{2s_0}\prod_{k=1}^n|1-\overline{z}\overline{\sigma_k}^{-1}|^{2s_k}\, \mathrm{conj}^*(\nu_{\overline{\sigma}})\wedge \omega_{\sigma} \\
& = & - \frac{1}{2\pi i}\iint_{\Pro^1(\C)}|z|^{2s_0} \prod_{k=1}^n |1- z \sigma_k^{-1} |^{2s_k}   \, \overline{ \nu_\sigma} \wedge  \omega_\sigma \ .
\end{eqnarray*}
The second equality follows from performing a change of variables via $\mathrm{conj}:\C\backslash\Sigma \rightarrow \C\backslash\overline{\Sigma}$, which reverses the orientation of $\C$, hence the minus sign. The third equality relies on $\mathrm{conj}^*(\nu_{\overline{\sigma}}) = \overline{\nu_\sigma}$, which is obvious. The result follows.
\end{proof}

\begin{cor}\label{cor LSigma sv}
Assume that $s_0,\ldots,s_n$ are real and generic \eqref{generic}. Then for all $1\leq i,j\leq n$ the single-valued Lauricella function $(L_\Sigma^\s)_{ij}$ \eqref{IntroLSigmasv} is a single-valued period of cohomology with coefficients:
$$(L_\Sigma^\s)_{ij} = \langle \nu_i, \s(-s_j\omega_j) \rangle^{\mathrm{dR}}$$
for $s_k>0$ for all $0\leq k\leq n$ and $s_0+\cdots+s_n<1/2$.
\end{cor}

\begin{proof}
This follows from Proposition \ref{propSVpairingCoeffs}.
\end{proof}

\begin{cor}
Under the assumptions of the previous corollary we have the matrix equality
\begin{equation}\label{splusminusDC}
L_\Sigma^\s(s_0,\ldots,s_n) = L_{\overline{\Sigma}}(-s_0,\ldots,-s_n)^{-1} L_\Sigma(s_0,\ldots,s_n)\ ,
\end{equation}
where it is understood that the Lauricella functions $(L_\Sigma)_{ij}$ and $(L_{\overline{\Sigma}})_{ij}$ are computed via choices of paths which are complex conjugate to each other.
\end{cor}

\begin{proof}
This follows from the definition of the single-valued period homomorphism and  corollary \ref{cor LSigma sv} since the classes $\nu_i$ are the dual basis, with respect to the de Rham pairing,  of the basis $(-s_i\omega_i)$  by Lemma \ref{lem: omega nu dR pairing}.
\end{proof}

\begin{rem}\label{rem: double copy lauricella global}
The inverse of the matrix $L_{\overline{\Sigma}}(-\underline{s})$ appearing in \eqref{splusminusDC} can be computed in terms of $L_{\overline{\Sigma}}(\underline{s})$ from the twisted period relation \eqref{eq: twisted period relation}. This leads to quadratic expressions of the single-valued Lauricella function in terms of ordinary Lauricella functions, which we call \emph{double copy formulae}. They read, in matrix form:
$$L^\s_\Sigma(\underline{s}) = \frac{1}{2\pi i}\, {}^tL_{\overline{\Sigma}}(\underline{s})\, I^{\mathrm{B}}_{\overline{\Sigma}}(-\underline{s})\, L_\Sigma(\underline{s})\ .$$
This is because the matrix $I^{\mathrm{dR}}_{\overline{\Sigma}}(-\underline{s})$ is the identity matrix in the bases $(+s_i\omega_i)$ and $(\nu_i)$ by Lemma \ref{lem: omega nu dR pairing}. 
They are related to double copy formulae in the physics literature such as the Kawai--Lewellen--Tye formula \cite{klt}, which was interpreted in the framework of cohomology with coefficients in \cite{mizera}.
\end{rem}

\section{Tannakian interpretations and `global' coaction}\label{sec: tannakianglobal}

Using some simple Tannakian formalism, we compute a global coaction formula on Tannakian lifts  of Lauricella functions. We also consider a more refined Tannakian formalism which takes into account the Frobenius at infinity and interpret single-valued Lauricella functions in this more refined  framework.

\subsection{Minimalist version}\label{subsubsec: tannakian T minimalist}  Let $k_{\mathrm{dR}} \subset \C$ and $\Q_{\mathrm{B}} \subset \C$ be two subfields of $\C$. Consider the $\Q$-linear abelian category $\mathcal{T}$ whose objects consist of triples 
$V=(V_{\mathrm{B}}, V_{\mathrm{dR}}, c)$ where $V_{\mathrm{B}}, V_{\mathrm{dR}}$ are finite-dimensional vector spaces over $\Q_{\mathrm{B}} $ and $k_{\mathrm{dR}}$ respectively, and $c: V_{\mathrm{dR}} \otimes_{k_{\mathrm{dR}}} \C \overset{\sim}{\rightarrow} V_{\mathrm{B}} \otimes_{\Q_{\mathrm{B}}} \C$ is a $\C$-linear isomorphism.   
The morphisms $\phi$ in $\mathcal{T}$ are  pairs of linear maps $\phi_{\mathrm{B}}$, $\phi_{\mathrm{dR}}$ compatible with the isomorphisms $c$.  The category $\mathcal{T}$ is Tannakian  with  two fiber functors
$$\omega_{\mathrm{dR}} : \mathcal{T} \To \mathrm{Vec}_{k_{\mathrm{dR}}} \qquad \hbox{ and } \qquad 
\omega_{\mathrm{B}} : \mathcal{T} \To \mathrm{Vec}_{\Q_{\mathrm{B}}}\ .$$
The  ring $\mathcal{P}_{\mathcal{T}}^{\mm} = \Or(\mathrm{Isom}^{\otimes}_{\mathcal{T}} (\omega_{\mathrm{dR}},\omega_{\mathrm{B}}))$  is the $(\Q_{\mathrm{B}}, k_{\mathrm{dR}})$-bimodule  spanned by equivalence classes of matrix coefficients  $[V, \sigma, \omega]^{\mm}$ where $\sigma \in V_{\mathrm{B}}^{\vee}$ and $\omega \in V_{\mathrm{dR}}$. The $k_{\mathrm{dR}}$-algebra  $\mathcal{P}_{\mathcal{T}}^{\dR} = \Or(\mathrm{Aut}^{\otimes}_{\mathcal{T}} (\omega_{\mathrm{dR}}))$ is  spanned by equivalence classes of matrix coefficients  $[V, f, \omega]^{\dR}$ where $f \in V_{\mathrm{dR}}^{\vee}$ and $\omega \in V_{\mathrm{dR}}$.  The multiplicative structure is given by tensor products. There is a natural coaction 
$$\Delta: \Pe^{\mm}_{\mathcal{T}} \To \Pe^{\mm}_{\mathcal{T}} \otimes_{k_{\mathrm{dR}}} \Pe^{\dR}_{\mathcal{T}}$$
which expresses $\Pe^{\mm}_{\mathcal{T}}$ as an algebra comodule over the Hopf algebra $\Pe^{\dR}_{\mathcal{T}}$.
It is given by the formula
\begin{equation} \label{Deltaformula}   \Delta [ V, \sigma, \omega]^{\mm} =  \sum_{i} \, [V, \sigma, e_i]^{\mm} \otimes [V,e_i^{\vee}, \omega]^{\dR}
\end{equation}
where the sum ranges over a $k_{\mathrm{dR}}$-basis $\{e_i\}$ of $V_{\mathrm{dR}}$, and $e_i^{\vee}$ denotes the dual basis of $V_{\mathrm{dR}}^{\vee}$.

\begin{defn} \label{defnofLP} For generic complex numbers $s_i$ \eqref{generic},  let $k_{\mathrm{dR}}=\Q^{\mathrm{dR}}_{\underline{s}}$, $\Q_{\mathrm{B}} =\Q^{\mathrm{B}}_{\underline{s}}$
and define 
$$M_{\Sigma} = \left( \   H^1_{\mathrm{B}}(X_{\Sigma}, \mathcal{L}_{\underline{s}}) \ , \  H^1_{\varpi}(X_{\Sigma}, \nabla_{\underline{s}}) \ , \  \comp_{\mathrm{B},\varpi}(\underline{s}) \ \right) \quad \in\quad  \mathrm{Ob}(\mathcal{T})  \ .$$
Define a matrix $\LP_{\Sigma}^{\mm}  \in   M_{n\times n}( \Pe^{\mm}_{\mathcal{T}})$ by 
$$(\LP_{\Sigma}^{\mm})_{ij} =   \Big[ M_{\Sigma} \  ,  \   \delta_i \otimes  x^{s_0} \prod_{k=1}^n (1-x \sigma_k^{-1})^{s_k}  \ , \   - s_j \omega_j \Big]^\mm\ ,$$
where the basis elements are given by \eqref{dRcoeffBasis} and \eqref{BcoeffBasis}. We will use the notation $M_\Sigma(\underline{s})$ and $L^\mm_{\Sigma}(\underline{s})$ when we want to make the dependence on $\underline{s}$ explicit. 

The image of $\LP_{\Sigma}^{\mm}$ under the period homomorphism 
\begin{eqnarray*} \per : \Pe^{\mm}_{\mathcal{T}}  &\To &  \C  \nonumber \\
\left[(V_{\mathrm{B}}, V_{\mathrm{dR}}, c), \sigma ,\omega\right]^{\mm} & \mapsto& \langle \sigma \, , \, c(\omega)\rangle  
\end{eqnarray*} 
is precisely the matrix of Lauricella functions (when the integral converges, see Lemma \ref{lem:comparison equals Lauricella coefficients}):
$$\mathrm{per} \left(\LP_{\Sigma}^{\mm}\right)_{ij} =       -   s_j   \int_{\delta_i}   x^{s_0} \prod_{k=1}^n (1-x \sigma_k^{-1})^{s_k}  \frac{dx}{x-\sigma_j}\ \cdot$$ 
For this reason we think of $L_{\Sigma}^{\mm}$ as a (global) motivic lift of the matrix of Lauricella functions \eqref{IntroLSigma}. Define a (global) de Rham version 
$$(\LP_{\Sigma}^{\dR})_{ij} =   \left[ M_{\Sigma} \  ,  \ \nu_i    \ , \   - s_j \omega_j \right]^{\dR} \ ,$$
where the forms $\nu_i$, defining classes in $H^1_\varpi(X_\Sigma,\nabla_{-\underline{s}})$, were defined in \eqref{eq:def nui}. Recall that the de Rham pairing \eqref{eq:canonicalderhampairing} induces an isomorphism $H^1_\varpi(X_\Sigma,\nabla_{-\underline{s}})\simeq H^1_\varpi(X_\Sigma,\nabla_{\underline{s}})^\vee$ and that Lemma \ref{lem: omega nu dR pairing} implies that the basis $\{\nu_i\}$ is dual to the basis $\{-s_i\omega_i\}$. 
\end{defn} 

\begin{rem}\label{rem: c zero}
The class of $\nu_i$ is the image of the relative homology class (viewed in homology with trivial coefficients) of the path $\delta_i$ under the map $c_0^\vee$ studied in \cite[\S 4.5.1]{BD1}. It would be interesting to know whether this map can be naturally defined at the level of cohomology with coefficients by relying on Hodge theoretic arguments as in [\emph{loc. cit.}].
\end{rem} 

\begin{example} For all $n\in \Z$, one has  `Tate' objects 
$$\Q_{\underline{s}}(-n) = ( \Q^{\mathrm{B}}_{\underline{s}}, \Q^{\mathrm{dR}}_{\underline{s}}, 1\mapsto (2\pi i)^n)\ . $$
 The Betti and de Rham pairings \eqref{eq:bettipairing} and \eqref{eq:derhampairing}, along with their compatibility with the comparison map \eqref{eq:compBdRcoefficients} (also known as the `twisted period relations'), can be succinctly  encoded as a perfect pairing
 $$M_{\Sigma}(-\underline{s}) \otimes M_{\Sigma}(\underline{s}) \To \Q_{\underline{s}}(-1)$$
 in the category $\mathcal{T}$. Equivalently, $M_{\Sigma}( \underline{s})^{\vee} \simeq M_{\Sigma}(-\underline{s})(1)$, where $(n)$ is the standard notation for Tate twist, i.e., tensor product with $ \Q_{\underline{s}}(n)$.
\end{example} 

\begin{prop} The (global) coaction satisfies:
$$\Delta  \LP_{\Sigma}^{\mm}=  \LP_{\Sigma}^{\mm}\otimes  \LP_{\Sigma}^{\dR}\ .$$
\end{prop} 
\begin{proof}
 This is an immediate consequence of the  formula \eqref{Deltaformula}. 
\end{proof} 

\begin{rem}
The coproduct $\Delta$ in the Hopf algebra $\Pe^\dR_{\mathcal{T}}$ is given by a formula similar to \eqref{Deltaformula} and satisfies: $\Delta  \LP_{\Sigma}^{\dR}=  \LP_{\Sigma}^{\dR}\otimes  \LP_{\Sigma}^{\dR}$.
\end{rem}

\subsection{Version with real Frobenius involutions}\label{subsubsec: tannakian infinity}
In order to incorporate the real Frobenius isomorphism, and hence single-valued periods, into a Tannakian framework, we are obliged to consider a more complicated version of the previous categorical construction. This is somewhat artificial, but is forced upon us by the fact that complex conjugation does not  define an involution  on 
$H^{1}_{\mathrm{B}}(X, \mathcal{L}_{\underline{s}})$, but rather  relates the cohomology of $\mathcal{L}_{\underline{s}}$ with that of $\mathcal{L}_{-\underline{s}}$.  For this reason, we must consider two de Rham, and two Betti realisations corresponding to coefficients  $\nabla_{\underline{s}}$,  $\nabla_{-\underline{s}}$  and  $\mathcal{L}_{\underline{s}}$, $\mathcal{L}_{-\underline{s}}$, which we denote  by superscripts $+/-$.

Let $s_i \in \R$ be real numbers satisfying the genericity conditions \eqref{generic}. Let $k_{\mathrm{dR}} \subset \C$ and $\Q_{\mathrm{B}} \subset \C$ be two subfields of $\C$. We let $(-)\otimes_{k_\mathrm{dR}}\overline{\C}$ denote tensor product taken with respect to the complex conjugate embedding.
One solution is to consider a category $\mathcal{T}_{\infty}$ defined in  a  similar manner as $\mathcal{T}$, except that the objects are given by: 
\begin{enumerate}
    \item A pair of finite-dimensional $k_{\mathrm{dR}}$-vector spaces $V^+_{\mathrm{dR}}$, $V_{\mathrm{dR}}^-$.
    \item Two pairs of finite-dimensional $\Q_{\mathrm{B}}$-vector spaces $V^+_{\mathrm{B}}$, $V_{\mathrm{B}}^-$ and $V^+_{\overline{\mathrm{B}}}$, $V_{\overline{\mathrm{B}}}^-$.
    \item Two $\C$-linear comparison isomorphisms
\begin{eqnarray} c_{\mathrm{B},\mathrm{dR}}^{+}: V_{\mathrm{dR}}^{+} \otimes_{k_{\mathrm{dR}}} \C & \overset{\sim}{\To}& V_{\mathrm{B}}^{+} \otimes_{\Q_{\mathrm{B}}} \C  \nonumber \ ,\\ 
    c_{\mathrm{\overline{B}}, \mathrm{dR}}^{+}: V_{\mathrm{dR}}^{+} \otimes_{k_{\mathrm{dR}}} \overline{\C} &\overset{\sim}{\To}&  V_{\mathrm{\overline{B}}}^{+} \otimes_{\Q_{\mathrm{B}}} \C  \ . \nonumber 
    \end{eqnarray}
    and another two  defined in the same way  with all $+$'s replaced by $-$.
    \item Two $\Q_B$-linear real Frobenius maps 
    $$ V_{B}^{+} \overset{\sim}{\To}  V_{\overline{B}}^{-}  \qquad \hbox{ and } \qquad  V_{\overline{B}}^{+} \overset{\sim}{\To}  V_{B}^{-} \  , $$ 
    and another two with the $+$'s and $-$'s interchanged. These maps will simply be denoted  by $F_{\infty}$, since the source and hence target will be clear from context.  The composition of any two such maps, when defined, is the identity: $F_{\infty} F_{\infty} = \id.$
\end{enumerate}
The morphisms $\phi$ between  objects are given by the data of two $k_{\mathrm{\mathrm{dR}}}$-linear maps $\phi_{\mathrm{dR}}^{+}, \phi_{\mathrm{dR}}^{-}$, and four 
$\Q_{\mathrm{B}}$-linear maps $\phi_{\mathrm{B}}^{+}$, $\phi_{\mathrm{B}}^{-}$,  $\phi_{\overline{\mathrm{B}}}^{+}$, $\phi_{\overline{\mathrm{B}}}^{-}$ which are compatible with $(3)$ and $(4)$. 
One checks that this category is Tannakian, equipped with six fiber functors $\omega_{\mathrm{dR}}^{\pm}, \omega_{\mathrm{B}}^{\pm}, \omega_{\overline{\mathrm{B}}}^{\pm}$ (over $k_{\mathrm{dR}}, \Q_{\mathrm{B}}$ respectively) which are obtained by forgetting all data except one of the vector spaces in $(1)$ or $(2)$. 

\begin{rem}
For objects of $\mathcal{T}_\infty$ coming from geometry such as the ones that we will shortly define, there are compatibilities between the real Frobenius isomorphisms $F_\infty$ and the comparison isomorphisms. We do not include such compatibilities in our definition because they will be irrelevant to the computations that we will be performing.
\end{rem}

The category $\mathcal{T}_\infty$ admits various rings of periods defined in a similar manner as before. Consider the 8 rings of periods:
$$\Pe_{\mathcal{T}_{\infty}}^{\mathrm{B}^{\pm} , {\mathrm{dR}}^{\pm}} = \mathcal{O}(  \mathrm{Isom}^{\otimes}_{\mathcal{T}_{\infty}} ( \omega^{\pm}_{\mathrm{dR}}, \omega^{\pm}_{\mathrm{B}}) )  \qquad \hbox{ and } \qquad  \Pe_{\mathcal{T}_{\infty}}^{\mathrm{\overline{B}^{\pm}} , \mathrm{dR}^{\pm}}  = \mathcal{O}(\mathrm{Isom}^{\otimes}_{\mathcal{T}_{\infty}} ( \omega^{\pm}_{\mathrm{dR}}, \omega^{\pm}_{\mathrm{\overline{B}}}))  \  , $$
corresponding to all possible choices of signs. 
The  four comparison isomorphisms define four period homomorphisms:
$$\per: \Pe_{\mathcal{T}_{\infty}}^{ \mathrm{B}^+, \mathrm{dR}^+} \To \C \qquad \hbox{ and } \qquad \per: \Pe_{\mathcal{T}_{\infty}}^{ \mathrm{\overline{B}}^+, \mathrm{dR}^+} \To \C \   $$
and similarly with $+$ replaced with $-$. 
The four Frobenius maps define isomorphisms: 
$$F_{\infty} \ : \  \Pe_{\mathcal{T}_{\infty}}^{\mathrm{B}^{+}, \bullet}  \ \cong  \ \Pe_{\mathcal{T}_{\infty}}^{\mathrm{\overline{B}}^{-}, \bullet} \quad \hbox{ and } \quad 
 F_{\infty} \ : \    \Pe_{\mathcal{T}_{\infty}}^{\mathrm{B}^{-}, \bullet}  \ \cong  \ \Pe_{\mathcal{T}_{\infty}}^{\mathrm{\overline{B}}^{+}, \bullet}
\ , $$
where $\bullet \in \{\mathrm{dR}^+, \mathrm{dR}^-\}.$ 
By composing with the period homomorphism $\per$, one obtains a full set of 8 period homomorphisms for each of our period rings, e.g., 
$\per\,  F_{\infty}: \Pe_{\mathcal{T}_{\infty}}^{\mathrm{B}^{+}, \mathrm{dR}^-} \rightarrow \C.$

There are also four possible rings of de Rham periods. We shall mainly consider two of them:
$$\Pe_{\mathcal{T}_{\infty}}^{\mathrm{dR}^-, \mathrm{dR}^+} =  \mathcal{O}(\mathrm{Isom}^{\otimes}_{\mathcal{T}_{\infty}} ( \omega^{+}_{\mathrm{dR}}, \omega^{-}_{\mathrm{dR}})) \qquad \hbox{ and } \qquad  \Pe_{\mathcal{T}_{\infty}}^{\mathrm{dR}^+, \mathrm{dR}^-} = \mathcal{O}(\mathrm{Isom}^{\otimes}_{\mathcal{T}_{\infty}} ( \omega^{-}_{\mathrm{dR}}, \omega^{+}_{\mathrm{dR}}))  \ . $$
Duality induces a canonical  isomorphism  $\Pe_{\mathcal{T}_{\infty}}^{\mathrm{dR}^-, \mathrm{dR}^+} \cong \Pe_{\mathcal{T}_{\infty}}^{\mathrm{dR}^+, \mathrm{dR}^-}$ which  we shall not make use of here. 
Both of these rings admit  a single-valued period map, which can be used to detect the non-vanishing of de Rham periods.

\begin{defn} There is a homomorphism `single-valued period'
$$\s^{-,+} : \Pe^{\mathrm{dR}^-, \mathrm{dR}^+}_{\mathcal{T}_{\infty}} \To \C \otimes_{k^{\mathrm{dR}}} \overline{\C}$$
 defined by the composite
$$\xymatrixcolsep{4pc}\xymatrix{V_{\mathrm{dR}}^+ \otimes_{k_{\mathrm{dR}}} \C  \ar[r]^{ c^+_{\mathrm{B}, \mathrm{dR}}} & V_{\mathrm{B}}^+ \otimes_{\Q_{\mathrm{B}}} \C  \ar[r]^{F_{\infty}\otimes\mathrm{id}} & V_{\mathrm{\overline{B}}}^- \otimes_{\Q_{\mathrm{B}}} \C \ar[r]^{( c^-_{\mathrm{\overline{B}}, \mathrm{dR}})^{-1}  } & V_{\mathrm{dR}}^- \otimes_{k_{\mathrm{dR}}} \overline{\C} \ .}$$
Since $\C$ and $\overline{\C}$ have different $k_{\mathrm{dR}}$-structures, this map defines a point on the torsor of isomorphisms from $\omega_{\mathrm{dR}}^+$ to $\omega_{\mathrm{dR}}^-$ only after extending scalars to $\C\otimes_{k_{\mathrm{dR}}} \overline{\C}$.  
The single-valued map therefore sends $[V, f, \omega]^{\dR}$ to  $ \langle f , ( c^-_{\mathrm{\overline{B}}, \mathrm{dR}})^{-1} F_{\infty} c^+_{\mathrm{B}, \mathrm{dR}} \,  \omega \rangle$.

However, in the case  when $k^{\mathrm{dR}} \subset \R$, which is the case that we shall mostly consider, we can  compose with the multiplication homomorphism $\C \otimes_{k^{\mathrm{dR}}} \overline{\C} \rightarrow \C$  to obtain a homomorphism  $$\s^{-,+} : \Pe^{\mathrm{dR}^-, \mathrm{dR}^+}_{\mathcal{T}_{\infty}} \To \C \ .$$
There is a  variant $\s^{+,-}$
 obtained by interchanging all $+$'s and $-$'s in the above.   When it is clear from the context, we drop the superscripts and simply write $\s$.  \end{defn}

\begin{rem}
Alternatively, one can view $\s^{-,+}$ as a morphism of the $(k_{\mathrm{dR}}, k_{\mathrm{dR}})$-bimodule $\Pe^{\mathrm{dR}^-, \mathrm{dR}^+}_{\mathcal{T}_{\infty}} $ to $\C$, with the bimodule structure on the latter given by $(k_{\mathrm{dR}}, \overline{k}_{\mathrm{dR}})$. In other words, for $\lambda_1, \lambda_2 \in k_{\mathrm{dR}}$ one has $ \s^{-, +} ( \lambda_1 \xi \lambda_2) = \lambda_1  \s^{-, +} ( \xi ) \overline{\lambda_2}$.  When $k$ is real, these bimodule structures coincide and we obtain a genuine linear map. 
\end{rem}

The composition of torsors between fiber functors defines several coaction morphisms, including:
 \begin{equation} \label{Coactionplusminus} \Delta : \Pe^{\mathrm{B}^+,\mathrm{dR}^+}_ {\mathcal{T}_{\infty}} \To  \Pe^{\mathrm{B}^+,\mathrm{dR}^-}_ {\mathcal{T}_{\infty}} \otimes_{k_{\mathrm{dR}}}  \Pe^{\mathrm{dR}^-,\mathrm{dR}^+}_ {\mathcal{T}_{\infty}}
\end{equation} 
and likewise with $+, -$ interchanged. The period homomorphisms defined earlier are compatible with composition of torsors between fiber functions. In particular, one has
\begin{equation} \label{perconvolution} \per \left(\xi\right)=  \left( \per\, F_{\infty} \otimes  \s \right)\left( \Delta  \xi \right) \qquad \hbox{ for all } \quad \xi \  \in \  \Pe^{\mathrm{B}^+,\mathrm{dR}^+}_ {\mathcal{T}_{\infty}}\ .
\end{equation}

\begin{defn} For $s_i$ real and generic \eqref{generic},  let $k_{\mathrm{dR}}= \Q^\mathrm{dR}_{\underline{s}}$, $\Q_{\mathrm{B}}= \Q^{\mathrm{B}}_{\underline{s}}$ and  define  an object of rank $n$ in $\mathcal{T}_{\infty}$, denoted by $\widetilde{M}_\Sigma$, whose underlying vector spaces are given by 
$$V_{\mathrm{dR}}^{+}  = H^1_{\varpi}(X_{\Sigma}, \nabla_{+ \underline{s}})    \quad , \quad V_{\mathrm{B}}^{+} =  H^1(\C\backslash \Sigma, \mathcal{L}_{+ \underline{s}})  \quad , \quad V_{\mathrm{\overline{B}}}^{+} =  H^1(\C\backslash \overline{\Sigma}, \mathcal{L}_{+ \underline{s}})$$
and similarly with all $+$'s replaced with $-$'s. 
The Frobenius maps $F_{\infty}$ are induced  on Betti cohomology by complex conjugation $\mathrm{conj}^*$ and its inverse. The comparison isomorphisms $ c^{+}_{\mathrm{B}, \mathrm{dR}} , c^{+}_{\mathrm{\overline{B}}, \mathrm{dR}}$ are defined by 
$$ \mathrm{comp}_{\mathrm{B}, \varpi}(+\underline{s}) :  H^1_{\varpi}(X_{\Sigma}, \nabla_{+ \underline{s}})\otimes_{k_\mathrm{dR}} \C \To 
H^1(\C\backslash \Sigma, \mathcal{L}_{+ \underline{s}})    \otimes_{\Q_{\mathrm{B}}} \C  $$ 
$$ \mathrm{comp}_{\mathrm{\overline{B}}, \varpi}(+\underline{s}) :  H^1_{\varpi}(X_{\Sigma}, \nabla_{+ \underline{s}})\otimes_{k^\mathrm{dR}} \overline{\C} \To 
H^1(\C\backslash \overline{\Sigma}, \mathcal{L}_{+ \underline{s}})    \otimes_{\Q_{\mathrm{B}}} \C  $$
and  $c^{-}_{\mathrm{B}, \mathrm{dR}} , c^{-}_{\mathrm{\overline{B}}, \mathrm{dR}}$ are similarly defined by replacing all $+$'s with $-$'s.  
Since $k_{\mathrm{dR}}=\Q^\mathrm{dR}_{\underline{s}} \subset \R$, the complex conjugate $\overline{\C}$ in the left-hand side of the previous equation can in fact be replaced with $\C. $
\end{defn} 

 The image of $\widetilde{M}_{\Sigma}$ under  the functor $\mathcal{T}_{\infty} \rightarrow \mathcal{T}$ which forgets all data except for $(V^+_{\mathrm{B}}, V_{\mathrm{dR}}^+, c_{\mathrm{B}, \mathrm{dR}}) $ is the object $M_{\Sigma}$   which was defined in Definition \ref{defnofLP}.

Let us define an $(n \times n)$ matrix 
$$\widetilde{\LP}^{\mm}_{{\Sigma}} \quad \in \quad M_{n\times n} ( \Pe^{\mathrm{B}^+,\mathrm{dR}^+}_{\mathcal{T}_{\infty}})$$
whose entries are 
$$ \left(\widetilde{\LP}_{\Sigma}^{\mm}\right)_{ij} =   \Big[ \widetilde{M}_{\Sigma} \  ,  \    \delta_i \otimes  x^{  s_0} \prod_{k=1}^n (1-x \sigma_k^{-1})^{ s_k}  \ , \   - s_j \omega_j \Big]^{\mm}\ ,$$
where   the Betti homology
class lies in $ H^1_{\mathrm{B}}(X_{\Sigma}, \mathcal{L}_{\underline{s}} ) ^\vee = H_1(\C\setminus\Sigma,\mathcal{L}_{\underline{s}}^\vee)$ and  the de Rham class lies  in   $H^1_{\varpi}(X_{\Sigma}, \nabla_{\underline{s}})$. We will use the notation $\widetilde{M}_{\Sigma}(\underline{s})$ and $\widetilde{L}^\mm_\Sigma(\underline{s})$ when we want to make the dependence on $\underline{s}$ explicit.  The image of $\widetilde{L}^\mm_\Sigma$ under the natural map 
 $\Pe^{\mathrm{B}^+,\mathrm{dR}^+}_{\mathcal{T}_{\infty}} \rightarrow  \Pe^{\mm}_{\mathcal{T}}   $  is $L^{\mm}_{\Sigma}$ and its period is the Lauricella matrix 
  \begin{equation} \label{PeriodofLsigmatilde} \per   \,  \widetilde{\LP}^{\mm}_{{\Sigma}} =    L_{\Sigma} \ . 
  \end{equation}
 
 \subsection{de Rham motivic version and single-valued periods}

    Now let us define a de Rham motivic Lauricella function 
    $$ \left(\widetilde{\LP}_{\Sigma}^{\dR}\right)_{ij} =   \left[ \widetilde{M}_{\Sigma} \  ,  \    \nu_i  \ , \  - s_j \omega_j  \right]^{\dR}  \quad \in \quad \Pe_{\mathcal{T}_{\infty}}^{\mathrm{dR}^-, \mathrm{dR}^+} \ , $$ 
    where the de Rham class  $ \nu_i $ is to be viewed in $H^1_\varpi(X_\Sigma,\nabla_{\underline{s}})=H^1_{\varpi}(X_{\Sigma}, \nabla_{-\underline{s}})^{\vee}$
    and the second is to be viewed in $H^1_{\varpi}(X_{\Sigma}, \nabla_{\underline{s}})$. Note that this differs from the earlier definition of the de Rham Lauricella matrix in the ring of de Rham periods of the category $\mathcal{T}$ because the de Rham classes reside in different cohomology groups. This is required so that we may speak of its single-valued period. By Corollary \ref{cor LSigma sv} the matrix of single-valued periods consists of the single-valued Lauricella hypergeometric functions \eqref{IntroLSigmasv}: 
    $$ \s ( \widetilde{\LP}^{\dR}_{\Sigma}) =\LP^{\s}_{\Sigma}\ ,$$
    whenever all the $s_i$ are real and generic and satisfy $ s_i>0$ for all $0\leq i \leq n$, and $s_0+\cdots + s_n < 1/2$. 
 
 \subsection{Coaction} \label{subsect: Ltildes}
 The coaction  \eqref{Coactionplusminus}, applied to the motivic Lauricella function, will give rise to a third but closely related quantity, given by the matrix 
 $$ \left( \widetilde{L}^{\mm, -}_{\Sigma}(\underline{s})\right)_{ij}= \Big[ \widetilde{M}_{\Sigma}(\underline{s}) \  ,  \   \delta_i \otimes  x^{  s_0} \prod_{k=1}^n (1-x \sigma_k^{-1})^{ s_k}  \ , \     s_j \omega_j \Big]^{\mm}   \quad \in \quad \Pe_{\mathcal{T}_{\infty}}^{\mathrm{B}^+, \mathrm{dR}^-}$$
 where the de Rham class   $ s_j \omega_j$ is in    $H^1_{\varpi}(X_{\Sigma}, \nabla_{-\underline{s}})$.
 It is the image under $F_{\infty}$ of 
 $$ \left(\widetilde{\LP}^{\mm}_{\overline{\Sigma}}(-\underline{s})\right)_{ij} =   \Big[ \widetilde{M}_{\Sigma}(\underline{s}) \  ,  \    \overline{\delta}_i \otimes  x^{-s_0} \prod_{k=1}^n (1-x \sigma_k^{-1})^{-s_k}  \ , \     s_j \omega_j \Big]^{\mm}   \quad \in \quad \Pe_{\mathcal{T}_{\infty}}^{\mathrm{\overline{B}}^-, \mathrm{dR}^-} \ ,$$
where   the Betti
class is viewed in $H_1(\C \backslash \overline{\Sigma}, \mathcal{L}_{-\underline{s}}^\vee)$ (see Remark \ref{rem: compute frobenius betti})
and  the de Rham class  $ s_j \omega_j$ is  viewed in  $H^1_{\varpi}(X_{\Sigma}, \nabla_{-\underline{s}})$ as before. More precisely, we have
$$       \widetilde{L}^{\mm, -}_{\Sigma}(\underline{s})  =   F_{\infty} \widetilde{\LP}^{\mm}_{\overline{\Sigma}}( - \underline{s}) $$
and hence, by definition of the period homomorphism $\per \, F_{\infty}  : \Pe_{\mathcal{T}_{\infty}}^{\mathrm{B}^+, \mathrm{dR}^-} \rightarrow \C $
the period is 
$$  \per \,  F_{\infty} \left( \widetilde{L}^{\mm, -}_{\Sigma}(\underline{s}) \right)= L_{\overline{\Sigma}} (-\underline{s}) \ .$$
\begin{cor}
The  coaction \eqref{Coactionplusminus} satisfies
$$\Delta    \widetilde{\LP}^{\mm}_{\Sigma}(\underline{s})=  \widetilde{\LP}^{\mm,-}_{\Sigma}(\underline{s})\otimes  \widetilde{\LP}^{\dR}_{\Sigma}(\underline{s})   \ .$$
\end{cor} 
\begin{proof} This is an immediate consequence of the formula for the coaction on matrix coefficients in a Tannakian category. 
\end{proof} 
Equation  \eqref{perconvolution}, applied to the matrix $ \widetilde{\LP}^{\mm}_{\Sigma}(\underline{s})$ reduces to the equation 
$$  L_{\Sigma}(\underline{s}) =  L_{\overline{\Sigma}}(-\underline{s})\, L^{\s}_{\Sigma}(\underline{s})$$
 which is another way of writing formula \eqref{splusminusDC}. Consider the following  refinement of  the period map:
$$\xymatrixcolsep{5pc}\xymatrix{ \Pe^{\mathrm{B}^+,\mathrm{dR}^+}_ {\mathcal{T}_{\infty}} \ar[r]^-{\Delta} & \Pe^{\mathrm{B}^+,\mathrm{dR}^-}_ {\mathcal{T}_{\infty}} \otimes_{k_{\mathrm{dR}}}  \Pe^{\mathrm{dR}^-,\mathrm{dR}^+}_ {\mathcal{T}_{\infty}}  \ar[r]^-{\per\,F_{\infty}\otimes 
 \s^{-,+}} &  \C\otimes_{\Q} \C\ .}$$ 
 Indeed,  by composing it with the multiplication map $\C\otimes_\Q\C\to \C$ one recovers the usual period map, by \eqref{perconvolution}.
Then under this map, the elements $(\widetilde{\LP}^{\mm}_{\Sigma}(\underline{s}))_{ij} $ for $1\leq i,j\leq n$
  map to 
\begin{multline} \sum_{\ell=1}^{n} \left(s_{\ell}    \int_{\delta_i}       x^{-s_0} \prod_{k=1}^n (1- x \overline{\sigma_k}^{-1})^{-s_k} \frac{dx}{x-\overline{\sigma_{\ell}}}  \right) \otimes  \\
  \frac{s_j }{2\pi i} \left( \iint_{\C} |z|^{ 2s_0} \prod_{k=1}^n |1- z \sigma_k^{-1} |^{  2s_k}   \,\left(  \frac{d\overline{z}}{\overline{z}-\overline{\sigma_{\ell}}}- \frac{d \overline{z}}{\overline{z}} \right)  \wedge \frac{dz}{z-\sigma_j} \right) \ .
  \end{multline}
The point of this formula is that both sides of the tensor product admit a Taylor
 expansion in the $s_i$, which is the subject of the next section.

\subsection{Variant} \label{par:variant formal s} 
We had to assume that the parameters $s_i$ are real in order to obtain a Tannakian interpretation of the single-valued period homomorphism. This was so that the subfield $\overline{k}(s_0,\ldots,s_n)\subset\C$ is isomorphic to $k(s_0,\ldots,s_n)$, which ensures that there is  a comparison map for de Rham cohomology associated to the complex conjugate embedding of $k$:
$$H^1_{\mathrm{dR}}(X_\Sigma,\nabla_{\underline{s}})\otimes_{k_{\underline{s}}^{\mathrm{dR}}}\overline{\C} \To H^1_{\mathrm{B}}(\C\backslash\overline{\Sigma},\mathcal{L}_{\underline{s}})\otimes_{\Q_{\underline{s}}^{\mathrm{dR}}}\C\ .$$
However, such an assumption is unnatural, for instance because  formulae such as \eqref{introDC} are true for all $s_i\in\C$ for which they make sense. A way to remedy this would be to treat the $s_i$ as formal parameters and work with modules over the polynomial ring $k[s_0,\ldots,s_n]$. This would also be needed to build a bridge between the framework of cohomology with coefficients that we discussed in this section and the Taylor expansions that we will consider in the rest of this article. 

\section{Laurent series expansion of periods of cohomology with coefficients} 
\label{sect: Renorm}
The Lauricella functions \eqref{IntroLSigma} are not \emph{a priori} defined for  $s_0,\ldots, s_n$ at the origin.  We show  using a renormalisation procedure that 
they  extend to a neighbourhood of  the origin and admit a Taylor expansion there. We prove a similar statement for their single-valued versions \eqref{IntroLSigmasv}. The reason for the (ab)use of the word `renormalisation' is explained in \cite{BD2}.

\subsection{Renormalisation of line integrals}

    Let $\Sigma=\{\sigma_0,\ldots,\sigma_n\}$ be distinct complex numbers with $\sigma_0=0$. We fix an index $i\in\{1,\ldots,n\}$ and a locally finite path $\delta_i$ in $\mathbb{C}\setminus\Sigma$ from $0$ to $\sigma_i$. More precisely, $\delta_i:(0,1)\to \mathbb{C}\setminus \Sigma$ is a smooth path that can be extended to a smooth path $\overline{\delta_i}:[0,1]\to \mathbb{C}$ satisfying $\overline{\delta_i}(0)=0$ and $\overline{\delta_i}(1)=\sigma_i$. As in all this paper, we assume that $\delta_i$ does not wind infinitely around $0$ or  $\sigma_i$ (Remark \ref{rem: choice delta i winding}).
    
    Let $\omega$ be a global meromorphic form on $\mathbb{P}^1(\mathbb{C})$ with logarithmic poles at $\Sigma\cup\{\infty\}$. We define
    $$\Omega= x^{s_0}\prod_{k=1}^n(1-x\sigma_k^{-1})^{s_k}\,\omega$$
    and view it as a  multi-valued $1$-form which depends holomorphically on the parameters $\underline{s}$, and that we wish to integrate along $\delta_i$.
    
    \begin{rem}\label{rem: convention branches}
    Since we shall only consider the integral of $\Omega$ along $\delta_i$,  it suffices to define the pull-back $\delta_i^*\Omega$, which is a  $1$-form on the contractible space $(0,1)$. It depends only on a determination of $\log(\sigma_i)$, which we choose once and for all. Hence $\sigma_i^{s_0} = \exp (s_0 \log(\sigma_i))$ is defined. This fixes a branch of $ \delta_i^*\Omega$ as follows. Close to $1$, (i.e., for $x$  near $\sigma_i$) the function $x^{s_0}$ is prescribed and so $\delta^*_i x^{s_0}$ is analytically continued to the open interval $(0,1)$. Close to $0$ (i.e., for $x$ near $0$) the function  $\log (1-x\sigma_k^{-1})$ is analytic and the terms $(1-x\sigma_k^{-1})^{s_k} = \exp  \left(s_k \log(1-x\sigma_k^{-1})\right)$ are well-defined, and so their pull-backs along $\delta$ are analytically continued to $(0,1)$.  
    \end{rem}
    
    \begin{rem}\label{rem: global class specialise}
    Let us fix a parameter $\underline{s}=(s_0,\ldots,s_n)\in\mathbb{C}^{n+1}$. The recipe given in Remark \ref{rem: convention branches} gives rise to a determination of $x^{s_0}\prod_{k=1}^n(1-x\sigma_k^{-1})^{s_k}$ along the path $\delta_i$, and thus a class in the cohomology group  $H_1^\lf(X_\Sigma(\mathbb{C}),\mathcal{L}_{\underline{s}}^\vee)$ which we denote $\delta_{i,\underline{s}}$. 
    \end{rem}
    
    Without any assumption on the residues of $\omega$ at $0$ and $\sigma_i$, the integral of $\Omega$ along $\delta_i$ defines a holomorphic function of the parameters $\underline{s}$ in the domain $\{\operatorname{Re}(s_0),\operatorname{Re}(s_i)>0\}$.

    \begin{defn} \label{definition: renormalisedforms} Following \cite{BD2}, let us define the renormalised version of $\Omega$ with respect to the points $\{0,\sigma_i\}$ to be
    $$\Omega^{\ren_{i}}=  \Omega -    \Res_0(\omega)\, x^{s_0}\frac{dx}{x} - \Res_{\sigma_i}(\omega)\Big(\sigma_i^{s_0}\prod_{k\neq i}(1-\sigma_i\sigma_k^{-1})^{s_k}\Big)\, (1-x\sigma_i^{-1})^{s_i} \frac{dx}{x-\sigma_i} \ \cdot$$
    \end{defn} 
    
    \begin{prop}\label{proprenorm}
    We have for every $\underline{s}$ in the domain $\{\operatorname{Re}(s_0),\operatorname{Re}(s_i)>0\}$, the equality
    \begin{equation}  \label{renormintomega}\int_{\delta_i}\Omega = \Res_0(\omega)\frac{\sigma_i^{s_0}}{s_0} - \Res_{\sigma_i}(\omega)\frac{\sigma_i^{s_0}\prod_{k\neq i}(1-\sigma_i\sigma_k^{-1})^{s_k}}{s_i} + \int_{\delta_i}\Omega^{\ren_i}\ .\end{equation}
    The integral of $\Omega^{\ren_i}$ along $\delta_i$ defines a holomorphic function of the parameters $\underline{s}$ in the domain $\{\operatorname{Re}(s_0),\operatorname{Re}(s_i)>-1\}$.
    \end{prop}

    \begin{proof}
    The equality follows from the computations:
    $$\int_{\delta_i}x^{s_0}\frac{dx}{x} = \frac{\sigma_i^{s_0}}{s_0} \qquad \mbox{and} \qquad \int_{\delta_i}(1-x\sigma_i^{-1})^{s_i}\frac{dx}{x-\sigma_i} = -\frac{1}{s_i}\ \cdot$$
    Since $\omega-\Res_0(\omega)\frac{dx}{x}$ does not have a pole at $0$, the singularities of $\Omega^{\ren_i}$ at $0$ are at worst of the type $x^{s_0}dx$ and therefore integrable for $\operatorname{Re}(s_0)>-1$. Likewise, its singularities at $\sigma_i$ are at worst of the type $(1-x\sigma_i^{-1})^{s_i}dx$ and are integrable for $\operatorname{Re}(s_i)>-1$.
    \end{proof}

    Since $\underline{s}=0$ is in the domain of convergence of the integral of $\Omega^{\ren_i}$ along $\delta_i$, we see that $ \eqref{renormintomega}$ yields a Laurent series expansion at $\underline{s}=0$ for the integral of $\Omega$ along $\delta_i$. We record the following special case which shows that the Lauricella functions $(L_\Sigma)_{ij}$ admit Taylor expansions around $\underline{s}=0$. Here the integrand of $(L_\Sigma)_{ij}$ is understood according to the conventions of Remark \ref{rem: convention branches}.

    \begin{prop}\label{prop: taylor series LSigma}
    We have the following equality
    $$(L_\Sigma)_{ij} = \mathbf{1}_{i=j}\Big(\sigma_i^{s_0}\prod_{k\neq i}(1-\sigma_i\sigma_k^{-1})^{s_k}\Big)-s_j\int_{\delta_i}\Omega_j^{\ren_i}$$
    where
    $$\Omega_j^{\ren_i}= \Big(x^{s_0}\prod_{k\neq i}(1-x\sigma_k^{-1})^{s_k} - \mathbf{1}_{i=j}\, \sigma_i^{s_0}\prod_{k\neq i}(1-\sigma_i\sigma_k^{-1})^{s_k}\Big)(1-x\sigma_i^{-1})^{s_i}\frac{dx}{x-\sigma_j}\ \cdot$$
    and the integral on the right-hand side defines a holomorphic function of the parameters $\underline{s}$ in the domain $\{\operatorname{Re}(s_0),\operatorname{Re}(s_i)>-1\}$. 
    \end{prop}
    
    \begin{proof}
    This follows from applying Proposition \ref{proprenorm} to the case $\omega=\omega_j=d\log(x-\sigma_j)$ and multiplying \eqref{renormintomega} by $-s_j$.
    \end{proof}
    
\subsection{Renormalisation of complex volume integrals}

    Let $\Sigma=\{\sigma_0,\ldots,\sigma_n\}$ be distinct complex numbers with $\sigma_0=0$. Let $\omega$ be a global meromorphic form on $\mathbb{P}^1(\mathbb{C})$ with logarithmic poles at $\Sigma\cup\{\infty\}$. We define
    $$\Omega^{\s} \ = \  |z|^{2 s_0}  \prod_{k=1}^n |1-z\sigma_k^{-1}|^{2s_k} \,  \omega$$
    and view it as a smooth $1$-form on $\mathbb{P}^1(\mathbb{C})\setminus (\Sigma\cup\{\infty\})$ which depends holomorphically on the parameters $\underline{s}$. We fix an index $i\in\{1,\ldots,n\}$ and consider the integral
    $$-\frac{1}{2\pi i}\iint_{\mathbb{C}}\left(\frac{d\overline{z}}{\overline{z}-\overline{\sigma_i}}-\frac{d\overline{z}}{\overline{z}}\right)\wedge \Omega^\s\ .$$
    Without any assumption on the residues of $\omega$, this integral defines a holomorphic function of the parameters $\underline{s}$ in the domain 
    $$D=\{\Real(s_0),\Real(s_i)>0\; ,\; \Real(s_k)>-1/2 \;\;(k\notin\{0,i\})\; , \; \Real(s_0)+\cdots+\Real(s_n)<1/2\}\ .$$
    One can check this by passing to local coordinates around every point of $\Sigma\cup\{\infty\}$.
    
    \begin{defn} \label{definition: renormalisedformssv} Following \cite{BD2}, let us define the renormalised version of $\Omega^\s$ with respect to the points $\{0,\sigma_i\}$ to be
    $$\Omega^{\s,\ren_{i}}=  \Omega^s -   \Res_0(\omega)\, |z|^{2s_0}\frac{dz}{z} - \Res_{\sigma_i}(\omega)\Big(|\sigma_i|^{2s_0}\prod_{k\neq i}|1-\sigma_i\sigma_k^{-1}|^{2s_k}\Big)\, |1-z\sigma_i^{-1}|^{2s_i} \frac{dz}{z-\sigma_i}\ \cdot$$
    \end{defn}

    The following lemma will be used several times in the sequel. It is a single-valued analogue of the identity $\int_0^{\sigma}z^s\frac{dz}{z}=\frac{\sigma^s}{s}$.

    \begin{lem} \label{lem: svlogintegralregterm} For any $\sigma,s \in \C$ with $\mathrm{Re}(s)>0$, we have
    $$-\frac{1}{2\pi i} \iint_{\C}  \left(\frac{d\overline{z}}{\overline{z}-\overline{\sigma}}  -\frac{d\overline{z}}{\overline{z}} \right)\wedge |z|^{2s}\frac{dz}{z} = \frac{|\sigma|^{2s}}{s}\ .$$
    \end{lem} 

    \begin{proof}
    One verifies in local polar coordinates  that the integral converges. Let $$F = - \frac{1}{s} |z|^{2s} \left( \frac{ d\overline{z}}{\overline{z}-\overline{\sigma}}  -\frac{d\overline{z}}{\overline{z}}  \right)\ . $$
    Its total derivative is  the integrand on the left-hand side. By Stokes' formula applied to  the complement of three discs in  $\Pro^1(\C)$  centered at $0$, $\infty$, $\sigma$,
    it suffices to compute 
    $$-\frac{1}{2\pi i}  \int_{D_\varepsilon}    F \ = \   \frac{1}{s}\frac{
1}{2\pi i}    \int_{D_\varepsilon}     |z|^{2s}\left(\frac{ d\overline{z}}{\overline{z}-\overline{\sigma}}  -\frac{d\overline{z}}{\overline{z}} \right)$$
    where $D_\varepsilon$ is the union of three circles of radius $\varepsilon$ winding negatively around  $0$, $\infty$, and $ \sigma$. When $\varepsilon$ goes to zero, the integral around the first and second vanish, and the integral around the third gives $2\pi i\,|\sigma|^{2s}$.
    \end{proof} 
    
    \begin{prop}\label{proprenormsv}
    We have for every $\underline{s}$ in the domain $D$ the equality
    \begin{equation}\label{renormintomegasv}
    \begin{split}
    -&\frac{1}{2\pi i}\iint_{\mathbb{C}}\left(\frac{d\overline{z}}{\overline{z}-\overline{\sigma_i}}-\frac{d\overline{z}}{\overline{z}}\right)\wedge\Omega^\s = \\ & \Res_0(\omega)\frac{|\sigma_i|^{2s_0}}{s_0} - \Res_{\sigma_i}(\omega)\frac{|\sigma_i|^{2s_0}\prod_{k\neq i}|1-\sigma_i\sigma_k^{-1}|^{2s_k}}{s_i} -\frac{1}{2\pi i} \iint_{\mathbb{C}}\left(\frac{d\overline{z}}{\overline{z}-\overline{\sigma_i}}-\frac{d\overline{z}}{\overline{z}}\right)\wedge\Omega^{\s,\ren_i}\ .
    \end{split}
    \end{equation}
    The integral on the right-hand side defines a holomorphic function of the parameters $\underline{s}$ in the domain 
    \begin{equation}\label{eq: domain d prime}
    D'=\{\operatorname{Re}(s_k)>-1/2 \;\;(k\in\{0,\ldots,n\})\; ,\; \Real(s_0)+\cdots+\Real(s_n)<1/2\}\ .
    \end{equation}
    \end{prop}

    \begin{proof}
    The equality follows from the computations
    $$-\frac{1}{2\pi i}\iint_{\mathbb{C}}\left(\frac{d\overline{z}}{\overline{z}-\overline{\sigma_i}}-\frac{d\overline{z}}{\overline{z}}\right)\wedge |z|^{2s_0}\frac{dz}{z} = \frac{|\sigma_i|^{2s_0}}{s_0} \ ,$$
    and
    $$-\frac{1}{2\pi i}\iint_{\mathbb{C}}\left(\frac{d\overline{z}}{\overline{z}-\overline{\sigma_i}}-\frac{d\overline{z}}{\overline{z}}\right)\wedge |1-z\sigma_i^{-1}|^{2s_i}\frac{dz}{z-\sigma_i} = - \frac{1}{s_i}\ ,$$
    which follows from Lemma \ref{lem: svlogintegralregterm} (for the second, apply the change of variables $z\leftrightarrow \sigma-z$ and multiply by $|\sigma|^{-2s}$). Since $\omega-\Res_0(\omega)\frac{dx}{x}$ does not have a pole at $0$, the singularities of $\Omega^{\s,\ren_i}$ at $0$ are at worst of the type $|z|^{2s_0}dz$ and therefore 
    $$\left(\frac{d\overline{z}}{\overline{z}-\overline{\sigma_i}}-\frac{d\overline{z}}{\overline{z}}\right)\wedge\Omega^{\s,\ren_i}$$
    is integrable around zero for $\operatorname{Re}(s_0)>-1/2$. Likewise, it is integrable around $\sigma_i$ for $\operatorname{Re}(s_i)>-1/2$.
    \end{proof}
    
    Since $\underline{s}=0$ lies in the domain $D'$, this proposition provides a Laurent series expansion at $\underline{s}=0$ for the integral on the left-hand side. We record the following special case which shows that the single-valued Lauricella functions \eqref{IntroLSigmasv} admit Taylor expansions around $\underline{s}=0$.
    
    \begin{prop}\label{prop: taylor series LSigma sv}
    We have the following equality
    $$(L^\s_\Sigma)_{ij} = \mathbf{1}_{i=j}\Big(|\sigma_i|^{2s_0}\prod_{k\neq i}|1-\sigma_i\sigma_k^{-1}|^{2s_k}\Big)+\frac{s_j}{2\pi i}\iint_{\mathbb{C}}\left(\frac{d\overline{z}}{\overline{z}-\overline{\sigma_i}}-\frac{d\overline{z}}{\overline{z}}\right)\wedge \Omega_j^{\s,\ren_i}$$
    where
    $$\Omega_j^{\s,\ren_i}= \Big(|z|^{2s_0}\prod_{k\neq i}|1-z\sigma_k^{-1}|^{2s_k} - \mathbf{1}_{i=j}\,|\sigma_i|^{2s_0}\prod_{k\neq i}|1-\sigma_i\sigma_k^{-1}|^{2s_k}\Big)|1-z\sigma_i^{-1}|^{2s_i}\frac{dz}{z-\sigma_j}\ \cdot$$
    The  integral on the right-hand side defines a holomorphic function of the parameters $\underline{s}$ in the domain $D'$ described in \eqref{eq: domain d prime}.
    \end{prop}
    
    \begin{proof}
    This follows from applying Proposition \ref{proprenormsv} to the case $\omega=\omega_j=d\log(x-\sigma_j)$ and multiplying \eqref{renormintomegasv} by $-s_j$.
    \end{proof}

\section{Notations relating to motivic fundamental groups}   \label{sect: Notations}
In this, the second part of the paper, we study  \eqref{IntroLSigma} from the point of view of the motivic fundamental group of the punctured Riemann sphere. This requires a number of notations and background, mostly from \cite{delignegoncharov} and \cite{brownnotesmot}, which we recall here.
Let $k\subset \C$ be a number field.

\subsection{Categorical and Tannakian} 
\begin{enumerate}
\item  Let $\MT(k)$ denote the category of mixed Tate motives over $k$.  One can replace $\MT(k)$ with a  category of Betti and de Rham realisations, with  essentially no change  to our  arguments.  The category   $\MT(k)$ has a canonical fiber functor $\varpi: \MT(k) \rightarrow \mathrm{Vec}_{\Q}$. Let us set
$$G_{\MT(k)}^{\varpi} =\mathrm{Aut}_{\varpi}^{\otimes}\,  \MT(k)\ . $$ 
It is an affine group scheme over $\Q$.  Let us denote by $\omega_{\mathrm{B}}: \MT(k) \rightarrow \mathrm{Vec}_{\Q}$ the Betti realisation functor with respect to the given embedding $k \subset \C$.   The de Rham realisation functor $\omega_{\mathrm{dR}}: \MT(k)\rightarrow \mathrm{Vec}_k$ is obtained from $\varpi$ by extending scalars: $\omega_{\mathrm{dR}}= \varpi \otimes_{\Q} k$.
\vspace{0.05in}

\item   Let $\Pe^{\mm} = \Or(\mathrm{Isom}^{\otimes}_{\MT(k)}(\varpi,\omega_{\mathrm{B}}))$ denote 
the $\Q$-algebra of motivic periods  on $\MT(k)$.
Let $$\Pe^{\varpi} = \Or(G_{\MT(k)}^{\varpi})$$
denote the   $\Q$-algebra of of (canonical,  i.e., `$(\varpi,\varpi)$')) de Rham  motivic periods. The latter is a graded Hopf algebra, and the former  is a graded algebra comodule over it.   Denote the corresponding motivic coaction by 
$$\Delta: \Pe^{\mm} \To \Pe^{\mm}  \otimes_{\Q} \Pe^{\varpi}\ .$$
We let $\per : \Pe^{\mm} \rightarrow \C$ denote the period homomorphism and $\s:\Pe^\varpi\to \mathbb{C}$ denote the single-valued period homomorphism \cite[\S 2.6]{BD1}.
  
\item   
\vspace{0.05in}
 Let $\Lef^{\mm} = [ \Q(-1), 1_{\mathrm{B}}^{\vee}, 1_{\varpi}]^{\mm}$, resp. $\Lef^{\varpi} = [ \Q(-1), 1_{\varpi}^{\vee}, 1_{\varpi}]^{\varpi}$, denote the Lefschetz motivic period, whose period is $2\pi i$, resp.  its (canonical) de Rham version.  

\vspace{0.05in}

\item Let $\Pe^{\mm,+}\subset \Pe^\mm$ denote the subspace of effective motivic periods. It is spanned by the motivic periods of objects $M\in \MT(k)$ with $W_{-1}M=0$. It is a non-negatively graded sub-algebra and sub-$\Pe^\varpi$-comodule of $\Pe^\mm$, and contains the Lefschetz motivic period $\Lef^\mm$. There is a canonical projection homomorphism
$$ \pi^{\mm,+}_{\varpi} : \Pe^{\mm,+} \To  \Pe^{\varpi} $$
which, in particular,  sends $\Lef^{\mm}$ to zero.  It can be  defined by composing the coaction $\Delta$ with projection onto the weight-graded zero piece 
$W_0  \Pe^{\mm,+} \cong \Q$.

\end{enumerate}

\subsection{Geometric}  \label{subsect: NotGeometric} Let  $\Sigma = \{\sigma_0, \sigma_1,\ldots, \sigma_n \} \subset \A^1(k)$ be distinct points with $\sigma_0=0$. Recall that $X_{\Sigma}= \A_k^1 \backslash \Sigma.$

\begin{enumerate}
\item Let $t_0 \in T_0\A^1_k$ denote the tangent vector $1$ at $0$. For each $ i\geq 1$, let 
$$t_{\sigma_i}  \quad \in \quad  T_{\sigma_i}\A_k^1$$
be the tangent vector  $\sigma_i$ based at the point $\sigma_i$.  Let $\pi_1^{\mathrm{top}}(X_\Sigma(\mathbb{C}),t_i,-t_j)$ denote the fundamental torsor of homotopy classes of paths between the tangential basepoints $t_i$ and $-t_j$. Concretely, an element of this set is represented by a (piecewise) smooth path $\gamma:(0,1)\to \mathbb{C}\setminus \Sigma$ that can be extended to a  path $\overline{\gamma}:[0,1]\to \mathbb{C}$ which is smooth at $0$ and $1$ and satisfies $\overline{\gamma}(0)=\sigma_i$ and $\overline{\gamma}(1)=\sigma_j$, with prescribed velocities $\overline{\gamma}'(0)=t_i$ and $\overline{\gamma}'(1)=t_j$. 
\vspace{0.05in}
\item  For all $0\leq i,j\leq n$, denote by
$${}_i \Pi^{\bullet}_{j}  = \pi_1^{\bullet}(X_{\Sigma}, t_i, -t_j) \qquad \hbox{ where } \quad  \bullet \in \{\mathrm{B}, \varpi , \mathrm{mot} \}$$
the Betti, (canonical) de Rham,  or motivic fundamental torsor of paths from the tangential base point  $t_i$ at  $\sigma_i$,  to $-t_j$ at  $\sigma_j$. The versions ${}_i\Pi_j^{\mathrm{B}}$ and ${}_i\Pi_j^\varpi$ are affine $\mathbb{Q}$-schemes; the version ${}_i\Pi_j^{\mathrm{mot}}$ is an affine scheme in the category $\MT(k)$, which simply means that its affine ring is isomorphic to an algebra ind-object $\mathcal{O}({}_i\Pi_j^{\mathrm{mot}})$ in $\MT(k)$, whose ($\mathrm{B}$, resp. $\varpi$) realisations are $\mathcal{O}({}_i\Pi_j^{\mathrm{B}})$ and $\mathcal{O}({}_i\Pi_j^{\varpi})$. There is a groupoid structure ${}_a \Pi^{\bullet}_b \times {}_b \Pi^{\bullet}_c \To {}_a \Pi^{\bullet}_c \ ,$ for all $a,b,c$ in the set of our tangential basepoints (composition of paths). 
There are maps
$$ \gamma \mapsto \gamma^{\mathrm{B}} \ :  \   \pi_1^{\mathrm{top}}(X_{\Sigma}(\C), t_i, -t_j)  \To {}_i \Pi^{\mathrm{B}}_{j}(\Q)\ $$
which are Zariski-dense and compatible with the groupoid structure. (They identify ${}_i\Pi_j^{\mathrm{B}}$ as the Mal\v{c}ev, or pro-unipotent, completion of $\pi_1^{\mathrm{top}}(X_\Sigma(\mathbb{C}), t_i,-t_j)$.)
\vspace{0.05in}
\item  The scheme ${}_i\Pi_j^\varpi$ does not depend on $i,j$, although the action of the (canonical) motivic Galois group $G_{\MT(k)}^\varpi$ upon it does depend on $i,j$. On $X_{\Sigma}$ we considered  the  logarithmic 1-forms 
$\omega_i$ for $i=0,\ldots,n$ \eqref{omegajdef}. 
Since they have residue $0$ or  $\pm 1$ at points of $\Sigma$, 
 they  generate  the canonical $\Q$-structure (or $\varpi$-structure) on the de Rham realisation of $H^1(X_{\Sigma}) \in \mathrm{Ob} (\MT(k))$, 
which we shall denote simply by $H_{\varpi}^1(X_{\Sigma})$. It is the $\Q$-vector space spanned by the classes $\omega_i$. 
The affine ring of the de Rham canonical torsor of paths  is 
$$\Or({}_i \Pi_{j}^\varpi)  \cong  \bigoplus_{n\geq 0}  H^1_{\varpi}(X_{\Sigma})^{\otimes n}\ .$$
It is isomorphic to the graded tensor coalgebra on $H^1_{\varpi}(X_{\Sigma})$, equipped with the shuffle product $\sha$ and deconcatenation coproduct. 
For any commutative unital $\Q$-algebra $R$, the $R$-points of ${}_i \Pi_{j}^\varpi$
$$ {}_i \Pi_{j}^\varpi (R) \  \subset \   R \langle \langle e_0, \ldots, e_n \rangle \rangle $$
is the set of group-like formal power series  with respect to the continuous coproduct for which the $e_i$ are primitive.  They are formal power series
$$ S   = \sum_{w\in \{e_0,e_1,\ldots, e_n\}^{\times}} S(w) \, w \quad   \in \quad   R \langle \langle e_0, \ldots, e_n \rangle \rangle $$
where the linear extension of the map $w\mapsto S(w)$ is a homomorphism with respect to the shuffle product.  The letters $e_i$, for $1\leq i \leq n$, are dual to  the $\omega_i$. 
\vspace{0.05in}

\vspace{0.05in}
\item  For all $0 \leq i \leq n$, let ${}_01_i \in {}_0\Pi_i^\varpi(\Q)$ denote the canonical  $\varpi$-path. It is defined by the augmentation map $\Or({}_0\Pi_i^\varpi) \rightarrow \Q$ onto the degree zero component (or by quotienting  by the Hodge filtration $F^1$). It is the formal power series $1 \in \Q\langle\langle e_0,\ldots, e_n \rangle \rangle$ consisting only of the empty word. 

\vspace{0.05in}
\item Since the motivic fundamental torsor of paths is the spectrum of an ind-object  in the category $\MT(k)$, there is a canonical  universal comparison isomorphism of schemes \cite[\S 4.1]{brownnotesmot}
$$ \comp^{\mm}_{\mathrm{B}, \varpi} \quad :   \quad {}_i\Pi^{\mathrm{B}}_j \times_{\Q} \Pe^{\mm} \overset{\sim}{\To} {}_i\Pi_j^\varpi \times_{\Q} \Pe^{\mm} $$     
for all $0\leq i,j \leq n$, compatible with the groupoid structure.
Composing $\comp^{\mm}_{\mathrm{B},\varpi}$ with the period homomorphism gives back the canonical comparison isomorphism whose coefficients are regularised interated integrals:
$$\comp_{\mathrm{B}, \varpi} \quad : \quad  {}_i\Pi^{\mathrm{B}}_j \times_{\Q} \C \overset{\sim}{\To} {}_i\Pi_j^\varpi \times_{\Q} \C\ .$$

\end{enumerate}

\section{Generalised associators and their beta quotients}

    We now study the generalised associators $\mathcal{Z}^i$, which are  formal power series in  non-commuting variables $e_0,\ldots,e_n$ whose coefficients are regularised iterated integrals on $X_\Sigma$. We compute their beta quotients, which are formal power series in commuting  variables $s_0,\ldots,s_n$, and use them to define a matrix $FL_\Sigma$ of power series which we prove  to be the matrix of Taylor series of the Lauricella functions $L_\Sigma$ \eqref{IntroLSigma} (Theorem \ref{thm: FL equals L} below, which is Theorem \ref{introthmperandsv} (\emph{i}) from the introduction).

\subsection{Generalised (motivic) associators}

Le us fix an index $1\leq i\leq n$. The generalised associator $\mathcal{Z}^i$, that we will soon define, is a formal power series that records \emph{all} the iterated integrals that can be computed on $X_\Sigma$ on a given (homotopy class of) path $\gamma_i$ between the tangential basepoints $t_0$ and $-t_i$. For this we will need to fix a class $\gamma_i\in \pi_1^{\mathrm{top}}(X_\Sigma(\mathbb{C}),t_0,-t_i)$ as in \S\ref{subsect: NotGeometric} (1) above. We will need to impose an extra condition on $\gamma_i$ that can be stated as follows. Consider the following maps between fundamental torsors:
\begin{equation}\label{eq: maps between fundamental groups}
\pi_1^{\mathrm{top}}(X_\Sigma(\mathbb{C}),t_0,-t_i) \longrightarrow \pi_1^{\mathrm{top}}(\mathbb{C}\setminus \{0,\sigma_i\},t_0,-t_i) \longrightarrow \pi_1^{\mathrm{top}}(\mathbb{C}\setminus \{\sigma_i\}, 0, -t_i)\simeq \mathbb{Z}\ .
\end{equation}
The first map is induced by the inclusion $X_\Sigma(\mathbb{C})\hookrightarrow \mathbb{C}\setminus \{0,\sigma_i\}$. The second map is also induced by the inclusion $\mathbb{C}\setminus \{0,\sigma_i\}\hookrightarrow\mathbb{C}\setminus \{\sigma_i\}$, where the tangential basepoint $t_0$ becomes the usual basepoint $0$.  There is a canonical element  in $\pi_1^{\mathrm{top}}(\mathbb{C}\setminus \{\sigma_i\}, 0, -t_i)$, namely the homotopy class of the straight path $t\mapsto \sigma_i t$, which provides the identification of
the torsor $\pi_1^{\mathrm{top}}(\mathbb{C}\setminus \{\sigma_i\}, 0, -t_i)$ with the fundamental group $\pi_1^{\mathrm{top}}(\mathbb{C}\setminus \{\sigma_i\}, 0)\simeq \mathbb{Z}$.

\begin{defn}\label{defi: gamma i admissible}
The class $\gamma_i\in \pi_1^{\mathrm{top}}(X_\Sigma(\mathbb{C}),t_0,-t_i)$ is said to be \emph{admissible} if its image by the composite \eqref{eq: maps between fundamental groups} is the class of the straight path $t\mapsto \sigma_it$.
\end{defn}

Roughly speaking, it means that $\gamma_i$ ``does not wind around $\sigma_i$''. A representative of an admissible class in $\pi_1^{\mathrm{top}}(X_\Sigma(\mathbb{C}), t_0,-t_i)$ can be constructed as follows: first travel via a small arc from the tangential basepoint $t_0$ to $\sigma_i\varepsilon$ for a small $\varepsilon>0$, and then travel via the straight path $t\mapsto \sigma_i t$ towards the tangential basepoint $-t_i$ (avoiding if necessary the points $\sigma_j$ that are on the line segment between $\sigma_i\varepsilon$ and $\sigma_i$ via small arcs).
We will see in Remark \ref{rem: admissible vanishing} below that the admissibility of $\gamma_i$ is equivalent to the vanishing of the following regularised iterated integral:
$$\int_{\gamma_i}\omega_i=0\ .$$
We fix an admissible class $\gamma_i$ for every $1\leq i\leq n$ for the remainder of this article.

\begin{defn} Define  a formal power series
$$\Zim        \  \in \          {}_0\Pi_i^\varpi (\Pe^{\mm}) \subset  \Pe^{\mm} \langle \langle e_0,\ldots, e_n\rangle\rangle   $$
by 
$\Zim = \comp^{\mm}_{\mathrm{B},\varpi} (\gamma^{\mathrm{B}}_i)$
where $\gamma^{\mathrm{B}}_i\in {}_0\Pi_i^{\mathrm{B}}(\Q)$ is the image of $\gamma_i$ as in \S\ref{subsect: NotGeometric} (2) above. It is called a \emph{generalised motivic associator}.
\end{defn}

Since $\Or({}_0 \Pi_i^\varpi)$ has weights $\geq 0$, 
it follows that $\Zim$ actually lies in $ {}_0\Pi_i (\Pe^{\mm,+})$.  Explicitly,
$$\Zim  = \sum_{w\in \{e_0,\ldots, e_n\}^{\times}} \left[ \Or(\pi_1^{\mathrm{mot}}(X_{\Sigma}, t_0, -t_i)), \gamma_i^{\mathrm{B}}, w\right]^{\mm} w\  $$
where the sum is over all words $w$ in $e_i$, which  are in turn dual to words in the $\omega_i$ \eqref{omegajdef} and hence define an element $w \in   \Or( {}_0\Pi_i^\varpi)$.   The path $\gamma_i^{\mathrm{B}}$ is viewed as an element in $\Or( {}_0\Pi^{\mathrm{B}}_i)^{\vee}$.   

\begin{defn}
The image $\ZZ^i =\mathrm{per} \left( \Zim \right)$  of $\Zim$ under the period homomorphism is called a \emph{generalised associator}.
\end{defn}

Explicitly, it is the group-like formal power series
$$ \ZZ^i= \sum_{w \in \{e_0,\ldots, e_n\}^{\times}}  \left( \int_{\gamma_i} w \right) w  \quad \in \quad \C \langle \langle e_0,\ldots, e_n\rangle \rangle$$
where the sum is over all words $w$ in $\{e_0,\ldots, e_n\}$, and the integral is the regularised iterated integral (from left to right) of the corresponding word in $\{\omega_0,\ldots, \omega_n\}$. 

When we wish to emphasize the dependence on the variables $e_i$, we shall write
 $\Zim(e_0,\ldots, e_n)$  for  $\Zim$,
 and so on.

\begin{example} \label{ExMZV1} Let $\Sigma = \{0,1\}$ and $k= \Q$. Then
$ \ZZ^{1,\mm}= \ZZ^{\mm}$ where
\begin{equation}\label{eq: motivic Drinfeld associator}\ZZ^{\mm} = \sum_{w \in \{e_0,e_1\}^{\times} } \zetam(w) w  \quad \in \quad  \Pe^{\mm}_{\MT(\Q)} \langle \langle e_0,e_1\rangle \rangle
\end{equation}
is the motivic Drinfeld associator.
Drinfeld's associator is $\ZZ = \mathrm{per}\left( \ZZ^{\mm} \right) \in \R \langle \langle e_0, e_1 \rangle \rangle$. 
\end{example}

 \subsection{Beta quotients}
 
Let $R$ be any commutative unital $\Q$-algebra.

\begin{defn}
Consider the abelianisation map 
$$ F\mapsto \overline{F} \quad :  \quad    R\langle \langle e_0,\ldots, e_n \rangle \rangle  \To R [[s_0,\ldots, s_n]]$$
which sends $e_i$ to $s_i$,  where the $s_i$ are commuting variables. 
We shall call $\overline{F}$ the \emph{abelianisation} of $F$.
\end{defn} 

Computing the abelianisation of a group-like series is easy, as the next lemma shows.

\begin{lem}\label{lem: abelianisationofF}
For a group-like series $F$, we have: 
$$\overline{F} = \exp \left(  F(e_0) s_0 + \cdots +F(e_n) s_n \right)= \prod_{k=0}^n \exp (F(e_k) s_k)\ .  $$
\end{lem}

\begin{proof}
Since $F$ is group-like, 
$F= \exp (\log F))$, where $\log(F)$ is  a Lie series of the form 
$$\log (F) = \sum_{k=0}^n F(e_k) e_k   +  \hbox{commutators}$$
The exponential and logarithm are  taken with respect to the concatenation product, and  commute with the abelianisation map.  Since abelianisation sends all commutators to zero we have $\overline{\log F}=\sum_{k=0}^nF(e_k)s_k$ and the result follows. 
\end{proof}

\begin{defn}
For any series $F \in R\langle \langle e_0,\ldots, e_n \rangle \rangle $ let us write 
\begin{equation} \label{Fsubscripts} F= F_{\varnothing} +  F_0 e_0 + \cdots + F_n e_n
\end{equation}
where $F_{\varnothing} \in R$ denotes the coefficient of the empty word (constant coefficient) and $F_{j}$ is obtained from $F$ by deleting the last letter from all words ending in $e_j$. We call the abelianisation $\overline{F_j}$ the $j^{th}$ \emph{beta-quotient} of $F$.
\end{defn} 

\begin{rem} The $\overline{F_j}$ are very closely related to the image of $F$ in what is known as the metabelian quotient.
\end{rem}

For any two series $A, B \in R\langle \langle e_0,\ldots, e_n \rangle \rangle $  their product satisfies
\begin{equation} \label{SubDerivation}  \left( A B \right)_j =  A  \left( B_j\right)  + A_j  B_{\varnothing}\ .
\end{equation} 

We verify that,  if $F$ is invertible, then 
\begin{equation} \label{SubInversion}
\overline{(F^{-1})_j} =  - \frac{1}{F_{\varnothing}} \frac{ \overline{F_j}  }{ \overline{F} } 
\end{equation}

This follows from applying \eqref{SubDerivation} to  $F F^{-1} = 1$ which  gives $F (F^{-1})_j + F_j F^{-1}_{\varnothing}=0$, and then applying the abelianisation map. 
All series $F$ that we shall consider are group-like (for the continuous coproduct on formal power series for which all letters $e_k$ are primitive) and therefore  have constant term $F_{\varnothing}= 1$.  

Recall that $F(w)$ denotes (a linear combination of) coefficients of words $w$ in $F$. 

 \begin{lem} \label{lem: abelianisationofFbeta}  For any series $F\in R\langle\langle e_0,\ldots,e_n\rangle\rangle$ we have:
$$ \overline{F_j} = \sum_{m_0, \ldots, m_n \geq 0} F\left( \big[e_0^{\sha m_0} \sha  \cdots \sha e_n^{\sha m_n}  \big]e_j\right) \frac{s_0^{m_0}}{m_0!} \cdots \frac{s_n^{m_n}}{m_n!} \ ,  $$ 
where  $[ w ] e_j$ denotes the right-concatenation of $e_j$ to any linear combination $w$ of words in the letters $e_0,\ldots, e_n$.   The previous expression can also be written 
$$\overline{F_j} = \sum_{m_0,\ldots, m_n \geq 0}   F \left( \left[e_{0}^{m_0} \sha \cdots \sha  e_n^{m_n} \right] e_j \right) s_0^{m_0} \cdots s_n^{m_n}   \ .$$
\end{lem} 

\begin{proof}
Notice that 
$$\overline{F_j} = \sum_w \overline{w}  \, F(w e_j) = \sum_{m_0,\ldots, m_n\geq 0 }   s_0^{m_0} \cdots s_n^{m_n} \left( \sum_{\overline{w} =  s_0^{m_0} \cdots s_n^{m_n} } F(w e_j)\right)$$
and substitute in the expression:
$$\left( \sum_{\overline{w} = s^{m_0}_{0} \cdots s_{n}^{m_n}} w \right)  = \frac{e^{\sha m_0}_{0}}{m_0!} \sha \cdots \sha \frac{e^{\sha m_n}_{n}}{m_n!} =  e^{m_0}_{0} \sha \cdots \sha e^{m_n}_{n} \ . $$
\end{proof} 

\subsection{Beta quotients of generalised associators}

    \begin{defn}\label{defi: FL}
    We define the following $n\times n$ matrices of power series:
    $$(FL_\Sigma^\mm)_{ij} = \mathbf{1}_{i=j}\overline{\mathcal{Z}^{i,\mm}} - s_j\overline{\mathcal{Z}^{i,\mm}_j} \;\;\; \in \; \Pe^\mm[[s_0,\ldots,s_n]] .$$
    $$(FL_\Sigma)_{ij} = \mathbf{1}_{i=j}\overline{\mathcal{Z}^{i}} - s_j\overline{\mathcal{Z}^{i}_j} \;\;\; \in \;\mathbb{C}[[s_0,\ldots,s_n]] .$$
    \end{defn}

    We clearly have $\mathrm{per}(FL_\Sigma^\mm)=FL_\Sigma$, where $\mathrm{per}:\Pe^\mm\to \mathbb{C}$ denotes the period map of $\MT(k)$ applied coefficientwise. Our next goal (Theorem \ref{thm: FL equals L} below) is to prove that $FL_\Sigma$ equals the matrix of Taylor series of the Lauricella functions $L_\Sigma$ \eqref{IntroLSigma}. (This justifies the notation $FL_\Sigma$ where the letter $F$ stands for ``formal''.) We now compute the abelianisation and beta quotients of the generalised associators $\mathcal{Z}^i$.\medskip

    We first need a lemma that clarifies the role of the choice of tangential basepoints. For a point $x=\gamma_i(t)$, for $t\in (0,1)$, we let $\gamma_i^x$ denote the restriction of $\gamma_i$ to the interval $(0,t)$. By abuse of notation, we manipulate functions and forms on $\mathbb{C}\setminus \Sigma$ when we actually mean their pullbacks to $(0,1)$ via the path $\gamma_i$.

    \begin{lem}\label{lem: iterated integrals length one}
    For $0\leq k\leq n$ we have the equalities, for $x=\gamma_i(t)$ with $t\in (0,1)$:
    $$\int_{\gamma_i^x}\omega_k = \begin{cases} \log(x) & \mbox{if } k=0\ ;  \\ \log(1-x\sigma_k^{-1}) & \mbox{if } 1\leq k\leq n\ . \end{cases}$$
    Furthermore, for $x=\sigma_i$, i.e, for $t=1$, we have:
    $$\int_{\gamma_i}\omega_k = \begin{cases} \log(\sigma_i) & \mbox{if } k=0\ ;  \\ \log(1-\sigma_i\sigma_k^{-1}) & \mbox{if } 1\leq k\neq i\leq n\ ; \\ 0 & \mbox{if } k=i\ .  \end{cases}$$
    \end{lem}

    \begin{proof}
    For the first claim, the $1\leq k\leq n$ case is clear since the differential form $\omega_j$ does not have a singularity at $0$. For $j=0$, we compute, by definition
    $$\int_{\gamma_i^x}\omega_0 = \mathrm{Reg}_{\varepsilon \to 0} \int_\varepsilon^t\gamma_i^*\omega_0=\mathrm{Reg}_{\varepsilon\to 0}(\log(x)-\log(\gamma_i(\varepsilon))) \ ,$$
    where applying $\mathrm{Reg}_{\varepsilon\to 0}$ amounts to formally setting $\log(\varepsilon)=0$ and $\varepsilon=0$ in the logarithmic asymptotic development \cite[Definition 3.237]{burgosfresanbook}. Since $\gamma_i$ extends to a smooth function on $[0,1]$ whose derivative at $0$ is $1$, we have $\gamma_i(\varepsilon)=\varepsilon + O(\varepsilon^2)$ and thus $\log(\gamma_i(\varepsilon))=\log(\varepsilon)+O(\varepsilon)$. We thus have $\mathrm{Reg}_{\varepsilon \to 0}(\log(\gamma_i(\varepsilon)))=0$ and the claim follows. 

    For the second claim, only the $k=i$ case requires a comment since in the other cases the form $\omega_j$ does not have a singularity at $\sigma_i$ and we can simply pass to the limit $x\to \sigma_i$ in the first statement. We will use the admissibility condition (Definition \ref{defi: gamma i admissible}). Since $\omega_i$ only has poles at $\sigma_i$ and $\infty$, its iterated integral along $\gamma_i$ only depends on the image of $\gamma_i$ in $\pi_1^{\mathrm{top}}(\mathbb{C}\setminus \{\sigma_i\}, 0, -t_i)$. By the admissibility condition this image is the class of $\widetilde{\gamma_i}(t)=\sigma_it$ and we can compute:
    $$\int_{\gamma_i}\omega_i = \int_{\widetilde{\gamma_i}}\omega_i= \mathrm{Reg}_{\varepsilon \to 0} \int_{\varepsilon}^{1-\varepsilon}\widetilde{\gamma_i}^*\omega_i = \mathrm{Reg}_{\varepsilon\to 0} \int_{\varepsilon}^{1-\varepsilon} \frac{dt}{t-1} = \mathrm{Reg}_{\varepsilon \to 0} \log\left(\frac{\varepsilon}{1-\varepsilon}\right) = 0\ .$$
    The claim follows.
    \end{proof}
    
    \begin{rem}\label{rem: admissible vanishing}
    From the proof of Lemma \ref{lem: iterated integrals length one} one sees that the admissibility condition on $\gamma_i$ (Definition \ref{defi: gamma i admissible}) is equivalent to the vanishing of the regularised iterated integral $\int_{\gamma_i}\omega_i$. In general that iterated integral could be any integer multiple of $2\pi i$.
    \end{rem}

    \begin{prop} \label{prop: periodAb}  The abelianisation of the generalised associator $\mathcal{Z}^i$ is:
    \begin{equation} \label{Ziperiod} \overline{\ZZ^i}= \sigma_i^{s_0}   \prod_{k\neq i }  (1 - \sigma_i\sigma_k^{-1})^{s_k}  \ . 
    \end{equation}
    For every $0\leq j\leq n$, the $j$-th beta quotient of $\mathcal{Z}^i$ is:
    \begin{equation}  \label{ZZijperiod} \overline{\ZZ^i_j} = \int_{\gamma_i}  x^{s_0} \prod _{k=1}^n   \left(1- x\, \sigma^{-1}_k\right)^{s_k}    \frac{dx}{x-\sigma_j}    \  .
    \end{equation} 
    These expressions are formal power series in the $s_i$ obtained by expanding the exponentials as power series and interpreting the various  logarithms  which appear  as coefficients as regularised iterated integrals via Lemma \ref{lem: iterated integrals length one}.
    \end{prop}
    
    \begin{proof}
    Since $\mathcal{Z}^i$ is a group-like series, the first claim follows from Lemma \ref{lem: abelianisationofF} and the computations of $\mathcal{Z}^i(e_j)$ performed in the second part of Lemma \ref{lem: iterated integrals length one}. The second claim follows from Lemma \ref{lem: abelianisationofFbeta} and the equality
    $$\int_{\gamma_i}[e_0^{\sha m_0}\sha\cdots\sha e_n^{\sha m_n}]e_j = \int_{\gamma_i}\Big(\int_{\gamma_i^x}\omega_0\Big)^{m_0}\cdots \Big(\int_{\gamma_i^x}\omega_n\Big)^{m_n} \omega_j\ ,$$
    which by the first part of Lemma \ref{lem: iterated integrals length one} equals
    $$\int_{\gamma_i}\log^{m_0}(x)\prod_{k=1}^n\log^{m_k}(1-x\sigma_k^{-1})\frac{dx}{x-\sigma_j}\ \cdot$$
    The claim follows.
    \end{proof}
 
    Equation   \eqref{ZZijperiod} in the case $\Sigma=\{0,1\}$ reduces to  Drinfeld's computation of the metabelian quotient of his associator in terms of the  usual beta function \cite{Drinfeld}. See also \cite{EnriquezGamma, LiGamma} for further developments.   
 
    The following lemma will be useful. 
    \begin{lem} \label{lem: integralzsigmavanishes}  For all $m\geq 0$, 
    \begin{equation}  \label{integralzsigmavanishes} \int_{\gamma_i}  \log^{m}(1-x \sigma_i^{-1}) \frac{dx}{x-\sigma_i} = 0\ . 
    \end{equation}
    \end{lem} 
    \begin{proof}
    The integral is proportional to the $(m+1)$-fold iterated integral of $\frac{dx}{x-\sigma_i}$ along $\gamma_i$, which,  by the shuffle product formula,  is in turn proportional to the $(m+1)$th power of  the integral  of $\frac{dx}{x-\sigma_i}$ along $\gamma_i$, which vanishes by Lemma \ref{lem: iterated integrals length one}. 
    \end{proof} 
 
\subsection{Comparing \texorpdfstring{$L_\Sigma$}{LSigma} and \texorpdfstring{$FL_\Sigma$}{FLSigma}}
    
    Up until now we have considered two types of integrals: convergent line integrals of smooth $1$-forms along a smooth (locally finite) path $\delta_i$ from $0$ to $\sigma_i$, and regularised iterated integrals of words in the $\omega_j$'s along a path $\gamma_i$ from $0$ to $\sigma_i$ with prescribed tangent directions at $0$ and $\sigma_i$. In order to compare these two different objects we have to impose some compatibility between $\gamma_i$ and $\delta_i$. Let us first fix parameters $\underline{s}=(s_0,s_1,\ldots,s_n)$. There is a natural map
    $$h_i:\pi_1^{\mathrm{top}}(X_\Sigma(\mathbb{C}),t_0,-t_i)\longrightarrow H_1^{\lf}(X_\Sigma(\mathbb{C}),\mathcal{L}_{\underline{s}}^\vee)$$
    defined in the following way. For a class in $\pi_1^{\mathrm{top}}(X_\Sigma(\mathbb{C}),t_0,-t_i)$ represented by a continuous map $\gamma:(0,1)\to \mathbb{C}\setminus \Sigma$, its image under $h_i$ has $\gamma$ as its underlying (locally finite) path, together with the section of $\mathcal{L}_{\underline{s}}^\vee$ given by
    $$x^{s_0}\prod_{k=1}^n(1-x\sigma_k^{-1})^{s_k}=\prod_{k=0}^n\exp\Big(s_k\int_{\gamma^x}\omega_k\Big)\ ,$$
    which makes sense thanks to Lemma \ref{lem: iterated integrals length one}. One checks that this is well-defined.
    
    Let $\delta_i:(0,1)\to \mathbb{C}\setminus \Sigma$ be a smooth map that can be extended to a smooth map $\overline{\delta_i}:[0,1]\to \mathbb{C}$, and which does not wind infinitely around $0$ or $\sigma_i$ (Remark \ref{rem: choice delta i winding}). Let us fix a determination of $\log(\sigma_i)$ for every $1\leq i\leq n$. In Remark \ref{rem: global class specialise} we defined out of this data a class $\delta_{i,\underline{s}}\in H_1^{\lf}(X_\Sigma(\mathbb{C}),\mathcal{L}_{\underline{s}}^\vee)$ for every parameter $\underline{s}=(s_0,\ldots,s_n)\in\mathbb{C}^{n+1}$. Let $\gamma_i$ be a class in $\pi_1^{\mathrm{top}}(X_\Sigma(\mathbb{C}),t_0,-t_i)$.
    
    \begin{defn} 
    We say that $\delta_i$ and $\gamma_i$ are compatible if for every $\underline{s}\in\mathbb{C}^{n+1}$ we have
    $$ h_i(\gamma_i) = \delta_{i,\underline{s}}\ .$$
    \end{defn}
    
    Note that this implies that the determinations of $\log(\sigma_i)$ are given by the regularised iterated integrals $\log(\sigma_i)=\int_{\gamma_i}\omega_0$.\medskip

    For all $1\leq j\leq n$ we define:
    \begin{equation} \label{OmegajDef} \Omega_j =   x^{s_0} \prod_{k=1}^n (1-x \sigma_{k}^{-1})^{s_k}\, \frac{dx}{x-\sigma_j} \ ,
    \end{equation} 
    which we interpret as a $1$-form on $\delta_i$ as in Remark \ref{rem: convention branches}. Let $\Omega^{\ren_i}_j$ denote its renormalised versions (Definition \ref{definition: renormalisedforms})  with respect to  $\{0, \sigma_i\}$, given by
    $$\Omega_j^{\ren,i} = \Big(x^{s_0}\prod_{k\neq i}(1-x\sigma_k^{-1})^{s_k} - \mathbf{1}_{i=j}\, \sigma_i^{s_0}\prod_{k\neq i}(1-\sigma_i\sigma_k^{-1})^{s_k}\Big)(1-x\sigma_i^{-1})^{s_i}\frac{dx}{x-\sigma_j}\ \cdot$$
    According to Proposition \ref{prop: taylor series LSigma}, the integral $\int_{\delta_i}\Omega_j^{\ren_i}$ defines a holomorphic function of the parameters $\underline{s}$ around $\underline{s}=0$ and thus has a Taylor expansion. In the next proposition we identify this Taylor expansion with the $j$-th beta quotient $\overline{\mathcal{Z}^i_j}$ of the generalised associator, defined via regularised iterated integrals along $\gamma_i$.

 \begin{prop} \label{lem: ConvergentZi}
For all $1\leq j \leq n$, the Taylor series at $\underline{s}=0$ of the integral $\int_{\delta_i} \Omega_j^{\mathrm{ren_i}}$ is $\overline{\mathcal{Z}^i_j}$.
\end{prop}

 Before proving the proposition, we need a technical lemma that allows us to exchange summation and integration.

\begin{lem}\label{lem: fubini}
Let us write $\Omega_j^{\ren_i}$ as a formal power series:
 $$\Omega^{\ren_i}_j=  \sum_{m_0,\cdots, m_n\geq 0} \Omega^{\ren_i}_j(\underline{m})\,  \frac{s^{m_0}_0\cdots s^{m_n}_n}{m_0! \cdots m_n!}\ \cdot$$
 Then we have the equality of formal power series:
 $$\int_{\delta_i}\Omega_j^{\ren_i} = \sum_{m_0,\ldots,m_n\geq 0} \int_{\delta_i}\Omega^{\ren_i}_j(\underline{m}) \, \frac{s^{m_0}_0\cdots s^{m_n}_n}{m_0! \cdots m_n!}\ \cdot$$
\end{lem}

\begin{proof}
Let us write $\delta_i^*\Omega_j^{\ren_i}(\underline{m})=f_{\underline{m}}(t)dt$. By Fubini, it is enough to prove that the multiple series
\begin{equation}\label{eq: power series in proof Fubini}
\sum_{m_0,\ldots,m_n\geq 0}\int_0^1|f_{\underline{m}}(t)|dt\, \frac{s^{m_0}_0\cdots s^{m_n}_n}{m_0! \cdots m_n!}
\end{equation}
is absolutely convergent for $|\underline{s}|$ small enough, and therefore we need to estimate the integral of $|f_{\underline{m}}(t)|$. We do so in the case $j=i$; the case $j\neq i$ is even simpler. For simplicity, we only treat the case $n=1$. The general case is similar and is left to the reader. We write $\sigma=\sigma_1$ and $\delta=\delta_1$.
For indices $m_0,m_1\geq 0$ we have, according to Definition \ref{definition: renormalisedforms}:
$$f_{m_0,m_1}(t) =\frac{\log^{m_0}(\delta(t))-\log^{m_0}(\sigma)}{\delta(t)-\sigma} \log^{m_1}(1-\delta(t)\sigma^{-1}) \,\delta'(t)\ .$$
Note that $\delta'(t)$ is bounded for $t\in (0,1)$. To prove convergence, we can clearly assume that there exist small closed discs $D_0$ and $D_\sigma$ around $0$ and $\sigma$ respectively, and constants $0<\alpha<\beta<1$, such that $\delta(t)\notin D_\sigma$ for all $t\in (0,\beta)$, and $\delta(t)\notin D_0$ for all $t\in (\alpha,1)$. We treat separately the cases $t\in (0,\beta)$ and $t\in (\alpha,1)$, and assume that $m_0,m_1\geq 1$.
\begin{enumerate}[--]
    \item For $t\in (0,\beta)$, $\log(1-\delta(t)\sigma^{-1})$ is bounded 
    and $\delta(t)-\sigma$ is 
    bounded below in absolute value, which gives an estimate
    $$|f_{m_0,m_1}(t)| < \left( |\log^{m_0}(\delta(t))| + |\log^{m_0}(\sigma)|\right) A^{m_1}$$
    for some positive constant $A$. Now since by assumption (Remark \ref{rem: choice delta i winding}) the argument of $\delta(t)$ is bounded as $t$ approaches zero, and since $|\log|\delta(t)||$ tends to infinity when $t$ goes to zero, we have $|\log(\delta(t))|<B|\log|\delta(t)||$ for some positive constant $B$, and therefore since $\delta$ is smooth at $0$, we have $|\log(\delta(t))|<B'|\log(t)|$ for all $t\in (0,1/2)$, for some positive constant $B'$. We deduce the estimate $|\log^{m_0}(\delta(t))| + |\log^{m_0}(\sigma)|< (B'')^{m_0}|\log(t)|^{m_0}$ for all $t\in (0,\beta)$, for some positive constant $B''$. Therefore,
    $$|f_{m_0,m_1}(t)| < C^{m_0+m_1}|\log(t)|^{m_0}$$
    for all $t\in (0,\beta)$, for some positive constant $C$.
    \item For $t\in (\alpha,1)$, the quotient $\frac{\log^{m_0}(\delta(t))-\log^{m_0}(\sigma)}{\delta(t)-\sigma}$ is bounded by $D^{m_0}$ for some positive constant $D$, and by the same argument as before (using the assumption of Remark \ref{rem: choice delta i winding}) we have an estimate $|\log^{m_1}(1-\delta(t)\sigma^{-1})|<E^{m_1}|\log(1-t)|^{m_1}$ for all $t\in (\alpha,1)$, for some positive constant $E$. Therefore,
    $$|f_{m_0,m_1}(t)|<F^{m_0+m_1}|\log(1-t)|^{m_1}$$
    for all $t\in (\alpha,1)$, for some positive constant $F$.
\end{enumerate}
Putting those two contributions together we get
$$\int_0^1|f_{m_0,m_1}(t)|dt < G^{m_0+m_1}\int_0^1(|\log(t)|^{m_0} + |\log(t)|^{m_1}) dt$$
for some positive constant $G$. Since $\int_0^1|\log(t)|^mdt=m!$ for all $m$, we therefore get the estimate
$$\int_0^1|f_{m_0,m_1}(t)|dt < G^{m_0+m_1}(m_0!+m_1!)$$
for some positive constant $G$. Therefore the series \eqref{eq: power series in proof Fubini} converges for $|s_0|,|s_1|<G^{-1}$, and the claim follows.
\end{proof}

We now give two proofs of Proposition \ref{lem: ConvergentZi}, since they are instructive.

\begin{proof}[First proof] 
 We  consider only the case $\overline{\ZZ^i_i}$ since the argument for $\overline{\ZZ^i_j}$ with $j \neq i$ is even simpler. 
According to Definition \ref{definition: renormalisedforms}, we have, with the notation of Lemma \ref{lem: fubini}:
\begin{multline} \Omega^{\ren_i}_i(\underline{m}) \qquad  = \qquad  \, \log^{m_i}(1-x \sigma_i^{-1}) \frac{dx}{x-\sigma_i}  \\
\times \left(\log^{m_0}(x) \prod_{1\leq k \neq i}  \log^{m_k}(1-x\sigma_k^{-1}) -  \log^{m_0}(\sigma_i) \prod_{1\leq k \neq i}  \log^{m_k}(1-\sigma_i \sigma_k^{-1})\right)  \ .\nonumber 
\end{multline}
According to Lemma \ref{lem: fubini}, the Taylor series of $\int_{\delta_i}\Omega_i^{\ren_i}$ has coefficients $\int_{\delta_i}\Omega_i^{\ren_i}(\underline{m})$, and we claim that they equal
\begin{equation} \label{inproofintomegarengamma} \int_{\delta}  \Omega_i^{\ren_i} (\underline{m})  \  = \  \int_{\gamma_i}   \Omega_i^{\ren_i} (\underline{m})  
\end{equation}
where the  integral on the left is an ordinary,  convergent integral, and the one on the right is regularised along the path $\gamma_i$ between tangential basepoints (\S \ref{subsect: NotGeometric}). 
To see this, use the fact that  regularisation  with respect to the tangential basepoint $-t_i$  is equivalent to  taking a primitive of   $ \Omega_i^{\ren_i} (\underline{m})$ in the  ring  $\C[[x-\sigma_i]][\log(x-\sigma_i)]$, and formally setting all 
$\log(x-\sigma_i)$ terms to zero, before in turn setting $x$ to $\sigma_i$. Since the term in brackets in  the above expression for    $ \Omega_i^{\ren_i} (\underline{m})$  vanishes at $x=\sigma_i$, it actually has a primitive in the subspace $(x-\sigma_i)\C[[x-\sigma_i]][\log(x-\sigma_i)]$, and one can simply take its limit as $x\rightarrow \sigma_i$, which is the procedure for computing an ordinary integral (without tangential basepoint regularisation). A simpler   argument applies at $x=0$ and proves \eqref{inproofintomegarengamma}.

The formula  for $\overline{\ZZ^i_i}$ follows by applying   Lemma \ref{lem: integralzsigmavanishes}
and  implies that
$$  \int_{\gamma_i}  \Omega_i^{\ren_i} (\underline{m})  \  = \int_{\gamma_i}  
\log^{m_0}(x) \prod^n_{k=1}  \log^{m_k}(1-x\sigma_k^{-1})  \frac{dx}{x-\sigma_i} \ . $$
This is precisely the coefficient of $\frac{s^{m_0}_0}{m_0!} \cdots \frac{s^{m_n}_n}{m_n!}$ in the Taylor expansion of \eqref{ZZijperiod}.
\end{proof} 

\begin{proof}[Second proof]
 For   $\gamma$ a path between (tangential) base-points $x,y \in X_{\Sigma}$, and $\omega$ a closed formal $1$-form taking values in  the Lie algebra of the ring of formal non-commutative power series $\mathbb{C}\langle\langle e_0,\ldots,e_n\rangle\rangle$, consider the formal power series defined by iterated integration:
$$I_{\gamma}(\omega) =1 + \int_{\gamma} \omega + \int_{\gamma} \omega \omega + \cdots \ .$$
It is known as  the transport  of (the connection associated to) $\omega$ along $\gamma$, and  satisfies the composition of paths formula $I_{\gamma\gamma'}(\omega)=I_\gamma(\omega)I_{\gamma'}(\omega)$.  
By applying the prescription for   computing iterated integrals with respect to tangential basepoints, we find that 
the transport of the formal one-form 
$\omega_{\Sigma} =e_0 \omega_0+ \cdots + e_n \omega_n$ along  the path $\gamma_i$  defined in \S \ref{subsect: NotGeometric} 
equals:
\begin{eqnarray} \ZZ^i = I_{\gamma_i}(\omega_\Sigma)  & =  & I_{\gamma_i^x}(\omega_{\Sigma}) I_{\nu_x}(\omega_{\Sigma})  \nonumber \\
 & = & \lim_{x\rightarrow \sigma_i} \left(   I_{\gamma_i^x}(\omega_{\Sigma}) I_{\nu_x}(\omega_{\Sigma})\right)   \nonumber \\
 & = & \lim_{x\rightarrow \sigma_i} \left(   I_{\gamma_i^x}(\omega_{\Sigma}) I_{\nu_x}(e_i \omega_i)\right)  \nonumber
\end{eqnarray} 
where  $x =\gamma_i(t)$ for some $0<t<1$;  $\gamma_i^x$ is the restriction of $\gamma_i$ to $[0,t]$; 
and $\nu_x$ is a path from $x$ to $-t_i$. The first equation is simply   the composition of  paths formula.  Since the left-hand side does not depend on the choice of point $x$, we may take a limit as $x \rightarrow \sigma_i$, which implies  the second equation. In the third equation, we view the path $\nu_x$   as a path  in the tangent space at $\sigma_i$, which we  identify with $\Pro^1\backslash \{\sigma_i,\infty\}$.  The form $e_i \omega_i$ is the localisation of $\omega_{\Sigma}$ to the punctured tangent space  and captures the divergent iterated integrals terminating in the letter $e_i$.

   It follows from \eqref{SubDerivation} that 
\begin{equation} \label{1stproofzzij} \overline{\ZZ^i_j} = \lim_{x \rightarrow \sigma_i} \left(\overline{I_{\gamma_i^x}(\omega_{\Sigma})} \,  \overline{I_{\nu_x}(e_i \omega_i)_j}    +   \overline{I_{\gamma_i^x}(\omega_{\Sigma})}_j \right) \ .
\end{equation}
The  tangential integral $ I_{\nu_x}(e_i \omega_i)$ is simply an exponential:
\begin{equation} \label{inprooftangentialexp} I_{\nu_x}(e_i \omega_i) = \exp ( -e_i \log (1-x \sigma^{-1}_i)) \  \end{equation}
and its abelianisation is therefore obtained by replacing $e_i$ with $s_i$ in the previous expression. 
Equation  \eqref{inprooftangentialexp} follows, for example,  from 
$$\int_x^{-t_i} \frac{dz}{z-\sigma_i} = \left( \int_x^{t_0} + \int_{t_0}^{-t_i} \right) \frac{dz}{z-\sigma_i} \ \overset{ \eqref{integralzsigmavanishes}}{=} \  \int_x^0  \frac{dz}{z-\sigma_i}+ 0 = - \log (1- x\sigma_i^{-1}) \ .$$
From equation \eqref{inprooftangentialexp}, the expression $ \overline{ I_{\nu_x}(e_i \omega_i)}_j$ vanishes if $j\neq i$, but equals $$\overline{ I_{\nu_x}(e_i \omega_i)_i}= \frac{1}{s_i}\left((1-x \sigma_i^{-1})^{-s_i} -1 \right) $$
otherwise. Thus if $j\neq i$,  the term   $\overline{I_{\gamma_i^x}(\omega_{\Sigma})}\overline{I_{\nu_x}(e_i \omega_i)_j}$ in  \eqref{1stproofzzij}  vanishes  and we find that
$$ \overline{\ZZ^i_j} = \lim_{x \rightarrow \sigma_i} \left(   \overline{I_{\gamma_i^x}(\omega_{\Sigma})}_j \right)= \lim_{x \rightarrow \sigma_i} \left(    \int_0^x    x^{s_0} \prod _{k=1}^n   \left(1- x\, \sigma^{-1}_k\right)^{s_k}    \frac{dx}{x-\sigma_j}         \right)  \ , $$
using the version of \eqref{ZZijperiod}  with the upper range of integration replaced with the point $x$ (which follows from the  computations in  the second paragraph of the proof of Proposition \ref{prop: periodAb} -- one needs only check that one can replace $\gamma_i^x$ with an ordinary path from $0$ to $x$, i.e., that the tangential component of $\gamma_i^x$ at the origin plays no role since the integral is convergent there).
This proves the  formula for $j\neq i$, thanks to Lemma \ref{lem: fubini}.  In the case $j=i$,   a  version of \eqref{Ziperiod}  with upper range of integration $x$ implies that  
$$\overline{I_{\gamma_i^x}(\omega_{\Sigma})} =  x^{s_0} \prod _{k=1}^n   \left(1- x\, \sigma^{-1}_k\right)^{s_k} \ .$$
Substituting this into   
 \eqref{1stproofzzij}  gives
$$\overline{\ZZ^i_i} = \lim_{x \rightarrow \sigma_i} \Bigg(
x^{s_0} \prod _{k=1}^n   \left(1- x\, \sigma^{-1}_k\right)^{s_k} \times \frac{1}{s_i}\left((1-x \sigma_i^{-1})^{-s_i}-1 \right) 
+
 \int_0^x    x^{s_0} \prod _{k=1}^n   \left(1- x\, \sigma^{-1}_k\right)^{s_k}    \frac{dx}{x-\sigma_i}         \Bigg) \ .
 $$
Using the identity $(1-x \sigma_i^{-1})^{s_i} \times  \left((1-x \sigma_i^{-1})^{-s_i} -1 \right)   =  1-(1-x \sigma_i^{-1})^{s_i}$,  
the previous expression can be rewritten in the form:
$$\lim_{x \rightarrow \sigma_i} \Bigg( 
\sigma_i^{s_0} \prod _{k\neq i}   (1-\sigma_i \sigma_k^{-1})^{s_k} \times \frac{1}{s_i}\left(1-(1-x \sigma_i^{-1})^{s_i} \right) 
+
 \int_0^x    x^{s_0} \prod _{k=1}^n   \left(1- x\, \sigma^{-1}_k\right)^{s_k}    \frac{dx}{x-\sigma_i}         \Bigg) \ .$$
Finally, substitute in the following identity
\begin{equation} \label{logidentityTBreg}
\frac{1}{s_i}\left(1-(1-x \sigma_i^{-1})^{s_i} \right)  = - \int_0^{x}  (1-x \sigma_i^{-1})^{s_i} \frac{dx}{x-\sigma_i}  \end{equation}
 to deduce the stated formula for $\overline{\ZZ^i_i}$, thanks to Lemma \ref{lem: fubini}.
\end{proof}

    We are now ready to prove part (\emph{i}) of Theorem \ref{introthmperandsv} from the introduction.

    \begin{thm}\label{thm: FL equals L}
    For all $i,j$, $(FL_\Sigma)_{ij}$ is the Taylor series at $\underline{s}=0$ of the Lauricella function $(L_\Sigma)_{ij}$.
    \end{thm}
 
    \begin{proof}
    This follows from comparing the definition of $(FL_\Sigma)_{ij}$ with the expression in Proposition \ref{prop: taylor series LSigma}, and using the expressions for $\overline{\mathcal{Z}^i}$ and 
     $\overline{\mathcal{Z}^i_j}$
    of 
    Propositions \ref{prop: periodAb} and \ref{lem: ConvergentZi}.
    \end{proof}

\section{Generalised single-valued associators and their beta quotients} 

    This section is the single-valued version of the previous one. We introduce the single-valued versions $\mathcal{Z}^{i,\s}$ of the generalised associators, which are non-commutative power series whose coefficients are single-valued versions of iterated integrals on $X_\Sigma$. Through their beta quotients we define a matrix $FL_\Sigma^\s$ of power series which we prove to be the matrix of Taylor series of the single-valued Lauricella functions $L_\Sigma^\s$ \eqref{IntroLSigmasv} (Theorem \ref{thm: FL equals L sv}, which is Theorem \ref{introthmperandsv} (\emph{ii}) from the introduction).

\subsection{Generalised de Rham and single-valued associators}

    We fix an index $1\leq i\leq n$. We start by defining de Rham and single-valued versions of the generalised associators from the previous section.

    \begin{defn}\label{def: generalised de Rham associator}
    We define
    $$  \Zio    \      \in   \      {}_0\Pi_i^\varpi  (\Pe^{\varpi}) \subset  \Pe^{\varpi} \langle \langle e_0,\ldots, e_n\rangle\rangle  $$
to be the canonical element in $\mathrm{Hom}\left( \Or( {}_0\Pi_i^\varpi), \Pe^{\varpi}\right)$ given by  the morphism of schemes
$ G^{\varpi}_{\MT(k)}  \rightarrow {}_0 \Pi_i^\varpi$
 induced by the action  $g\mapsto g .  {}_0\!1_i$ of $G^{\varpi}_{\MT(k)}$ on  the canonical $\varpi$-path ${}_0 1_i$. It is called a \emph{generalised (canonical) de Rham associator}.
    \end{defn}
    
    It is given explicitly by the group-like formal power series
    $$\Zio  = \sum_{w\in \{e_0,\ldots, e_n\}^{\times}} \left[ \Or(\pi_1^{\mathrm{mot}}(X_{\Sigma}, t_0, -t_i)), {}_01_i, w\right]^{\varpi} w$$
    whose coefficients are (canonical) de Rham versions of iterated integrals. Since the empty iterated integral along $\gamma_i$ is 1, 
    it follows that $\Zio$ is the image of $\Zim$ under the coefficient-wise application of the  projection $\pi^{\mm,+}_{\varpi}$, i.e., 
    \begin{equation}\label{eq: Zio dRproj Zim}
    \Zio = \pi^{\mm,+}_{\varpi} \Zim\ .
    \end{equation}
    
    \begin{defn}
    We define
    $$\mathcal{Z}^{i,\s} = \s(\mathcal{Z}^{i,\varpi}) \;\; \in \; \mathbb{C}\langle\langle e_0,\ldots,e_n\rangle\rangle$$
    where $\s:\Pe^\varpi\to \mathbb{C}$ is the single-valued period map of $\MT(k)$ applied coefficientwise. It is called a \emph{generalised single-valued associator}.
    \end{defn}
    
    The coefficients of $\mathcal{Z}^{i,\s}$ are single-valued versions of iterated integrals.
    
    \begin{rem}
    As should be clear from the definitions, the power series $\mathcal{Z}^{i,\varpi}$ and $\mathcal{Z}^{i,\s}$ do not depend on the choice of a (class of a) path $\gamma_i$ as in the previous section. 
    \end{rem}
    
    \begin{example} \label{ExMZV1 sv} Let $\Sigma = \{0,1\}$ and $k= \Q$. Then $\varpi= \omega_{\mathrm{dR}}$ and $ \ZZ^{1,\varpi}= \ZZ^{\dR}$ where
    $$\ZZ^{\dR} = \sum_{w \in \{e_0,e_1\}^{\times} } \zeta^\dR(w) w  \quad \in \quad  \Pe^{\dR}_{\MT(\Q)} \langle \langle e_0,e_1\rangle \rangle $$  
    is the de Rham Drinfeld associator. It is obtained from the motivic Drinfeld associator \eqref{eq: motivic Drinfeld associator} by replacing every motivic multiple zeta value $\zetam$ with its de Rham version $\zeta^\dR$. Its image under the single-valued period map is the Deligne associator
    $$\mathcal{Z}^\s = \sum_{w\in\{e_0,e_1\}^\times}\zeta^\s(w)w \quad \in \quad \mathbb{R}\langle\langle e_0,e_1\rangle\rangle\ ,$$
    whose coefficients are single-valued multiple zeta values \cite{brownSVMZV}.
    \end{example} 
    
    The following definition is parallel to Definition \ref{defi: FL}.

    \begin{defn}
    We define the following $n\times n$ matrices of power series:
    $$(FL_\Sigma^\varpi)_{ij} = \mathbf{1}_{i=j}\overline{\mathcal{Z}^{i,\varpi}} - s_j\overline{\mathcal{Z}^{i,\varpi}_j} \;\;\; \in \; \Pe^\varpi[[s_0,\ldots,s_n]] .$$
    $$(FL_\Sigma^\s)_{ij} = \mathbf{1}_{i=j}\overline{\mathcal{Z}^{i,\s}} - s_j\overline{\mathcal{Z}^{i,\s}_j} \;\;\; \in \;\mathbb{C}[[s_0,\ldots,s_n]] .$$
    \end{defn}

    We clearly have 
    $$\s(FL_\Sigma^\varpi)=FL_\Sigma^\s$$ 
    where $\s:\Pe^\varpi\to \mathbb{C}$ denotes the single-valued period map of $\MT(k)$ applied coefficientwise. By \eqref{eq: Zio dRproj Zim} we also have
    $$FL_\Sigma^\varpi = \pi^{\mm,+}_\varpi FL^\mm_\varpi $$
    where $\pi^{\mm,+}_\varpi:\Pe^{\mm,+}\to \Pe^\dR$ is the de Rham projection applied coefficientwise.\medskip
    
    Our next goal (Theorem \ref{thm: FL equals L sv} below) is to prove that $FL_\Sigma^\s$ equals the matrix of Taylor series of the single-valued Lauricella functions $L_\Sigma^\s$ \eqref{IntroLSigmasv}. To this end, we first compute the abelianisation and beta quotients of the generalised single-valued associators $\mathcal{Z}^{i,\s}$. Our techniques can be used more generally to give integral formulae for the single-valued periods of motivic torsors of paths between tangential basepoints.
   
    \begin{prop}\label{prop: abelianisation Z i s}
    The abelianisation of the generalised single-valued associator $\mathcal{Z}^{i,\s}$ satisfies 
    $$\overline{\mathcal{Z}^{i,\s}} = \left|\sigma_i \right|^{2s_0} \prod_{k\neq i}  \left|1-\sigma_i \sigma_k^{-1}\right|^{2s_k}    \ .   $$
    \end{prop}
    
    \begin{proof}
    This follows from Lemma \ref{lem: abelianisationofF} and the following claim about the length one coefficients of $\mathcal{Z}^{i,\s}$:
    $$\mathcal{Z}^{i,\s}(e_k)= \begin{cases} \log|\sigma_i|^2 & \mbox{ for } k=0\ ;\\ \log|1-\sigma_i\sigma_k^{-1}|^2 & \mbox{ for } 1\leq k\neq i\leq n\ ;\\ 0 & \mbox{ for } k=i\ .  \end{cases}$$
    Let us prove this claim.  By definition and since $\omega_k$ only has poles at $\sigma_k$ and $\infty$, we have 
    $$\mathcal{Z}^{i,\varpi}(e_k)= \left[W_2\mathcal{O}(\pi_1^{\mathrm{mot}}(\mathbb{A}^1_k\setminus \{\sigma_k\},t_0,-t_i)),{}_01_i,[\omega_k]\right]^\varpi\ .$$
    The object $W_2\mathcal{O}(\pi_1^{\mathrm{mot}}(\mathbb{A}^1_k\setminus \{\sigma_k\},t_0,-t_i))$ denotes the weight $\leq 2$ (or length $\leq 1$) subobject; it has rank $2$, with de Rham basis $(1,[\omega_k])$ and Betti basis $([\gamma],[\gamma']-[\gamma])$ for some paths $\gamma$, $\gamma'$ from $t_0$ to $-t_i$ such that the closed path $\gamma'\gamma^{-1}$ is homologous to a small positively oriented loop around $\sigma_k$. (Note that for $k\neq 0$, $t_0$ denotes the usual basepoint $0$, and for $k\neq i$, $-t_i$ denotes the usual basepoint $\sigma_i$.) We note that by definition, ${}_01_i(1)=1$ and ${}_01_i([\omega_k])=0$. The period matrix in these bases is 
    $$P = \left(\begin{matrix} 1 & \int_\gamma\omega_k\\ 0 & 2\pi i\end{matrix}\right)\ .$$
    By the definition of the single-valued period homomorphism \cite[Definition 2.5]{BD1}, $\mathcal{Z}^{i,\s}(e_k)=\s(\mathcal{Z}^{i,\varpi}(e_k))$ is the top right coefficient of the single-valued period matrix $\overline{P}^{-1}P$, where $\overline{P}$ denotes the complex conjugate matrix. An easy computation shows that this equals $\mathcal{Z}^{i,\s}(e_k)=2\,\Real(\int_\gamma\omega_k)$ and the claim follows from Lemma \ref{lem: iterated integrals length one}.
    \end{proof}
    
\subsection{Beta quotients of generalised single-valued associators}

 We start with a version of beta quotients which have finite, rather than tangential, basepoints. We extend Definition \ref{def: generalised de Rham associator} and let$$I^{\varpi}(x,y)  \ \in \  {}_x\Pi^\varpi_{y}(\Pe^{\varpi})$$ 
    denote the generating series of canonical de Rham periods from $x$ to $y$, where $x$ and $y$ are either finite basepoints in $X_\Sigma(k)$ or tangential basepoints $\pm t_k$, for $0\leq k\leq n$. We let $I^\s(x,y)=\s(I^\varpi(x,y))$ denote its image by the single-valued period homomorphism. We thus have $\mathcal{Z}^{i,\bullet}=I^\bullet(t_0,-t_i)$ for $\bullet\in \{\varpi,\s \}$. Note that the same method of proof as in Proposition \ref{prop: abelianisation Z i s} gives the abelianisation of $I^\s(x,y)$:
    \begin{equation} \label{svIomegaabelian}
    \overline{ I^{\s}(x,y)} =  \left|\frac{y}{x}\right|^{2s_0} \prod_{k=1}^n \left| \frac{1-y \sigma_k^{-1}}{1-x \sigma_k^{-1}}\right|^{2s_k}\ \cdot
    \end{equation}

    We will need the following lemma, which gives a single-valued version of the integral
$$ \int_{x}^z \frac{dw}{w-\sigma} = \log \left( \frac{z-\sigma}{x-\sigma} \right)\ .$$

    \begin{lem} \label{lem: svlogint} Suppose that $ \sigma,x,z\in \C$ are distinct. Then
    $$ - \frac{1}{2\pi i} \iint_{\C} \left( \frac{d \overline{w}}{\overline{w}-\overline{z}} -\frac{d\overline{w}}{\overline{w}-\overline{x}} \right) \wedge \frac{d  w}{ w - \sigma}  \ = \  \log \left|  \frac{ z -\sigma }{  x - \sigma} \right|^2 \ .$$
    \end{lem} 

    \begin{proof} 
    See \cite[\S6.3]{BD1}.
    \end{proof}
    
    \begin{prop}\label{prop: beta quotient generic iterated integrals sv}
    For every $0\leq j\leq n$, the $j$-th beta quotient of $I^\s(x,y)$ is:
    \begin{equation} \label{svIomegaxij}  \overline{ I^{\s}(x,y)_j}  = -\frac{1}{2\pi i} \iint_{\C}  \left|\frac{z}{x}\right|^{2s_0} \prod_{k=1}^n \left| \frac{1-z \sigma_k^{-1}}{1-x \sigma_k^{-1}}\right|^{2s_k} \left(\frac{d \overline{z}}{\overline{z}- \overline{y}}  - \frac{d \overline{z}}{\overline{z} -\overline{x}}  \right)\wedge\frac{d z }{z - \sigma_j}  \ .
    \end{equation} 
    This expression is a formal power series in the $s_i$ obtained by expanding the exponentials as power series.
    \end{prop} 

    \begin{proof} The following  argument is slightly more intuitive using motivic, rather than de Rham periods, so we shall first compute $I_{\gamma}^{\mm}(x,y) \in   {}_x\Pi_{y}^\varpi(\Pe^{\mm})$, the image under the universal comparison map (\S\ref{subsect: NotGeometric} (5)) of a path $\gamma \in \pi_1(\C\backslash \Sigma, x,y)$, and then use the projection
    $$ I^{\varpi}(x,y) = \pi^{\mm,+}_{\varpi} \left(I_{\gamma}^{\mm}(x,y)\right) $$
    to deduce a formula for $I^{\varpi}(x,y)$.  Since $x,y$ are  ordinary basepoints, the motive underlying the torsor of paths  is given by Beilinson's  cosimplicial construction \cite[\S3.3]{delignegoncharov} and  $$I^{\mm}_\gamma(x,y) = \sum_{w \in \{e_0,\ldots, e_n\}^{\times} ,  |w|=\ell } [ H^{\ell} (X_{\Sigma}^{\ell}, Y^{\ell}),    \!\left[ \gamma  \, \Delta_{\ell}\right],w]^{\mm}\,  w$$
    where $|w|$ denotes the length of a word $w$, and the  divisor $ Y^\ell \subset X_{\Sigma}^{\ell}$ is  
    $$Y^{\ell} = \{z_1=x\} \cup \{z_1=z_2\} \cup \ldots \cup \{z_{\ell-1}=z_{\ell}\} \cup \{z_{\ell}=y\}\ ,$$
    and $\Delta_{\ell}$ is the standard simplex
    $$\Delta_{\ell} = \{ t_i \in \R\ : \  0 \leq t_1\leq t_2 \leq \ldots \leq t_{\ell} \leq 1 \} \subset \R^{\ell}\ .$$
    The coordinates $z_1,\ldots, z_{\ell}$ are the coordinates on $X_{\Sigma}^{\ell}$. 
    By Lemma \ref{lem: abelianisationofFbeta} the coefficient of $\frac{s_0^{m_0}}{m_0!} \ldots \frac{s_n^{m_n}}{m_n!}$ in 
    $\overline{I_{\gamma}^{\mm}(x,y)_j}$ is  
    $$  \left[H^{m+1} (X_{\Sigma}^{m+1}, Y^{m+1}) \ , \   [ \gamma \,\Delta_{m+1}]  \ , \ e_0^{\sha m_0}\sha \ldots \sha e_n^{\sha m_n} e_j  \right]^{\mm}   $$
    where $m = m_0+\cdots + m_n$. 
    By expanding out the shuffle products we get a sum of $m!$ terms indexed by permutations $\sigma\in \mathfrak{S}_m$. After permuting the coordinates the sum can be rewritten as
    \begin{equation} \label{MetabMotive}  \xi=  \left[H^{m+1} (X_{\Sigma}^{m+1} \ , \   \widetilde{Y}^{m+1}) \ , \  \left[ \gamma\,  C_{m+1}\right]  \ , \ e_0^{m_0} \ldots e_n^{m_n} e_j \right]^{\mm} 
    \end{equation} 
    where
    $$\widetilde{Y}^{m+1} = \bigcup_{\sigma\in \mathfrak{S}_m} \sigma Y^{m+1} \qquad \hbox{ and }\qquad    C_{m+1} = \bigcup_{\sigma\in\mathfrak{S}_m} \sigma \, \Delta_{m+1}$$
    and $\sigma \in \mathfrak{S}_{m}$ ranges over permutations of all but the last coordinate, i.e., 
    $$\sigma(z_1,\ldots,z_m,z_{m+1})=(z_{\sigma^{-1}(1)},\ldots,z_{\sigma^{-1}(m)},z_{m+1})\ .$$
    The union of the $m!$ simplices $\sigma\Delta_{m+1}$ glue together to form a cone
    $$C_{m+1} =  \{  t_i \in \R: 0\leq t_1,\ldots, t_{m} \leq t_{m+1}\leq 1 \} \subset \R^{m+1} \ .$$
    The boundary of $\gamma (C_{m+1})$  is contained in the complex points of the divisor $V\subset X_{\Sigma}^{m+1}$ defined by the union of $\{z_i=x\}$,  $\{z_i = z_{m+1}\}$ for $1\leq i \leq m$, and  $\{z_{m+1}=y\}$. 
    In \eqref{MetabMotive}, therefore, we can replace $H^{m+1} (X_{\Sigma}^{m+1}, \widetilde{Y}^{m+1})$ with 
    $H^{m+1} (X_{\Sigma}^{m+1}, V)$.
    Now take the image of   \eqref{MetabMotive} under the projection $\pi^{\mm,+}_{\varpi}.$  By \cite[\S 4.4]{BD1}, the image of the homology framing under the rational period map $c_0^{\vee}$ studied in loc. cit. is the differential form (writing $z=z_{m+1}$):
    $$\nu= (-1)^{\frac{m(m+1)}{2}}\bigwedge_{i=1}^m \left(  \frac{dz_i}{z_i-z} -\frac{dz_i}{z_i-x}    \right) \wedge  \left(\frac{dz}{z-y} -\frac{dz}{z-x}\right) \ . $$
    It follows, then, from Theorem 3.17 in \cite{BD1} that the single-valued period of 
    $ \pi^{\mm,+}_{\varpi}  \xi$ is 
    $$\frac{(-1)^{\frac{m(m+1)}{2}}}{(-2\pi i)^{m+1}}\iint_{\C^{m+1}}  \overline{\nu} \wedge  \frac{dz_1}{z_1 -\beta_1}\wedge \cdots \wedge\frac{dz_m}{z_m- \beta_m} \wedge \frac{dz}{z-\sigma_j}$$
    where $(\beta_1,\ldots, \beta_m)$ is $(0^{m_0}, \sigma_1^{m_1}, \ldots, \sigma_n^{m_n})$ (a sequence of $m_0$ $0$'s followed by $m_1$ $\sigma_1$'s and so on), corresponding to the differential form associated to the word $e_0^{m_0} \ldots e_n^{m_n} e_j$ in the affine ring of the de Rham fundamental groupoid. 
    Now rearrange the integrand and apply Lemma \ref{lem: svlogint} repeatedly to perform the $m$ integrals:
    $$ \frac{-1}{2\pi i} \iint_{\C}     \left(  \frac{d\overline{z_i}}{\overline{z_i}-\overline{z}} -\frac{d\overline{z_i}}{\overline{z_i}-\overline{x}}    \right)      \wedge \frac{dz_i}{z_i -\beta_i} = \log \left( \left| \frac{z-\beta_i}{x- \beta_i}\right|^2 \right)  $$  
    for $1\leq i \leq m$ to obtain
    $$-\frac{1}{2\pi i} \iint_{\C}  \log^{m_0} \left(\left|\frac{z}{x}\right|^{2}\right)  \prod_{k=1}^n \log^{m_k} \left(  \left| \frac{1-z \sigma_k^{-1}}{1-x \sigma_k^{-1}}\right|^{2} \right)\left(  \frac{d \overline{z}}{\overline{z}- \overline{y}}  -\frac{d\overline{z}}{\overline{z}-\overline{x}}\right)\wedge\frac{d z }{z - \sigma_j}  \ .$$
    This yields \eqref{svIomegaxij} after expanding the exponential factors as power series in the $s_i$.
    \end{proof} 
    
    Our next step is to replace $x$ by $t_0$ and $y$ by $-t_i$ in \eqref{svIomegaxij}. This is slightly subtle because Beilinsons's description of the motivic fundamental group with finite basepoints that was used in the proof of Proposition \ref{prop: beta quotient generic iterated integrals sv} is not available in the case of tangential basepoints. Thus, we will proceed as in \cite[\S 4]{delignegoncharov} and use the composition of paths to travel between two tangential basepoints by using finite basepoints as intermediate steps. We will thus need to understand the behaviour of single-valued versions of iterated integrals between a tangential basepoint and an infinitesimally close finite basepoint. 
    
    We will consider non-commutative power series $F(\tau)\in \mathbb{C}\langle\langle e_0,e_1,\ldots,e_n\rangle\rangle$ depending on a ``small'' rational point $\tau\in k^\times$ and having constant coefficient $1$. We say that such a series $F(\tau)$ has \emph{logarithmic growth} if each of its coefficients $a(\tau)$ satisfies $a(\tau)=O(\log^r|\tau|)$ for some integer $r$ (that may depend on the coefficient) when $\tau\to 0$. We say that such a series $F(\tau)$ is \emph{asymptotic to $1$} and write $F(\tau) \sim_{\tau \to 0}1$ if each of its non-constant coefficients $a(\tau)$ satisfies $a(\tau) = O(|\tau|^{1-\varepsilon})$ for every $\varepsilon >0$ when $\tau\to 0$. The class of series with logarithmic growth and the subclass of series asymptotic to $1$ are stable under products and inversion. This implies that the following relation is an equivalence relation that is compatible with products:
    $$F(\tau) \underset{\tau\to 0}{\sim} G(\tau) \qquad \Leftrightarrow \qquad F(\tau) G(\tau)^{-1} \underset{\tau\to 0}{\sim} 1\ .$$
    
    \begin{lem}\label{lem: infinitesimal sv iterated integral}
    We have, for $\tau\in X_\Sigma(k)$:
    $$I^\s(t_0,\tau) \underset{\tau\to 0}{\sim} \exp(e_0\log|\tau|^2)\ .$$
    \end{lem}
    
    \begin{proof}
    By a slight generalisation of \cite[(5.4)]{brownSVMZV} to the case of the projective line minus several points $\Sigma$ (which follows, for example, by the argument in \emph{loc. cit.}, \S6.3), we have an expression of the form
    $$I^{\s}(t_0, \tau) =  I(t_0,\tau) \widetilde{I} (t_0,\tau) $$
    where $\widetilde{I}$ is  the complex conjugate of the series $I(t_0,\tau)$, in which  the letters $e_i$, for $i\geq 1$, are replaced with certain power series,  the  letter $e_0$ is unchanged, and all words are reversed.  Note that in that paper, the order is reversed because iterated integrals were computed from left to right, but it makes no difference to the conclusion of the lemma.  
            Since $$ I(t_0,\tau) \underset{\tau\to 0}{\sim}\exp(e_0\log \tau )\ , $$ it follows also that
    $\widetilde{I} (t_0,\tau)   \underset{\tau\to 0}{\sim}\exp( e_0\log \overline{\tau }) $
 and hence $I^{\s}(t_0, \tau) \underset{\tau\to 0}{\sim} \exp( e_0 \log |\tau|^2)$. 
    \end{proof}
   For all $1\leq j\leq n$ we define:
    \begin{equation} \label{OmegajDef sv} \Omega_j^\s =   |z|^{2s_0} \prod_{k=1}^n |1-z \sigma_{k}^{-1}|^{2s_k}\, \frac{dz}{z-\sigma_j} \ .
    \end{equation} 
    Let $\Omega^{\s,\ren_i}_j$ denote its renormalised versions (Definition \ref{definition: renormalisedformssv})  with respect to  $\{0, \sigma_i\}$, given by
    $$\Omega_j^{\s,\ren,i} = \Big(|z|^{2s_0}\prod_{k\neq i}|1-z\sigma_k^{-1}|^{2s_k} - \mathbf{1}_{i=j}\, |\sigma_i|^{2s_0}\prod_{k\neq i}|1-\sigma_i\sigma_k^{-1}|^{2s_k}\Big)|1-z\sigma_i^{-1}|^{2s_i}\frac{dz}{z-\sigma_j}\ \cdot$$
    According to Proposition \ref{prop: taylor series LSigma sv}, the integral $-\frac{1}{2\pi i}\iint_{\mathbb{C}}\left( \frac{d \overline{z}}{ \overline{z}-\overline{\sigma_i} }    -  \frac{d \overline{z}}{ \overline{z} }  \right)  \wedge\Omega_j^{\s,\ren_i}$ defines a holomorphic function of the parameters $\underline{s}$ around $\underline{s}=0$ and thus has a Taylor expansion. We now identify this Taylor expansion with the $j$-th beta quotient $\overline{\mathcal{Z}^{i,\s}_j}$ of the generalised single-valued associator $\mathcal{Z}^{i,\s}$.

    \begin{prop}   \label{thmsvZios} 
    
    For all $1\leq j \leq n$, the $j$-th beta quotient of $\overline{\mathcal{Z}^{i,\s}_j}$ is the Taylor series at $\underline{s}=0$ of the integral 
    \begin{equation}\label{svZio}
    -\frac{1}{2\pi i}\iint_{\mathbb{C}}\left( \frac{d \overline{z}}{ \overline{z}-\overline{\sigma_i} }    -  \frac{d \overline{z}}{ \overline{z} }  \right)  \wedge\Omega_j^{\s,\ren_i}\ .
    \end{equation}
    \end{prop}

    \begin{proof} By the  composition of paths formula, we have for any small enough $\tau\in X_\Sigma(k)$, the equality: 
    \begin{equation} \label{inproofziomega} \mathcal{Z}^{i,\s} = I^{\s}(t_0,\tau)  \, I^{\s}(\tau, \sigma_i-\tau)\,  I^{\s}(\sigma_i -\tau, -t_i) 
    \end{equation}
    Now Lemma \ref{lem: infinitesimal sv iterated integral} implies that we have 
    $$I^{\s}(t_0,\tau) \underset{\tau\to 0}{\sim}  \exp (e_0 \log |\tau|^2) \qquad \mbox{ and } \qquad  I^\s(\sigma_i-\tau, -t_i) \underset{\tau\to 0}{\sim}  \exp(-e_i\log|\tau|^2+e_i\log|\sigma_i|^2)\ .  $$
    (For the second claim one needs to apply the change of variables $x\mapsto 1-x\sigma_i^{-1}$.)
    Since $\mathcal{Z}^{i,\s}$ is independent of $\tau$,    we conclude that
    $$\overline{\mathcal{Z}^{i,\s}_j} =  \mathrm{Reg}_{\tau\rightarrow 0} \,  \overline{\left( I^{\s}(\tau, \sigma_i-\tau)   \exp(e_i \log|\sigma_i|^2)  \right)_j} $$
    where $\mathrm{Reg}_{\tau\rightarrow 0}$ means the following: formally set  $\log|\tau|^2$ to zero, and then take the limit as $\tau\rightarrow 0$. Since the coefficients in the formal power series  $I^{\s}(\tau, \sigma_i-\tau)$ 
    can be expressed as elements  in $\C[[\tau, \overline{\tau}]][\log|\tau|^2]$, this operation is well-defined. By \eqref{SubDerivation} we thus get:
    \begin{equation} \label{szioinprooflimit} \overline{\mathcal{Z}^{i,\s}_j}  =  \mathrm{Reg}_{\tau \rightarrow 0}  \left(    \overline{ I^{\s}(\tau, \sigma_i-\tau) } \, \overline{\left( |\sigma_i|^{2e_i} \right)_j} + \overline{  I^{\s}(\tau, \sigma_i-\tau)_j } \right)\ .
    \end{equation} 
    Note that by \eqref{svIomegaabelian} we have
    $$\overline{I^\s(\tau,\sigma_i-\tau)} = \left|\frac{\sigma_i-\tau}{\tau}\right|^{2s_0} \prod_{k=1}^n\left|\frac{1-(\sigma_i-\tau)\sigma_k^{-1}}{1-\tau\sigma_k^{-1}}\right|^{2s_k}$$
    so that
     $$\mathrm{Reg}_{\tau\to 0}(\overline{I^\s(\tau,\sigma_i-\tau)}) = |\sigma_i|^{2s_0}  \, |\sigma_i|^{-2s_i}\, \prod_{k\neq i}   |1-\sigma_i\sigma_k^{-1}|^{2s_k} \ .$$
    Indeed, the factors $|\tau|^{2s_0}$ and $|\tau|^{2s_i}$ are sent to $1$ by the regularisation map. We also have $\overline{(|\sigma_i|^{2e_i})_j}=\mathbf{1}_{i=j}\frac{1}{s_i}(|\sigma_i|^{2s_i}-1)$
    which yields
    \begin{equation}\label{eq: first summand in proof before}
    \mathrm{Reg}_{\tau\to 0}(\overline{ I^{\s}(\tau, \sigma_i-\tau) } \, \overline{\left( |\sigma_i|^{2e_i} \right)_j}) = \mathbf{1}_{i=j}|\sigma_i|^{2s_0} \,\left( \prod_{k\neq i} |1-\sigma_i\sigma_k^{-1}|^{2s_k} \right)\  \frac{1}{s_i}(1-|\sigma_i|^{-2s_i})\,\cdot
    \end{equation}
    Using a variant of Lemma \ref{lem: svlogintegralregterm} we interpret the right-most factor as:
    $$  \frac{1}{s_i} \left(1- |\sigma_i|^{-2s_i}\right)  =  \mathrm{Reg}_{\tau \rightarrow 0} \,\left(  \frac{1}{2\pi i} \iint_{\C} \left| 1- z \sigma_i^{-1} \right|^{2s_i} \left(\frac{d\overline{z}}{\overline{z}-\overline{\sigma_i} + \overline{\tau} }  -\frac{d\overline{z}}{\overline{z}} \right)\wedge \frac{dz}{z- \sigma_i}\right)  \ .$$
    This allows us to rewrite \eqref{eq: first summand in proof before} as
    \begin{equation}\label{first summand in proof}
    \mathrm{Reg}_{\tau\to 0} \Bigg(\frac{1}{2\pi i}\iint_{\mathbb{C}}\mathbf{1}_{i=j} |\sigma_i|^{2s_0}  \prod_{k\neq i}|1-\sigma_i\sigma_k^{-1}|^{2s_k} \times |1-z\sigma_i^{-1}|^{2s_i} \left(\frac{d\overline{z}}{\overline{z}-\overline{\sigma_i} + \overline{\tau} }  -\frac{d\overline{z}}{\overline{z}} \right)\wedge \frac{dz}{z- \sigma_i} \Bigg).
    \end{equation}
    By Proposition \ref{prop: beta quotient generic iterated integrals sv},
    $$\overline{I^\s(\tau,\sigma_i-\tau)_j} = -\frac{1}{2\pi i}\iint_{\mathbb{C}} \left(\left|\frac{z}{\tau}\right|^{2s_0} \prod_{k=1}^n \left| \frac{1-z \sigma_k^{-1}}{1-\tau  \sigma_k^{-1}}\right|^{2s_k} \right) \left( \frac{d \overline{z}}{\overline{z}- \overline{\sigma_i} +\overline{\tau}} - \frac{d{\overline{z}}}{ \overline{z}- \overline{\tau}}  \right)\wedge\frac{d z }{z - \sigma_j}\ \cdot
    $$
    which implies  that 
    \begin{equation}\label{second summand in proof}
    \mathrm{Reg}_{\tau\to 0}(\overline{I^\s(\tau,\sigma_i-\tau)_j}) = \mathrm{Reg}_{\tau\to 0} \left(-\frac{1}{2\pi i}\iint_{\mathbb{C}}\left(\frac{d\overline{z}}{\overline{z}-\overline{\sigma_i}+\overline{\tau}}-\frac{d\overline{z}}{\overline{z}}\right)\wedge \Omega_j^{\s}\right)\ .
    \end{equation}
  since one has 
        $$  \mathrm{Reg}_{\tau\to 0}\, \left(\left|\frac{z}{\tau}\right|^{2s_0} \prod_{k=1}^n \left| \frac{1-z \sigma_k^{-1}}{1-\tau  \sigma_k^{-1}}\right|^{2s_k} \right)    \frac{d z }{z - \sigma_j}  = \Omega_j^{\s}\ $$
        and  the left-hand side can be expanded out in terms in $\tau, \overline{\tau}$ and $\log|\tau|^2$, which factor out of the integral.
       Resubstituting \eqref{first summand in proof} and \eqref{second summand in proof} into \eqref{szioinprooflimit} yields
    $$\overline{\mathcal{Z}^{i,\s}_j} = \mathrm{Reg}_{\tau\to 0} \left(-\frac{1}{2\pi i}\iint_{\mathbb{C}}\left(\frac{d\overline{z}}{\overline{z}-\overline{\sigma_i}+\overline{\tau}}-\frac{d\overline{z}}{\overline{z}}\right)\wedge \Omega^{\ren,i}_j\right) = -\frac{1}{2\pi i}\iint_{\mathbb{C}}\left(\frac{d\overline{z}}{\overline{z}-\overline{\sigma_i}}-\frac{d\overline{z}}{\overline{z}}\right)\wedge \Omega^{\ren,i}_j\ .$$
    Here $\mathrm{Reg}_{\tau\to 0}$ is simply the limit when $\tau\to 0$ since the last integral is convergent by Proposition \ref{prop: taylor series LSigma sv}. The claim follows.
    \end{proof}

\subsection{Comparing \texorpdfstring{$L_\Sigma^\s$}{LSigma sv} and \texorpdfstring{$FL_\Sigma^\s$}{FLSigma sv}}

    We are now ready to prove part (\emph{ii}) of Theorem \ref{introthmperandsv} from the introduction.

    \begin{thm}\label{thm: FL equals L sv}
     For every $i,j$, $(FL_\Sigma^\s)_{ij}$ is the Taylor series at $\underline{s}=0$ of the single-valued Lauricella function $(L_\Sigma^\s)_{ij}$.
    \end{thm}
    
    \begin{proof}
    This follows from comparing the definition of $(FL^\s_\Sigma)_{ij}$ with the expression in Proposition \ref{prop: taylor series LSigma sv} and using the expressions for $\overline{\mathcal{Z}^{i,\s}}$ and $\overline{\mathcal{Z}^{i,\s}_j}$    from Propositions \ref{prop: abelianisation Z i s} and  \ref{thmsvZios}, respectively. 
    \end{proof}

\section{`Local' motivic coaction}

We compute the action of the motivic Galois group (or equivalently, the motivic coaction) on the full motivic torsor of paths, and use it to deduce a formula for the `local' coaction on the beta quotients and on the Lauricella functions viewed as formal power series in their parameters.

\subsection{Formula for the motivic Galois action} 
Since the ${}_i\Pi_j^\varpi$ are (dual to) realisations of ind-objects in the Tannakian category $\MT(k)$, they admit an action of the motivic Galois group. More precisely, the Galois group $G^\varpi_{\MT(k)}$ acts \emph{on the left} on the $\Q$-algebra  $\mathcal{O}({}_0\Pi_i^\varpi)$ and thus naturally acts \emph{on the right} on the set of points ${}_0\Pi_i^\varpi(R)$ for every $\Q$-algebra $R$. 
Let 
$$\lambda : G^{\varpi}_{\MT(k)} \rightarrow \G_m \;\; , \; g\mapsto \lambda_g$$ 
denote the homomorphism given by the action of $G^{\varpi}_{\MT(k)}$  on $\varpi(\Q(-1))= \Q$.

\begin{prop}  \label{propactiononZ} Let $R$ be any $\Q$-algebra. For every element $g\in G^{\varpi}_{\MT(k)}(R)$,  its right action on any $F\in {}_0\Pi_i^\varpi(R)$ is given by a version of Ihara's formula:
\begin{equation}  \label{gIharaaction}
  (F\cdot g)(e_0,e_1,\ldots, e_n)  =  F \left( \lambda_g e_0, \lambda_gG_1 e_1 G_1^{-1}, \ldots, \lambda_g G_n e_n G_n^{-1}\right) G_i \end{equation} 
where $G_k \in \Q \langle \langle e_0,\ldots, e_n \rangle \rangle $ is the group-like formal power series $G_k = {}_0 1_k \cdot g$ for all $1\leq k\leq n$.
  \end{prop} 

\begin{proof} The argument is a very mild  generalisation of the  argument given in \cite[\S 5]{delignegoncharov} (with the reverse conventions) or \cite[Proposition 2.5]{brownICM} so we shall be brief.
One first computes the action of $g$ on ${}_0\Pi_0^\varpi$.   It acts   on the element $\exp(e_0) \in {}_0 \Pi_0^\varpi$  by scaling: 
$$\exp(e_0)\cdot g = \exp(\lambda_g e_0)$$ 
since $\exp(e_0)$ is in the image of the local monodromy
$\pi_1^{\mathrm{mot}}(\G_m, 1)$, which is isomorphic to $\Q(1)=\Q(-1)^\vee$. The fact that we get a $\lambda_g$ and not a $\lambda_g^{-1}$ is because $g$ acts on the right. Another way to see this is that $g$ acts on the element `coefficient of $e_0^n$', which lies in $\varpi(\Q(-n))$, by $\lambda_g^n$.
For all $1 \leq i \leq n$,  the element $\exp(e_{i}) \in {}_i\Pi_i^\varpi$  is in the image of the local monodromy
$$x\mapsto \sigma_i x:  (\G_m, 1) \To  ((T_{\sigma_i} \A^1_k)^{\times}, t_i) $$
and hence, by a similar argument, is also acted upon by $g$ by scaling $\exp(e_i)\cdot g = \exp(\lambda_g e_i)$. We transport this action back to ${}_0\Pi_0^\varpi$ via
$$ {}_i\Pi_i^\varpi \longrightarrow {}_0\Pi_0^\varpi \quad , \quad {}_iF_i\mapsto ({}_01_i)\, {}_iF_i\, ({}_01_i)^{-1} $$
where ${}_xF_y\in {}_x\Pi_y^\varpi$ denotes the element defined by a power series $F\in \Q\langle\langle e_0,\ldots,e_n\rangle\rangle$. Since the action of the motivic Galois group is compatible with the composition of paths,  we deduce
that $g$ acts on $\exp(e_i)\in {}_0\Pi_0^\varpi$ via $\exp(e_i)\mapsto G_i\exp(\lambda_g e_i)G_i^{-1}=\exp(\lambda_gG_ie_iG_i^{-1})$ for all $1\leq i\leq n$ since $G_i$ is by definition ${}_01_i\cdot g$. 
Finally, use the torsor structure 
$${}_0\Pi_0^\varpi \longrightarrow {}_0\Pi_i^\varpi \quad ,\quad {}_0F_0 \mapsto {}_0F_0 \, {}_01_i$$
to conclude that the action of $g$ on any $F \in {}_0 \Pi_i^\varpi $ is indeed as claimed.
\end{proof}

The motivic Galois group acts in (at least) two different ways on the set ${}_0\Pi_i^\varpi(\Pe^\mm)$:
\begin{enumerate}
    \item  \emph{on the right} via the \emph{Ihara action} \eqref{gIharaaction} for $R=\Pe^\mm$, described in Proposition \ref{gIharaaction};
    \item  \emph{on the left} via its action \emph{on the coefficients} $\Pe^\mm$, i.e.,   term by term on the coefficients of formal power series in $\mathcal{P}^\mm\langle\langle e_0,\ldots,e_n\rangle\rangle$.
\end{enumerate}
We are interested in computing the action  (2) on the generalised motivic associators $\mathcal{Z}^{i,\mm}$. The next lemma shows that this action is  equivalent to  the action  (1) on these elements.

\begin{lem}\label{lem: two actions}
The Ihara action (1) and the action on the coefficients (2) coincide on the motivic associator $\mathcal{Z}^{i,\mm}$, i.e., we have for every $g\in G^\varpi_{\MT(k)}(\Q)$:
$$\mathcal{Z}^{i,\mm}\cdot g = g\cdot \mathcal{Z}^{i,\mm}\ .$$
\end{lem}

\begin{proof}
We have a map $\mathcal{O}({}_0\Pi_i^\varpi)\to \Pe^\mm$ that sends a word $w$ to the motivic period $[\mathcal{O}({}_0\Pi_i^{\mathrm{mot}}),\gamma_i,w]^\mm$. It is left $G_{\MT(k)}^\varpi$-equivariant by definition, and induces a map
\begin{equation}\label{eq: in proof two actions}
{}_0\Pi_i^\varpi(\mathcal{O}({}_0\Pi_i^\varpi))\longrightarrow {}_0\Pi_i^\varpi(\Pe^\mm)
\end{equation}
that is also left $G_{\MT(k)}^\varpi$-equivariant, where $G_{\MT(k)}^\varpi$ acts on the coefficients, i.e., by (2). It is also obviously right $G_{\MT(k)}^\varpi$-equivariant for the action (1). By definition, $\mathcal{Z}^{i,\mm}$ is the image under \eqref{eq: in proof two actions} of the element $\Phi^i\in {}_0\Pi_i^\varpi(\mathcal{O}({}_0\Pi_i^\varpi))$ which corresponds to the map $\id:\mathcal{O}({}_0\Pi_i^\varpi)\to \mathcal{O}({}_0\Pi_i^\varpi)$. It is thus enough to prove the claim that  $\Phi^i\cdot g=g\cdot \Phi^i$, where the right and left actions of $G_{\MT(k)}^\varpi$ on ${}_0\Pi_i^\varpi(\mathcal{O}({}_0\Pi_i^\varpi))=\mathrm{Hom}_{\mathrm{Alg}}(\mathcal{O}({}_0\Pi_i^\varpi),\mathcal{O}({}_0\Pi_i^\varpi))$ are on the source and on the target, respectively. It is obvious that $\id \cdot g = g \cdot \id$, from which the claim follows. 
\end{proof}

Applying Proposition \ref{gIharaaction} to the series $F= \Zim$ and using Lemma \ref{lem: two actions}, we deduce that the action (2) of $G_{\MT(k)}^\varpi$ term by term on the coefficients of $\mathcal{Z}^{i,\mm}$ is given by the formula:
\begin{equation} \label{gactingonZim} g \cdot \Zim (e_0,e_1,\ldots, e_n)  =  \Zim \left(  \lambda_g e_0,  \lambda_g G_1 e_1 G_1^{-1}, \ldots, \lambda_g G_n e_n G_n^{-1}\right) G_i \  ,
\end{equation} 
where $G_k={}_01_i\cdot g$. The same formula holds with $\mm$ replaced by $\varpi$. 
This formula  can  be re-expressed as a universal coaction formula for
$$ \Delta \Zim  \in \left( \Pe^{\mm} \otimes_{\Q} \Pe^{\varpi} \right)  \langle \langle e_0,\ldots, e_n \rangle \rangle$$
 where $\Delta$ is  applied term by term to each coefficient and acts trivially on the $e_i$. The action of $g$ is retrieved  by the usual formula $g \cdot\Zim = (\id \otimes g)\, \Delta\Zim$.

\begin{prop}  \label{propCoactZim} The coaction $\Delta:\Pe^\mm\langle\langle e_0,e_1,\ldots,e_n\rangle\rangle \to (\Pe^\mm\otimes \Pe^\varpi)\langle\langle e_0,e_1,\ldots,e_n\rangle\rangle$ applied to the generalised motivic associator $\mathcal{Z}^{i,\mm}$ is:
\begin{equation} \label{DeltamZim} \Delta\Zim  = \Zim \left( \Lef^\varpi e_0, \Lef^\varpi e'_1,\ldots, \Lef^\varpi e'_n     \right)   \, \Zio     
\end{equation}
where $e_1', \ldots, e_n'$ are defined by 
\begin{equation} \label{ejprimeconjugate}
e_k'   =  \left(  \mathcal{Z}^{k,\varpi}  \right)  e_k    \left(\mathcal{Z}^{k,\varpi} \right)^{-1} \qquad \hbox{ for all } \quad  1\leq k \leq n\ .
\end{equation}
 \end{prop} 
 
In the right-hand side of \eqref{DeltamZim}, we view $\mathcal{Z}^{i,\mm}$ inside $(\Pe^\mm\otimes \Pe^\varpi)\langle\langle e_0,e_1,\ldots,e_n\rangle\rangle$ via the natural inclusion $\Pe^\mm\langle\langle e_0,e_1,\ldots,e_n\rangle\rangle\subset (\Pe^\mm\otimes \Pe^\varpi)\langle\langle e_0,e_1,\ldots,e_n\rangle\rangle$ which replaces every coefficient $a^\mm$ with $a^\mm\otimes 1$. Similarly, the terms $\Lef^\varpi e_0$, $\Lef^\varpi e'_k$ and $\mathcal{Z}^{i,\varpi}$ are viewed in $(\Pe^\mm\otimes \Pe^\varpi)\langle\langle e_0,e_1,\ldots,e_n\rangle\rangle$ via the natural inclusion   $\Pe^\varpi\langle\langle e_0,\ldots,e_n\rangle\rangle \subset (\Pe^\mm\otimes \Pe^\varpi)\langle\langle e_0,e_1,\ldots,e_n\rangle\rangle$ which replaces every coefficient $b^\varpi$ with $1\otimes b^\varpi$. Thus we interpret the term $\Zim \left(\Lef^\varpi e_0, \Lef^\varpi e'_1,\ldots, \Lef^\varpi e'_n \right)$ as a composition of noncommutative formal power series with coefficients in the ring $\mathcal{P}^\mm\otimes \mathcal{P}^\varpi$. This composition makes sense since $e_0$ and the $e'_k$ have vanishing constant coefficient. All products in the formulae in the Proposition are given by concatentation of non-commutative formal power series.

\begin{proof} 
For all $ g \in G^{\varpi}_{\MT(k)}$, $G_k={}_01_k\cdot g$ is obtained by applying to $g$ the function (more precisely the power series of functions) $\mathcal{Z}^{k,\varpi}\in {}_0\Pi_k^\varpi(\mathcal{O}(G^\varpi_{\MT(k)}))$, and $\lambda_g$ is obtained by applying to $g$ the function $\Lef^\varpi$. The claim follows from equation \eqref{gactingonZim}. 
\end{proof}

\begin{example}  In the setting of Example \ref{ExMZV1} and Example \ref{ExMZV1 sv}, 
Proposition \ref{propCoactZim}  yields
$$\Delta \ZZ^{\mm}(e_0,e_1)  =  \ZZ^{\mm} \big(\Lef^{\dR} e_0,  \Lef^\dR \ZZ^{\dR} e_1 \left(\ZZ^{\dR}\right)^{-1} \big) \,  \ZZ^{\dR}(e_0,e_1)$$  
which is a motivic version of  Ihara's formula, and expresses the coaction on motivic multiple zeta values.  For example, reading off the coefficient of $-e_1 e_0^{n-1}$ yields 
$$\Delta\, \zetam(n) = \zetam(n) \otimes (\Lef^{\dR})^n + 1 \otimes \zeta^{\dR}(n)\ .$$
Note that this provides a very efficient method of computing the coaction.
\end{example}

\subsection{Motivic coaction on the beta quotients}\label{subsec: coaction beta quotients}

We want to use Proposition \ref{propCoactZim} to derive a formula for the motivic coaction on the series $FL_\Sigma^\mm(s_0,\ldots,s_n)$. We now make a slight modification and consider the \emph{normalised coaction}
$$\Delta_{\mathrm{nor}}:\Pe^\mm[[s_0,\ldots,s_n]] \rightarrow (\Pe^\mm\otimes \Pe^\varpi) [[ s_0,\ldots,s_n]]$$
obtained by acting via $\Delta$ on each coefficient and on the formal variables $s_k$ by
$$\Delta_{\mathrm{nor}}(s_k)=(1\otimes (\Lef^\varpi)^{-1})\,s_k\ .$$
This is equivalent to viewing the variables $s_k$ as spanning a copy of $\Q(1)_{\mathrm{dR}}$. In particular, they have motivic weight $-2$. This is the correct normalisation which makes the factors
$$y^{s_k}= \exp(s_k\log(y))=\sum_{n\geq 0} \frac{\log^n(y)}{n!}s^n$$
have total weight $0$, and  all the coefficients in the expansion of $FL_\Sigma^\mm$ have total weight $0$. Another effect of this normalisation is that it makes the coaction formulae below land in the subspace
$$\Pe^\mm[[s_0,\ldots,s_n]]\otimes_{\Q[[s_0,\ldots,s_n]]} \Pe^\varpi[[s_0,\ldots,s_n]] \;\subset \; (\Pe^\mm\otimes_\Q \Pe^\varpi)[[s_0,\ldots,s_n]]\ .$$
Note that  the tensor product on the left-hand side is an ordinary,  not a  completed, tensor product. It is a highly restrictive condition for an element in the space on the right-hand side to lie in the subspace defined by the left-hand side. It is a crucial, and non-trivial fact, that this condition is satisfied  for the image of the motivic coaction on generalised associators.  Indeed, this can already be seen from from the abelianisation of \eqref{DeltamZim}:
\begin{equation}\label{eq: coaction nor ab}
\Delta_{\mathrm{nor}}\overline{\mathcal{Z}^{i,\mm}}(s_0,\ldots,s_n) = \overline{\mathcal{Z}^{i,\mm}}(s_0,\ldots,s_n) \otimes \overline{\mathcal{Z}^{i,\varpi}}((\Lef^\varpi)^{-1}s_0,\ldots,(\Lef^\varpi)^{-1}s_n)\ .
\end{equation}

\begin{thm} \label{thmMainCoactionMmm}
The (normalised) motivic coaction,  applied to the entries of $\FL_{\Sigma}^{\mm}$,  satisfies:
\begin{equation} \label{CoactiononMmm} \Delta_{\mathrm{nor}}   \FL_{\Sigma}^{\mm} (s_0,\ldots, s_n)=  \FL_{\Sigma}^{\mm} ( s_0,\ldots, s_n)\otimes \FL_{\Sigma}^{\varpi}((\Lef^{\varpi} )^{-1}s_0,\ldots, (\Lef^{\varpi} )^{-1}s_n) \ .\end{equation} 
\end{thm}

\begin{proof}  It is convenient  to compute modulo $\Lef^{\varpi}=1$ and restore all powers of $\Lef^{\varpi}$ at the end, since they are uniquely determined by the weight grading. 
Using formula \eqref{DeltamZim} we have 
$$\Delta_{\mathrm{nor}} \overline{\Zim_j}  =  \overline{\left(\Zim \left(e_0, e'_1,\ldots, e'_n     \right)  \, \Zio \right)_j}    \ . $$
The right-hand side reduces via \eqref{SubDerivation} to
$$\overline{\Zim}(s_0,s_1,\ldots,s_n)\otimes  \overline{\Zio_j}(s_0,s_1,\ldots,s_n)  +   \overline{\Zim \left( e_0,  e'_1,\ldots, e'_n     \right)_j}  $$
since $e'_k$ is conjugate to $e_k$  via equation \eqref{ejprimeconjugate} and so they have the same image $\overline{e_k} = \overline{e'_k} = s_k$  under abelianisation.   The previous expression can in turn be written
$$\overline{\Zim}\otimes  \overline{\Zio_j}  +   \sum_{k=1}^n \overline{\Zim_k}\otimes \overline{(e'_k)_j} \   = \   \overline{\Zim}\otimes \overline{\Zio_j}  +   \sum_{k=1}^n \overline{\Zim_k} \otimes \overline{ \left( \mathcal{Z}^{k,\varpi} e_k\left(\mathcal{Z}^{k,\varpi} \right)^{-1}\right)_j }  \ ,$$
by applying definition \eqref{Fsubscripts}  to  $\overline{\Zim \left( e'_0, e'_1,\ldots, e'_n     \right)_j}$ and using \eqref{ejprimeconjugate}.
 We have
$$\overline{ \left( \mathcal{Z}^{k,\varpi} e_k\left(\mathcal{Z}^{k,\varpi} \right)^{-1}\right)_j } 
 \ \overset{\eqref{SubDerivation}}{ = }\  \overline{ \mathcal{Z}^{k,\varpi}e_k }\overline{\left(\left(\mathcal{Z}^{k,\varpi} \right)^{-1}\right)_j } + \overline{  \mathcal{Z}^{k,\varpi} }   \mathbf{1}_{j=k}
\ \overset{\eqref{SubInversion}}{ = }\ 
  -    s_k   \,  \overline{\mathcal{Z}^{k,\varpi}_j}  + \mathbf{1}_{j=k}  \overline{  \Zjo }  $$
 Putting the pieces together and multiplying by $-s_j$ we get
 \begin{equation}\label{eq: coaction nor beta}
 \Delta_{\mathrm{nor}}\left(-s_j\overline{\mathcal{Z}^{i,\mm}_j}\right) = -s_j\overline{\mathcal{Z}^{i,\mm}} \otimes \overline{\mathcal{Z}^{i,\varpi}_j} - s_j \overline{\mathcal{Z}^{i,\mm}_j}\otimes \overline{\mathcal{Z}^{j,\varpi}} + \sum_{k=1}^n s_k\overline{\mathcal{Z}^{i,\mm}_k} \otimes s_j\overline{\mathcal{Z}^{k,\varpi}_j} \ .
 \end{equation}
 Now by \eqref{eq: coaction nor ab} we have 
 \begin{equation}\label{eq: coaction nor ab in proof} \Delta_{\mathrm{nor}}\overline{\mathcal{Z}^{i,\mm}} = \overline{\mathcal{Z}^{i,\mm}}\otimes \overline{\mathcal{Z}^{i,\varpi}}\ .
 \end{equation}
 Substituting \eqref{eq: coaction nor beta} and \eqref{eq: coaction nor ab in proof} into the definitions $(FL_\Sigma^\bullet)_{ij} = \mathbf{1}_{i=j}\overline{\mathcal{Z}^{i,\bullet}}-s_j\overline{\mathcal{Z}^{i,\bullet}_j}$ we get 
 $$ \Delta_{\mathrm{nor}}FL_\Sigma^\mm = FL_\Sigma^\mm\otimes FL_\Sigma^\varpi\ .$$
 On the other hand, homegeneity in the weight forces the right-hand side of the coaction to have weight equal to the degree in the $s_i$.  This determines the powers of $\Lef^{\varpi}$ as in  equation \eqref{CoactiononMmm}.
  \end{proof}

\begin{rem}
The (normalised) coproduct in $\Pe^\varpi[[ s_0,\ldots,s_n]]$ is given on the elements $FL^\varpi_\Sigma$ by a formula similar to \eqref{CoactiononMmm}: $\Delta_{\mathrm{nor}}FL^\varpi_\Sigma(s_0,\ldots, s_n)=  \FL_{\Sigma}^{\varpi} ( s_0,\ldots, s_n)\otimes \FL_{\Sigma}^{\varpi}((\Lef^{\varpi} )^{-1}s_0,\ldots, (\Lef^{\varpi} )^{-1}s_n)$.
\end{rem}

\section{Example: Single-valued version of \texorpdfstring{${}_2F_1$}{2F1}  and double copy formula}\label{section: 2F1}

A large family of  special functions commonly found in  the mathematical literature can be derived as special cases or limits of the Gauss hypergeometric function. As a result, one can derive single-valued versions for a range of special functions from the single-valued versions of the hypergeometric function \eqref{eq: double copy Fabc}, \eqref{eq: double copy Gabc}, which we shall prove here.

\subsection{The Gauss hypergeometric function}
We denote it by $F={}_2F_1$. It is defined for $y\in \mathbb{C}$, $|y|<1$, by the power series
\begin{equation}\label{eq: definition 2F1 series}
F(a,b,c; y) = \sum_{n=0}^{\infty}  \frac{ (a)_n (b)_n}{(c)_n}  \frac{y^n}{n!}
\end{equation}
where   $(x)_n = \prod_{i=1}^n (x+i-1)$ is  the  rising Pochhammer symbol. It is a solution of the famous hypergeometric differential equation:
    \begin{equation}\label{eq: 2F1 differential equation}
    \left(y(1-y)\frac{d^2}{d y^2} + (c-(a+b+1)y)\frac{d}{d y} -ab \right)F(a,b,c;y)=0\ .
    \end{equation}

    \subsubsection{Integral representation}\label{subsubsec: def Fcal}
    
    Traditionally, $F$ is  viewed as a function of $y$ for  fixed values of the exponents $a,b,c$.  In this case, it admits an analytic continuation to a multivalued function on $\C\backslash \{0,1\}$ via the following integral representation which is valid for $\Real(c)>\Real(b)>0$:
    \begin{equation}\label{eq: 2F1 integral} 
    F(a,b,c;y)  = \frac{1}{\beta(b,c-b) }\int_0^1 x^{b-1} (1-x)^{c-b-1} (1-yx)^{-a} dx\ ,
    \end{equation}
    where $\beta$ denotes Euler's beta function.

    To fix branches, it is convenient to assume that $y\notin \mathbb{R}_{>0}$.
    The path of integration in \eqref{eq: 2F1 integral} can then be chosen to be the line segment $(0,1)$, and the branch of $(1-yx)^{-a}=\exp(-a\log(1-yx))$ is determined by $\log(1-yx)=\int_0^x d\log(1-yu)$ for $x\in (0,1)$.

    We multiply through by the $\beta$ factor and set
    \begin{equation}\label{eq: Fabcy}
    \mathcal{F}(a,b,c;y) = \beta(b, c-b) \, F(a,b,c;y) = \int_0^1x^b(1-x)^{c-b}(1-yx)^{-a}\frac{dx}{x(1-x)}\ \cdot
    \end{equation}
    It can be expressed in terms of the Lauricella functions \eqref{IntroLSigma} for $\Sigma=\{\sigma_0,\sigma_1,\sigma_2\}$ with
    $$\sigma_0=0 \; , \quad \sigma_1=1\; ,\quad \sigma_2=y^{-1}\qquad \mbox{ and }\qquad s_0=b\; , \quad s_1=c-b\; ,\quad s_2=-a\  .$$
    Indeed, we find that
    \begin{equation}\label{eq: expression Fabcy Lauricella}
    \mathcal{F}(a,b,c;y) = \frac{c}{b(c-b)}L_{11} + \frac{1}{b} L_{12}
    \end{equation}
    where we use the shorthand notation 
    \begin{equation}\label{eq: shorthand Lij}
    L_{ij}=\left(L_{\{0,1,y^{-1}\}}(b,c-b,-a)\right)_{ij}\ .
    \end{equation}
    Equation \eqref{eq: expression Fabcy Lauricella} can be proven by  writing $\frac{dx}{x(1-x)} = \frac{dx}{x} + \frac{dx}{1-x}$, eliminating $\frac{dx}{x}$ using    $$b \,\frac{dx}{x} - (c-b) \,\frac{dx}{1-x} + a\,  \frac{y\, dx}{1-xy} = d\log\left(x^b(1-x)^{c-b}(1-yx)^{-a}\right)\ ,$$ and integrating by parts.

    \subsubsection{Contiguity relations}
    By  integrating by parts in the integral representation \eqref{eq: Fabcy} one proves the following ``contiguity relations'' for  $\mathcal{F}(a,b,c;y)$, which are valid when $\Real(c)>\Real(b)>0$. 
    \begin{equation}
    \label{eq: contiguity}
    \left\{ \begin{aligned}
    \mathcal{F}(a,b,c;y) & =\frac{c}{b}\mathcal{F}(a,b+1,c+1;y) -\frac{a}{b}y\, \mathcal{F}(a+1,b+1,c+2;y) \\
    \mathcal{F}(a,b,c;y) & = \frac{c}{c-b}\mathcal{F}(a,b,c+1;y) +\frac{a}{c-b}y\,\mathcal{F}(a+1,b+1,c+2;y) 
    \end{aligned} \right.
    \end{equation}
    One can use these relations to prove that the integral  \eqref{eq: Fabcy} can be analytically continued as a holomorphic function of  $a$, $b$, $c$ in the domain $b,c-b\notin \mathbb{Z}_{\leq 0}$ (this  also follows from the  expression \eqref{eq: definition 2F1 series} and the properties of the beta function). For instance, the first relation extends $\mathcal{F}(a,b,c;y)$ to $\Real(c-b)>0, \Real(b)>-1$ and the second extends it to $\Real(c-b)>-1, \Real(b)>0$.
    
    \subsubsection{Laurent series expansion}
    When viewed as a function of $a, b, c$, the function $\mathcal{F}(a,b,c)$ can be renormalised around zero as in Proposition \ref{proprenorm}:
    \begin{equation*}\begin{split}
    \mathcal{F}(a,b,c;y)  = \frac{1}{b}+\frac{(1-y)^{-a}}{c-b}  +\int_0^1\Omega_{a,b,c} \ . 
    \end{split}\end{equation*}
    where the following form,  denoted by   $\Omega^{\ren,1}$ in   
    Definition \ref{definition: renormalisedforms}, 
    $$\Omega_{a,b,c}=x^b(1-x)^{c-b}(1-yx)^{-a}\frac{dx}{x(1-x)} - x^b\frac{dx}{x} - (1-x)^{c-b}(1-y)^{-a}\frac{dx}{1-x}$$
    is absolutely integrable on $(0,1)$ for $\Real(c)>\Real(b)>-1$. 
    After expanding it, one obtains a Laurent series in the variables $a$, $b$, $c-b$:
    \begin{equation}\label{eq: laurent expansion Fabc}\begin{split}
    \mathcal{F}(a,&b,c;y) = \frac{1}{b}+\frac{(1-y)^{-a}}{c-b} + \sum_{\substack{i,j,k\geq 0\\ (j,k)\neq (0,0)}} \frac{b^i}{i!}  \frac{(c-b)^j}{j!} \frac{(-a)^k}{k!} \int_0^1\log^i(x)  \log^j(1-x) \log^k(1-yx) \frac{dx}{x}  \\
    & + \sum_{\substack{i,j,k\geq 0\\ (i,k)\neq (0,0)}} \frac{b^i}{i!}  \frac{(c-b)^j}{j!} \frac{(-a)^k}{k!} \int_0^1  \log^i(x)  \log^j(1-x) \left(\log^k(1-yx)-\mathbf{1}_{i=0}\log^k(1-y)\right)\frac{dx}{1-x} \ \cdot
    \end{split}
    \end{equation}
    The first integral converges at $x=0$ since at least one of the terms $\log^j(1-x)$ or $\log^k(1-yx)$ vanishes at $x=0$; the second integral converges at $x=1$ for similar reasons.
 
    \subsubsection{The companion functions $G$ and $\mathcal{G}$}\label{subsubsec: def Gcal}
    
    The formulae involving  $F$ and $\mathcal{F}$ frequently involve companion functions $G$ and $\mathcal{G}$ which  we now introduce. We first define
    \begin{equation} \label{Gabcasintegral}
      \mathcal{G}(a,b,c;y)= \int_{\infty}^{y^{-1}} x^b (1-x)^{c-b} (1-yx)^{-a} \, \frac{dx}{x(1-x)}  
    \end{equation}
    for all $a,b,c$ such that $\mathrm{Re}(c)<\mathrm{Re}(a)+1<2$. When discussing specific branches, under the assumption that $y\notin \mathbb{R}_{>0}$, we adopt the convention that the above integral is computed along the path given by $x= y^{-1}/t$ for $t\in (0,1)$.  Let us set $\log(-1)=\pi i$ for our fixed choice of $i$. This determines a branch of the complex logarithm on $\mathbb{C}\setminus [0,\infty)$, and in particular a value of $\log(y^{-1})$. (Beware that this branch satisfies $\log(y^{-1})=-\log(y)+2\pi i$.) We fix a branch of $x^b(1-x)^{c-b}(1-yx)^{-a}$ along $(\infty,y^{-1})$ by fixing branches of $\log(y^{-1}/t)$, $\log(1-y^{-1}/t)$, $\log(1-yy^{-1}/t)=\log(1-1/t)$ on $(0,1)$. We set $\log(y^{-1}/t)=\log(y^{-1})-\log(t)$; $\log(1-y^{-1}/t)=\log(-1)+\log(y^{-1})-\log(t)+\log(1-yt)$, where $\log(1-yt)$ equals $0$ if $t=0$; and finally  $\log(1-1/t)=\log(-1)+\log(1-t)-\log(t)$.
    
    We have the following expression for  $\mathcal{G}$ in terms of $\mathcal{F}$:
    $$\mathcal{G}(a,b,c;y)= e^{\pi i(c-a-b)}y(y^{-1})^c\,\mathcal{F}(1+b-c,1+a-c,2-c;y)\ ,$$
    which follows by making the change of variables $x=y^{-1}/t$ in \eqref{Gabcasintegral} and comparing with \eqref{eq: Fabcy}. This proves that $\mathcal{G}(a,b,c;y)$ extends to a holomorphic function of the variables $a,b,c$ in the domain $a,c-a\notin\mathbb{Z}_{\geq 1}$
    and a multi-valued holomorphic function of $y\in\mathbb{C}\setminus\{0,1\}$. It is a solution of the hypergeometric differential equation \eqref{eq: 2F1 differential equation}. Note that $\mathcal{G}(a,b,c;y)$ is not meromorphic at $y=0$ because of the prefactor $(y^{-1})^c$.
    
        The following expression of $\mathcal{G}$ in terms of the Lauricella functions \eqref{eq: shorthand Lij} will be proved below (\S\ref{subsubsec: pairings 2F1}) using homological intersection pairings:
    \begin{equation}\label{eq: Gcal in terms of L}
    \begin{split}
    \mathcal{G}(a,b,c;y) = \frac{c\, e^{2\pi i(c-b)}}{b(e^{2\pi ia}-e^{2\pi ic})}\Bigg(\frac{e^{2\pi ic}-e^{2\pi ib}}{c-b}L_{11}+\frac{e^{2\pi ic}-e^{2\pi ib}}{c}L_{12}& \\  - \frac{e^{2\pi ic}-1}{c-b}&L_{21}-\frac{e^{2\pi ic}-1}{c}L_{22}\Bigg)\ .
    \end{split}
    \end{equation}
    Here we choose the path $\delta_2$ to be the straight path between $0$ and $y^{-1}$. The branch of $\log(x)$ on this path is our chosen branch of the complex logarithm on $\mathbb{C}\setminus [0,\infty)$; the branches of $\log(1-x)$ and $\log(1-yx)$ are chosen so that they vanish when $x=0$.
    
    Since the integral \eqref{Gabcasintegral} is convergent near $a=b=c=0$, it  has the  Taylor expansion:
    \begin{equation}\label{eq: laurent expansion Gabc}
    \mathcal{G}(a,b,c;y) = \sum_{i,j,k\geq 0} \frac{b^i}{i!}\frac{(c-b)^j}{j!}\frac{(-a)^k}{k!}\int_\infty^{y^{-1}}\log^i(x)\log^j(1-x)\log^k(1-yx)\frac{dx}{x(1-x)}\ .
    \end{equation}
    
    We introduce the following normalisation of the function $\mathcal{G}(a,b,c;y)$:
    \begin{equation}\label{eq: G in terms of calG}
    G(a,b,c;y) =  \frac{\sin(\pi a)\sin(\pi (c-a))}{\pi \sin(\pi c)}\, \beta(b,c-b)^{-1}\, \mathcal{G}(a,b,c;y) \ .
    \end{equation}
    The prefactor is chosen so that $G(a,b,c;y)$ is symmetric in $a$ and $b$. Indeed, it can be expressed in terms of $F$ as:
   \begin{equation} \label{Gabcdef} G(a,b,c;y) = e^{\pi i(c-a-b)}  \frac{y(y^{-1})^c}{1-c}\,\frac{F(1+b-c,1+a-c,2-c;y)}{\beta(a,c-a)\,\beta(b,c-b)} \ .\end{equation} 
    This expression can be derived from the identity 
    \begin{equation}\label{eq: identity beta use 2F1}
    \beta(1+a-c,1-a) = (1-c)^{-1}\,\beta(a,c-a)^{-1}\frac{\pi \sin(\pi c)}{\sin(\pi a)\sin(\pi (c-a))}\ ,
    \end{equation}
    which easily follows from the functional equations of the gamma function.
 
\subsection{Cohomology with coefficients}\label{subsec: cohomology with coefficients 2F1}

    As in \S\ref{sec: cohomology coeff} we can view $\mathcal{F}(a,b,c;y)$ as a ``period'' of cohomology with coefficients. We work with parameters $a,b,c\in\mathbb{C}$ that are generic in the sense that 
    \begin{equation}\label{eq: generic abc} 
    a,b,c,c-a,c-b\notin\mathbb{Z}\ .
    \end{equation}
   This is the genericity condition \eqref{generic} plus the requirement that $c\notin\mathbb{Z}$ which we add in order to be able to work with preferred bases of cohomology with coefficients.
    Let $k\subset \mathbb{C}$ be a subfield and let us fix $y\in k\setminus \{0,1\}$. We consider the coefficient fields $k_{\mathrm{dR}}=k(a,b,c)$ and $\mathbb{Q}_{\mathrm{B}}=\mathbb{Q}(e^{2\pi ia}, e^{2\pi ib},e^{2\pi ic})$ and work in the corresponding category $\mathcal{T}$ as in \S\ref{subsubsec: tannakian T minimalist}. (We could work in the more refined setting of the category $\mathcal{T}_\infty$ as in \S\ref{subsubsec: tannakian infinity} without any substantial changes.) We write $M_{a,b,c}(y)$ for the object $M_{\{0,1,y^{-1}\}}(b,c-b,-a)$ in $\mathcal{T}$; it has rank $2$ and we now describe its de Rham and Betti components   along with the comparison between the two, and the intersection pairings.
    
    \subsubsection{de Rham} The de Rham component $M_{a,b,c}(y)_{\mathrm{dR}}$ is the first hypercohomology group (in the Zariski topology) of $X=\mathbb{A}^1_{k_{\mathrm{dR}}}\setminus \{0,1,y^{-1}\}$ with coefficients in the twisted de Rham complex $(\Omega^\bullet_X,\nabla_{a,b,c})$ where
    $$\nabla_{a,b,c}(1) = b \, d\log(x) + (c-b)\,d\log(1-x) - a\,d\log(1-yx)\ .$$
        Concretely, it is spanned by the (classes of) $d\log(x)$, $d\log(1-x)$, $d\log(1-yx)$, modulo the relation $\nabla_{a,b,c}(1)=0$. Since we are interested in \eqref{eq: Fabcy} we choose to work with the \emph{ad hoc} basis of $M_{a,b,c}(y)_{\mathrm{dR}}$ consisting of (the classes of) the logarithmic forms
    \begin{equation} \label{etaformsdefinitionsection9}
    \eta_1=\frac{dx}{x(1-x)} \quad \mbox{ and } \quad \eta_2=\frac{dx}{1-yx}\ \cdot
    \end{equation}
    They are expressed in terms of the basis \eqref{dRcoeffBasis} via the change of basis matrix
    \begin{equation}  \label{COBetasandomegas} \left(\begin{matrix}\eta_1 \\ \eta_2\end{matrix}\right) = \left( \begin{matrix} \frac{c}{b(c-b)}& \frac{1}{b} \\ 0 & -\frac{1}{ay} \end{matrix}\right) \left(\begin{matrix} (b-c)\,\omega_1\\ a\,\omega_2\end{matrix}\right)
    \end{equation} 
    where $\omega_1=d\log(1-x)$ and $\omega_2=d\log(1-yx)$. This matrix is invertible because $c\neq 0$ by \eqref{eq: generic abc} and therefore  $(\eta_1,\eta_2)$ indeed forms a basis of $M_{a,b,c}(y)_{\mathrm{dR}}$.

    \subsubsection{Betti} The Betti component $M_{a,b,c}(y)_{\mathrm{B}}$ is the first singular cohomology group of $\mathbb{C}\setminus \{0,1,y^{-1}\}$ with coefficients in the $\mathbb{Q}_{\mathrm{B}}$-local system $\mathcal{L}_{a,b,c}$ whose local sections are branches of the function $x^{-b}(1-x)^{-(c-b)}(1-yx)^{+a}$. It is convenient to work with locally finite homology via the isomorphism
    \begin{equation}\label{eq: lf iso 2F1} 
    M_{a,b,c}(y)_{\mathrm{B}}^\vee \stackrel{\sim}{\longrightarrow} H_1^{\mathrm{lf}}(\mathbb{C}\setminus \{0,1,y^{-1}\},\mathcal{L}_{a,b,c}^\vee)\ .
    \end{equation}
    We choose to work with a basis of $M_{a,b,c}(y)_{\mathrm{B}}^\vee$ which has the simplest possible intersection pairing and consider the two locally finite paths defined by the following parametrisations:
    $$p_1:(0,1) \to \mathbb{C}\setminus \{0,1,y^{-1}\} \;\; ,\;\;  t\mapsto t \qquad \mbox{and} \qquad  p_2:(0,1) \to \mathbb{C}\setminus \{0,1,y^{-1}\} \;\; ,\;\;  t\mapsto  \frac{y^{-1}}{t}\ \cdot$$
    Note that for this to make sense we have to assume as before that $y\notin (1,+\infty)$ so that $p_1(t)\neq y^{-1}$, $p_2(t)\neq 1$ for all $t\in (0,1)$. Note that $p_1$ is a locally finite path from $0$ to $1$, $p_2$ is a locally finite path from $\infty$ to $y^{-1}$, and the images of $p_1$ and $p_2$ have disjoint closures. We denote by $\varphi_1$ the class in $M_{a,b,c}(y)_{\mathrm{B}}^\vee$ corresponding via \eqref{eq: lf iso 2F1} to the path $p_1$ with the canonical branch of $x^b(1-x)^{c-b}(1-yx)^{-a}$ defined in \S\ref{subsubsec: def Fcal}. We denote by $\varphi_2$ the class in $M_{a,b,c}(y)_{\mathrm{B}}^\vee$ corresponding via \eqref{eq: lf iso 2F1} to the path $p_2$, together  with the branch of $x^b(1-x)^{c-b}(1-yx)^{-a}$ defined in \S\ref{subsubsec: def Gcal}.
    
    In order to justify the fact that $(\varphi_1,\varphi_2)$ indeed forms a basis of $M_{a,b,c}(y)_{\mathrm{B}}^\vee$ we use the homological Betti intersection pairing
    $$\langle\;,\,\rangle_{\mathrm{B}}:M_{-a,-b,-c}(y)_{\mathrm{B}}^\vee\otimes_{\Q_{\mathrm{B}}} M_{a,b,c}(y)_{\mathrm{B}}^\vee \longrightarrow \Q_{\mathrm{B}}\ ,$$
    discussed in \S\ref{subsubsec: pairings 2F1} below. We denote by $\varphi_i^-$ the class $\varphi_i$ viewed in $M_{-a,-b,-c}(y)_{\mathrm{B}}^\vee$, and by $\varphi_i^+$ the class $\varphi_i$ viewed in $M_{a,b,c}(y)_{\mathrm{B}}^\vee$, for $i=1,2$. We compute in the proof of Lemma \ref{lem: B pairing 2F1} below:
    $$\langle\varphi_1^-,\varphi_1^+\rangle_{\mathrm{B}} = \frac{1}{2i}\frac{\sin(\pi c)}{\sin(\pi b)\sin(\pi(c-b))} \quad \mbox{ and } \quad \langle\varphi_2^-,\varphi_2^+\rangle_{\mathrm{B}} = - \frac{1}{2i}\frac{\sin(\pi c)}{\sin(\pi a)\sin(\pi(c-a))} \ ,$$
    which are non-zero because $c\notin\mathbb{Z}$ by assumption \eqref{eq: generic abc}, and  $\langle\varphi_1^-,\varphi_2^+\rangle_{\mathrm{B}}=0$. It follows that  the classes $\varphi_1^+$ and $\varphi_2^+$ are linearly independent in $M_{a,b,c}(y)_{\mathrm{B}}^\vee$.
    
    \subsubsection{Period matrix}
    
    We let $P_{a,b,c}(y)$ denote the matrix of the natural comparison isomorphism
    $$\comp_{\mathrm{B},\mathrm{dR}}: M_{a,b,c}(y)_{\mathrm{dR}}\otimes_{k_{\mathrm{dR}}} \mathbb{C} \longrightarrow M_{a,b,c}(y)_{\mathrm{B}}\otimes_{\Q_{\mathrm{B}}} \mathbb{C}$$
    in the bases $(\eta_1,\eta_2)$ and $(\varphi_1,\varphi_2)$. It is naturally described in terms of the functions $\mathcal{F}$ and $\mathcal{G}$.
    
    \begin{prop}\label{prop: period matrix Fabcy}
    For all generic values of $a,b,c$, the period matrix of $M_{a,b,c}(y)$ reads:
    $$P_{a,b,c}(y)=\renewcommand*{\arraystretch}{1.5}\begin{pmatrix} \mathcal{F}(a,b,c;y) & \mathcal{F}(1+a,1+b,2+c;y) \\ \mathcal{G}(a,b,c;y) & \mathcal{G}(1+a,1+b,2+c;y) \end{pmatrix}\ .$$
    \end{prop}
    
    \begin{proof}
    By the argument in Lemma  \ref{lem:comparison equals Lauricella coefficients}, the $(i,j)$-th  entry of the period matrix $P_{a,b,c}(y)_{i,j}$ is given explicitly by an integral of the form
    $$ \langle \varphi_i, \mathrm{comp}_{\mathrm{B},\mathrm{dR}}(\eta_j)\rangle = \int_{p_i}x^b(1-x)^{c-b}(1-yx)^{-a}\,\eta_j$$
    whenever  $a,b,c$ lie in the region where the integral converges.  
    Note that there exist no  values of $a,b,c$ for which all four entries of the period matrix are simultaneously convergent.     We first verify, therefore, that each individual entry of the period matrix is as stated for a  restricted range of values of $a,b,c$. We then explain  how to extend this for all generic values of  $a,b,c.$

    Firstly, the top left entry $(i=j=1)$ is defined for all generic $a,b,c$  with $\Real(c)>\Real(b)>0$   by $\mathcal{F}(a,b,c;y)$ by definition \eqref{eq: Fabcy}. For the bottom left entry $(i=2,j=1)$ use the integral formula
    \eqref{Gabcasintegral}, valid in the region $\mathrm{Re}(c)<\mathrm{Re}(a)+1<2$. In order to deduce formulae for the entries in  the right-hand column $j=2$, use the fact that 
    $$\eta_2= \frac{dx}{1-yx} = x(1-x) (1-xy)^{-1} \frac{dx}{x (1-x)}$$
    which amounts to shifting $(b,c-b,-a) \mapsto (b+1,c-b+1,-a-1)$ in the arguments of $\mathcal{F}$ or $\mathcal{G}$ respectively, i.e., $(a,b,c) \mapsto (a+1,b+1,c+2).$  Thus the top right entry $(i=1,j=2)$ is valid for $\Real(c)+1>\Real(b)>-1$  and the bottom right $(i=j=2)$ for  $\mathrm{Re}(c)<\mathrm{Re}(a)<0$. 
     
    Finally, to show that these formulae remain valid for generic values of $a,b,c$ we use the fact that the entries of $P_{a,b,c}(y)$, as well as the functions $\mathcal{F}(a,b,c;y)$ and $\mathcal{G}(a,b,c;y)$ satisfy contiguity relations. This follows from the fact that multiplication by $x$, $(1-x)$ and $(1-xy)$ can be expressed as cohomology relations. For example, using the fact that 
    $$\nabla_{a,b,c}(1) = b \frac{dx}{x} + (b-c) \frac{dx}{1-x} + a \frac{y \, dx}{1-xy} $$
    we find that  the following relation holds in  $M_{a,b,c}(y)_{\mathrm{dR}}$:
    $$[x \eta_1] = \left[ \frac{dx}{1-x} \right] = \frac{b}{c} [\eta_1] + \frac{ay}{c} [\eta_2] $$
    since both sides differ by  $\frac{1}{c} \nabla_{a,b,c}(1)$. Therefore 
    $$P_{a,b+1,c+1}(y)_{1,1} = \frac{b}{c} P_{a,b,c}(y)_{1,1} +\frac{ay}{c} P_{a,b,c}(y)_{1,2}  $$
    From this and similar relations for the other entries, we deduce that all entries of $P_{a,b,c}(y)$ extend to a holomorphic function of $(a,b,c)$ on the domain where they are generic. Since the same is true for $\mathcal{F}, \mathcal{G}$, we conclude by analytic continuation.
    \end{proof}

    \subsubsection{Intersection pairings and twisted period relations}\label{subsubsec: pairings 2F1}
    
    The de Rham intersection pairing
    \begin{equation}\label{eq: dR pairing 2F1}
    \langle\;,\,\rangle^{\mathrm{dR}} : M_{-a,-b,-c}(y)_{\mathrm{dR}}\otimes_{k_{\mathrm{dR}}} M_{a,b,c}(y)_{\mathrm{dR}} \longrightarrow k_{\mathrm{dR}}
    \end{equation}
    is easily computed. Let us denote by $\eta_i^-$ the class $\eta_i$ viewed in $M_{-a,-b,-c}(y)_{\mathrm{dR}}$ and by $\eta_i^+$ the class $\eta_i$ viewed in $M_{a,b,c}(y)_{\mathrm{dR}}$, for $i=1,2$. 
    
    \begin{lem}\label{lem: dR pairing 2F1}
    The matrix of the de Rham intersection pairing \eqref{eq: dR pairing 2F1} in the bases $(\eta_1^-,\eta_2^-)$ and $(\eta_1^+,\eta_2^+)$ is as follows:
    $$I^{\mathrm{dR}}_{a,b,c}(y)=\left(\begin{matrix} -\frac{c}{b(c-b)} & 0 \\ 0 & \frac{1}{y^2}\frac{c}{a(c-a)} \end{matrix}\right)\ .$$
    \end{lem}
    
    \begin{proof}
    We check that
    $$\begin{pmatrix}\eta^-_1 \\ \eta^-_2\end{pmatrix}= \begin{pmatrix} -1 & 0 \\ \frac{(b-c)}{y(a-c)} & \frac{c}{y(a-c)} \end{pmatrix} \begin{pmatrix} \nu_1 \\ \nu_2 \end{pmatrix}\  \quad \hbox{ and } \quad 
    \left(\begin{matrix}\eta^+_1 \\ \eta^+_2\end{matrix}\right) = \left( \begin{matrix} \frac{c}{b(c-b)}& \frac{1}{b} \\ 0 & -\frac{1}{ay} \end{matrix}\right) \left(\begin{matrix} (b-c)\,\omega_1\\ a\,\omega_2\end{matrix}\right)
        $$
    where the second equation is \eqref{COBetasandomegas}.  The result follows from  Lemma \ref{lem: omega nu dR pairing} which states in this case that $\langle \nu_1, \omega_1\rangle^{\mathrm{dR}}= (b-c)^{-1}$, $\langle \nu_1, \omega_2\rangle^{\mathrm{dR}} = \langle \nu_2, \omega_1\rangle^{\mathrm{dR}} =0$ and $\langle \nu_2, \omega_2 \rangle^{\mathrm{dR}} = a^{-1}.$
     \end{proof}

    We now turn to the (cohomological) Betti intersection pairing 
    \begin{equation}\label{eq: B pairing 2F1}
    \langle\;,\,\rangle^{\mathrm{B}} : M_{-a,-b,-c}(y)_{\mathrm{B}}\otimes_{\mathbb{Q}_{\mathrm{B}}} M_{a,b,c}(y)_{\mathrm{B}} \longrightarrow \mathbb{Q}_{\mathrm{B}}
    \end{equation}
    Let us denote by $\varphi_i^-$ the class $\varphi_i$ viewed in $M_{-a,-b,-c}(y)_{\mathrm{B}}^\vee$, and by $\varphi_i^+$ the class $\varphi_i$ viewed in $M_{a,b,c}(y)_{\mathrm{B}}^\vee$, for $i=1,2$.
    
    \begin{lem}\label{lem: B pairing 2F1}
    The matrix of the Betti intersection pairing \eqref{eq: B pairing 2F1} in the bases $(\varphi_1^-,\varphi_2^-)$ and $(\varphi_1^+,\varphi_2^+)$ is as follows:
    $$I^{\mathrm{B}}_{a,b,c}(y)=\left(\begin{matrix} 2i\frac{\sin(\pi b)\sin(\pi(c-b))}{\sin(\pi c)} & 0 \\ 0 & -2i \frac{\sin(\pi a)\sin(\pi (c-a))}{\sin(\pi c)}\end{matrix}\right)\ .$$
    \end{lem}
    
    \begin{proof}
    By the same computation as in \cite{kitayoshida}  one easily computes the homological Betti pairing:
    $$\langle\varphi_1^-,\varphi_1^+\rangle_{\mathrm{B}} = \frac{1}{2i}\frac{\sin(\pi c)}{\sin(\pi b)\sin(\pi(c-b))} \quad , \quad \langle\varphi_2^-,\varphi_2^+\rangle_{\mathrm{B}}=-\frac{1}{2i}\frac{\sin(\pi c)}{\sin(\pi a)\sin(\pi(c-a))} \ ,$$
    and $\langle\varphi_i^-,\varphi_j^+\rangle_{\mathrm{B}}=0$ for $i\neq j$ since  the closures of  $p_1$ and $p_2 $ do not intersect. In fact, the top left entry  reduces to a statement equivalent to \cite{mimachiyoshida}, equation (7), and the bottom right can be deduced from it by applying  $z\mapsto z/y$. One  deduces  the matrix of the cohomological Betti pairing by inverting the matrix $(\langle\varphi_i^-,\varphi_j^+\rangle_{\mathrm{B}})$ and the claim follows.
    \end{proof}
    
    The compatibility between the intersection pairings and the comparison isomorphism gives rise to the \emph{twisted period relations}:
    \begin{equation}\label{eq: twisted period relations 2F1 matrix} 
    {}^tP_{-a,-b,-c}(y) \, I^{\mathrm{B}}_{a,b,c}(y) \, P_{a,b,c}(y) = 2\pi i\, I^{\mathrm{dR}}_{a,b,c}(y)\ .
    \end{equation}
    Reinserting the beta factor \eqref{eq: 2F1 integral} leads to quadratic relations for the hypergeometric function $F$ as in \cite[\S 4]{ChoMatsumoto}. For example, the bottom left entry is equivalent to Gauss' relation
    $$F(a,b,c;y)\,F(1-a,1-b,2-c;y) = F(-a+c,-b+c,c;y)\,F(1+a-c, 1+b-c, 2-c;y)\ .$$
    
     We end this paragraph with a proof of the expression \eqref{eq: Gcal in terms of L} of $\mathcal{G}$ in terms of Lauricella functions.
    
    \begin{proof}[Proof of \eqref{eq: Gcal in terms of L}]
    We denote by $\delta_1^\pm$ (resp. $\delta_2^\pm$) the class in $M_{\pm a, \pm b,\pm c}(y)_{\mathrm{B}}^\vee$ of the path $(0,1)$ (resp. of the path $(0,y^{-1})$) together with the branch of $x^b(1-x)^{c-b}(1-yx)^{-a}$ defined in \S\ref{subsubsec: def Fcal} (resp. in \S\ref{subsubsec: def Gcal}). By the same computation as in \cite{kitayoshida} we see that the matrix of the Betti intersection pairing \eqref{eq: B pairing 2F1} in the bases $(\delta_1^-,\delta_2^-)$ and $(\delta_1^+,\delta_2^+)$ is as follows:
    $$J = \frac{1}{e^{2\pi i b}-1}\left(\begin{matrix} \frac{e^{2\pi ic}-1}{e^{2\pi i(c-b)}-1} & e^{2\pi ib} \\ 1 & \frac{e^{2\pi i(b-a)}-1}{e^{-2\pi ia}-1} \end{matrix}\right)\ .$$
    Let us denote by $\psi$ the class  in $M_{a,b,c}(y)_{\mathrm{B}}^\vee$ of the path from $\infty$ to $y^{-1}$ together with the branch of $x^b(1-x)^{c-b}(1-yx)^{-a}$ defined in \S\ref{subsubsec: def Gcal}. By the same computation as in \cite{kitayoshida} we have $\langle \delta_1^-,\psi\rangle_{\mathrm{B}} = 0$ because $(0,1)$ and $(\infty,y^{-1})$ do not intersect, and 
    $$\langle \delta_2^-,\psi\rangle_{\mathrm{B}}=\frac{e^{2\pi i(c-b-a)}}{e^{-2\pi ia}-1}\ \cdot$$
    Therefore we have $\psi=u_1\delta_1^+ + u_2\delta_2^+$ where
    $$ \left(\begin{matrix}u_1\\ u_2\end{matrix}\right)=J^{-1}\left(\begin{matrix}0 \\ \frac{e^{2\pi i(c-b-a)}}{e^{-2\pi ia}-1} \end{matrix}\right) = \frac{e^{2\pi i(c-b-a)}}{e^{2\pi ic}-e^{2\pi i a}}\left(\begin{matrix} e^{2\pi ib}-e^{2\pi i c} \\ e^{2\pi i c}-1 \end{matrix}\right)\ .$$
    On the other hand, we have the identity in $M_{a,b,c}(y)_{\mathrm{dR}}$:
    $$\eta_1 = \frac{c}{b(c-b)}(b-c)\omega_1 + \frac{1}{b}a\omega_2\ ,$$
    which follows from \eqref{COBetasandomegas}. This leads to an expression of $\mathcal{G}(a,b,c;y)=\langle\psi,\mathrm{comp}_{\mathrm{B},\mathrm{dR}}(\eta_1)\rangle$ in terms of the $L_{ij}=\langle \delta_i,\mathrm{comp}_{\mathrm{B},\mathrm{dR}} (-s_j\omega_j)\rangle$ for $1\leq i,j\leq 2$. One checks that it is given by \eqref{eq: Gcal in terms of L}.
    \end{proof}

\subsection{The single-valued period matrix and the double copy formula}

    We have the \emph{single-valued period map} 
    $$\s:M_{a,b,c}(y)_{\mathrm{dR}}\otimes_{k_{\mathrm{dR}}}\mathbb{C}\longrightarrow M_{-a,-b,-c}(y)_{\mathrm{dR}}\otimes_{k_{\mathrm{dR}}}\overline{\mathbb{C}}$$
    where $\overline{\mathbb{C}}$ denotes $\mathbb{C}$ with the conjugate structure of a $k_{\mathrm{dR}}$-algebra (it amounts to conjugating $y$). As discussed in \S\ref{subsubsec: def sv}, strictly speaking, this only makes sense if $a,b,c$ are real. In general one would need to replace the target of $\mathbf{s}$ with the analytic de Rham cohomology of $\mathbb{C}\setminus\overline{\Sigma}$ with differential $\nabla_{-a,-b,-c}$, which does not have a natural $k_{\mathrm{dR}}$-structure. However, all the formulae below make sense in this setting and we will continue to treat $a,b,c$ as complex numbers.

    Let us denote by $S_{a,b,c}(y)$ the single-valued period matrix of $M_{a,b,c}(y)$, i.e., the matrix of $\s$, in the bases $(\eta_1^+,\eta_2^+)$ in the source and $(\eta_1^-,\eta_2^-)$ in the target. From the recipe given in Remark \ref{rem: compute frobenius betti} one sees that the matrix of the real Frobenius in the bases $(\varphi_1^+,\varphi_2^+)$ and $(\varphi_1^-,\varphi_2^-)$ is the identity matrix. Therefore, the definition of $\s$ reads, in matrix form:
    \begin{equation}\label{eq: def sv 2F1 matrix}
    S_{a,b,c}(y) = P_{-a,-b,-c}(\overline{y})^{-1}\, P_{a,b,c}(y)\ .
    \end{equation}
    This leads directly to formulae for the entries of the single-valued period matrix $S_{a,b,c}(y)$, using the fact that the determinant  of the period matrix can be computed explicitly (see, e.g., \cite{loesersabbah}). In more detail, we find that in the bases used here we have
    \begin{equation} \det \left( P_{a,b,c}(y) \right)  = e^{\pi i(c-a-b)}\,y^{-c-1} (1-y)^{c-a-b}\, \beta(b,c-b)\beta(-a,-c+a) \ . 
    \end{equation}
    It follows then from \eqref{eq: def sv 2F1 matrix} that the top left entry of $S_{a,b,c}(y)$ is the single-valued function:
    \begin{multline} S_{a,b,c}(y)_{1,1}= e^{-\pi i(c-a-b)}\overline{y}^{1-c}(1-\overline{y})^{c-a-b}\,\beta(a,c-a)^{-1}\,\beta(-b,-c+b)^{-1}\\ \times \Bigg(\mathcal{F}(a,b,c;y)\, \mathcal{G}(1-a,1-b,2-c;\overline{y})  \nonumber
    - \mathcal{G}(a,b,c;y)\, \mathcal{F}(1-a,1-b,2-c;\overline{y})\Bigg)\ .
    \end{multline} 
    However, a more symmetric looking formula, called \emph{double copy formula}, is obtained by using the twisted period relations, which is the approach which we shall adopt  henceforth.

    \subsubsection{The single-valued period matrix}
    
    We provide formulae (see \eqref{eq: formula sv period matrix 2F1} below) for the single-valued periods of $M_{a,b,c}(y)$ after composing $\s$ with the isomorphism
    \begin{equation}\label{eq: intersection pairing iso 2F1}
    M_{-a,-b,-c}(y)_{\mathrm{dR}}\otimes_{k_{\mathrm{dR}}}\overline{\mathbb{C}}\longrightarrow M_{a,b,c}(y)_{\mathrm{dR}}^\vee \otimes_{k_{\mathrm{dR}}}\overline{\mathbb{C}}
    \end{equation}
    obtained from the de Rham intersection pairing \eqref{eq: dR pairing 2F1}. Note that the matrix of \eqref{eq: intersection pairing iso 2F1} is 
    \begin{equation}\label{eq: intersection pairing sign 2F1}
    I^{\mathrm{dR}}_{-a,-b,-c}(\overline{y})=-\,{}^tI^{\mathrm{dR}}_{a,b,c}(\overline{y})\ .
    \end{equation} 
    \begin{prop}\label{prop: integral formula sv 2F1}
    The  matrix $I^{\mathrm{dR}}_{-a,-b,-c}(\overline{y})\, S_{a,b,c}(y)$ has entries given by the complex integrals
    \begin{equation}\label{eq: integral formula sv 2F1}
    \langle\eta_i,\s\eta_j\rangle^{\mathrm{dR}} = -\frac{1}{2\pi i}\iint_{\mathbb{C}} |z|^{2b}|1-z|^{2(c-b)}|1-yz|^{-2a} \,\overline{\eta_i}\wedge\eta_j\ ,
    \end{equation}
    for $1\leq i,j\leq 2$, which converge  for $a,b,c$ in the domains
    $$\begin{cases} 0< \Real(b)<\Real(c)<\Real(a)+1<2 & \mbox{ for } (i,j)=(1,1)\ ;\\
    -\frac{1}{2} < \Real(b) <\Real(c)+\frac{1}{2} < \Real(a)+1 < \frac{3}{2}& \mbox{ for } (i,j)=(1,2) \mbox{ and } (i,j)=(2,1)\ ;\\
    -1<\Real(b)<\Real(c)+1<\Real(a)+1<1 & \mbox{ for } (i,j)=(2,2)\ .
    \end{cases}$$
    \end{prop}

    \begin{proof}
    This follows from Proposition \ref{propSVpairingCoeffs}.
    \end{proof}
    
    \begin{rem}\label{rem: contiguity sv 2F1}
    By Proposition \ref{prop: period matrix Fabcy} the entries of the period matrix $P_{a,b,c}(y)$ are holomorphic functions of generic arguments $a,b,c$. This implies that the entries of the single-valued period matrix $S_{a,b,c}(y)$ have the same property and that \eqref{eq: integral formula sv 2F1} extends to a holomorphic function of generic $a,b,c$. One could also prove this directly  by using variants of the contiguity relations \eqref{eq: contiguity}.
    \end{rem}
    
    \begin{defn}
    For generic values of $a,b,c$, the single-valued versions of $\mathcal{F}(a,b,c;y)$ and $\mathcal{G}(a,b,c;y)$ are given by:
    $$\mathcal{F}^{\,\s}(a,b,c;y) = \langle -\eta_1 , \s\eta_1 \rangle^{\mathrm{dR}}\ \qquad \mbox{ and }\qquad \mathcal{G}^{\,\s}(a,b,c;y) = \langle -y\eta_2,\s\eta_1\rangle^{\mathrm{dR}}\ .
    $$
    \end{defn}
    
    The heuristic for these formulae (see Remark \ref{rem: c zero}) is that the form $-\eta_1$ has simple poles and residues $-1$ at $0$ and $1$ at $1$, and is  the image under the map $c_0^\vee$ of \cite{BD1} of the class of a path from $0$ to $1$. Similarly, the form $-y\eta_2=d\log(1-yx)$ has simple poles and residues $-1$ at $\infty$ and $1$ at $y^{-1}$ and is thus the image under $c_0^\vee$ of the class of a path from $\infty$ to $y^{-1}$.

    \begin{rem}\label{rem: why Fs is sv of F}
    By expressing $\eta_1$ in  the basis $((b-c)\omega_1,a\omega_2)$ as in \eqref{COBetasandomegas} and noting that $-\eta_1$ equals the form $\nu_1$ defined in \eqref{eq:def nui} one gets the following expression:
    \begin{equation}\label{eq: expression Fabcy sv Lauricella}
    \mathcal{F}^{\,\s}(a,b,c;y) = \frac{c}{b(c-b)}L^\s_{11} + \frac{1}{b} L^\s_{12}\ .
    \end{equation} 
    where we have used the shorthand notation 
    $$L^\s_{ij} = \left(L^\s_{\{0,1,y^{-1}\}}(b,c-b,-a)\right)_{ij}\ .$$
    In view of the similar expression \eqref{eq: expression Fabcy Lauricella}, Theorem \ref{introthmperandsv} implies that $\mathcal{F}^{\,\s}(a,b,c;y)$ has a Laurent expansion in the variables $a,b,c-b$, whose coefficients are obtained by applying the de Rham projection and the single-valued period map to (motivic lifts of) the coefficients of the Laurent expansion \eqref{eq: laurent expansion Fabc} of $\mathcal{F}(a,b,c;y)$. This justifies the fact that we call $\mathcal{F}^{\,\s}$ a \emph{single-valued version} of $\mathcal{F}$. A similar computation, using the identity
    $$-y\eta_2 = - \frac{c-b}{c-a}\,\nu_1 + \frac{c}{c-a}\nu_2 $$
    in $M_{a,b,c}(y)_{\mathrm{dR}}$, leads to the expression 
    \begin{equation}\label{eq: expression Gabcy sv Lauricella}
    \mathcal{G}^\s(a,b,c;y) = \frac{c}{b(a-c)}\left( L^\s_{11} + \frac{c-b}{c}L^\s_{12} - \frac{c}{c-b}L^\s_{21} - L^\s_{22} \right)\ .
    \end{equation}
    By comparing it with \eqref{eq: Gcal in terms of L} and applying Theorem \ref{introthmperandsv} one deduces that $\mathcal{G}^\s(a,b,c;y)$ has a Laurent expansion in the variables $a,b,c-b$ whose coefficients are obtained by applying the de Rham projection and the single-valued period map to (motivic lifts of) the coefficients of the Laurent expansion of $\mathcal{G}(a,b,c;y)$. This is because the de Rham projection sends the Lefschetz element $\mathbb{L}^{\mathfrak{m}}$ of $\MT(k)$, which is a motivic lift of $2\pi i$, to zero.
    \end{rem}
    
    \begin{rem}\label{rem: Gs not in terms of Fs}
    Note that even though $\mathcal{G}(a,b,c;y)$ is essentially $\mathcal{F}(1+b-c,1+a-c,2-c;y)$, there does not seem to be a natural way to express $\mathcal{G}^{\,\s}$ in terms of $\mathcal{F}^{\,\s}$. 
    This should not be surprising since the single-valued versions of the functions $\mathcal{F}$ and $\mathcal{G}$ are related to the original functions by looking at Laurent expansions around $a=b=c=0$, and the Laurent expansion of $\mathcal{F}(1+b-c,1+a-c,2-c;y)$ is not directly related to that of $\mathcal{F}(a,b,c;y)$.
    \end{rem}

    Proposition \ref{prop: integral formula sv 2F1} gives the integral formulae:
    \begin{equation}\label{eq: expression Fcal sv integral}
    \mathcal{F}^{\,\s}(a,b,c;y)=\frac{1}{2\pi i}\iint_{\mathbb{C}}|z|^{2b}|1-z|^{2(c-b)}|1-zy|^{-2a}\frac{d\overline{z}\wedge dz}{|z|^2|1-z|^2}
    \end{equation}
    and
    \begin{equation}\label{eq: expression Gcal sv integral}
    \mathcal{G}^{\,\s}(a,b,c;y) = \frac{1}{2\pi i}\iint_{\mathbb{C}} |z|^{2b}|1-z|^{2(c-b)}|1-zy|^{-2a}\frac{\overline{y}\,d\overline{z}\wedge dz}{(1-\overline{y}\,\overline{z})z(1-z)}\ ,
    \end{equation}
    that are valid in the domains $0<\Real(b)<\Real(c)<\Real(a)+1<2$ and $-\frac{1}{2} < \Real(b) <\Real(c)+\frac{1}{2} < \Real(a)+1 < \frac{3}{2}$ respectively.
    
    In view of Remark \ref{rem: contiguity sv 2F1}, $\mathcal{F}^{\,\s}$ and $\mathcal{G}^{\,\s}$ are holomorphic functions of the generic complex parameters $a,b,c$. This will also be apparent in the double copy formulae \eqref{eq: double copy Fabc} and \eqref{eq: double copy Gabc} below. The single-valued period matrix $S_{a,b,c}(y)$ can now be written entirely in terms of $\mathcal{F}^{\,\s}$ and $\mathcal{G}^{\,\s}$ via:
    \begin{equation}\label{eq: formula sv period matrix 2F1}
    -I^{\mathrm{dR}}_{-a,-b,-c}(\overline{y})\,S_{a,b,c}(y) = \renewcommand*{\arraystretch}{1.5}\begin{pmatrix} \mathcal{F}^{\,\s}(a,b,c;y) & y^{-1}\,\mathcal{G}^{\,\s}(a,b,c;\overline{y}) \\ \overline{y}^{-1}\,\mathcal{G}^{\,\s}(a,b,c;y) & \mathcal{F}^{\,\s}(a+1,b+1,c+2;y) \end{pmatrix}\ ,
    \end{equation}
where the de Rham intersection matrix $I^{\mathrm{dR}}$ can be found in Lemma \ref{lem: dR pairing 2F1}.

    \subsubsection{Double copy formula}

    \begin{prop}
    We have the equality, for all generic values of $a,b,c$:
    $$I^{\mathrm{dR}}_{-a,-b,-c}(\overline{y})\, S_{a,b,c}(y) = \frac{1}{2\pi i}\, {}^tP_{a,b,c}(\overline{y})\,I^{\mathrm{B}}_{-a,-b,-c}(\overline{y})\, P_{a,b,c}(y)\ .$$
    \end{prop}
    
    \begin{proof}
    This formula follows directly from \eqref{eq: twisted period relations 2F1 matrix} and \eqref{eq: def sv 2F1 matrix}.
    \end{proof}
    
    The top left entry of the double copy formula reads:
    \begin{equation}\label{eq: double copy Fabc}\begin{split}
    \mathcal{F}^{\,\s}(a,b,c;y) = \frac{\sin(\pi b)\sin(\pi(c-b))}{\pi \sin(\pi c)}\, \mathcal{F}(a,b,&c;y)\,\mathcal{F}(a,b,c;\overline{y}) \\
    & - \frac{\sin(\pi a)\sin(\pi(c-a))}{\pi \sin(\pi c)}\, \mathcal{G}(a,b,c;y)\,\mathcal{G}(a,b,c;\overline{y})\ .
    \end{split}
    \end{equation}
    The bottom left entry is: 
    \begin{equation}\label{eq: double copy Gabc}\begin{split}
   \overline{y}^{-1}\,\mathcal{G}^{\,\s}(a,b,c;y) =   \frac{\sin(\pi b) \sin(\pi (c-b))}{\pi\sin(\pi c))} \, \mathcal{F}(a, b, c;y)\, \mathcal{F}(a&+1, b+1, c+2; \overline{y}) \\
    -\frac{\sin(\pi a) \sin(\pi (c-a))}{ \pi \sin(\pi c))} & \, \mathcal{G}(a, b, c; y)  \, \mathcal{G}(a+1, b+1, c+2; \overline{y}) \ .
    \end{split}
    \end{equation}
    The other entries are easily deduced from these two. 
   
    \begin{rem}Certain special cases of entries of  single-valued period matrix $S_{a,b,c}(y)$ were previously considered in \cite{mimachiyoshida, mimachicomplex}, along the lines of Remark \ref{re: sv is not conj of coeff}. 
    The presentation above seems to be the first systematic approach to constructing all the single-valued periods associated to the hypergeometric function.
    \end{rem}

\subsection{The single-valued hypergeometric function}
    Recall the formula:
    $$\beta^{\,\s}(s_0,s_1)=\frac{1}{2\pi i}\iint_{\mathbb{C}}|z|^{2s_0}|1-z|^{2s_1}\frac{d\overline{z}\wedge dz}{|z|^2|1-z|^2} = \frac{\Gamma(s_0)\,\Gamma(s_1)\,\Gamma(1-s_0-s_1)}{\Gamma(s_0+s_1)\,\Gamma(1-s_0)\,\Gamma(1-s_1)}$$
    for the single-valued (or `complex') version of the beta function \cite[\S 1.1]{BD2}. In view of \eqref{eq: 2F1 integral} we propose the following definition of a \emph{single-valued hypergeometric function}. We also define a single-valued of the function $G(a,b,c;y)$.
    
    \begin{defn}\label{def: F G sv}
    For generic values of $a,b,c$, the single-valued versions of $F$ and $G$    are given by:
    $$F^{\,\s}(a,b,c;y)=  \beta^{\,\s}(b,c-b)^{-1}\,\mathcal{F}^{\,\s}(a,b,c;y) \;\; \mbox{ and }\;\; G^{\,\s}(a,b,c;y) = \frac{a(c-a)}{c}\,\beta^{\,\s}(b,c-b)^{-1}\,\mathcal{G}^{\,\s}(a,b,c;y)\ .$$
    \end{defn}
    
    As explained in Remark \ref{rem: why Fs is sv of F}, the coefficients of the Laurent expansion of $F^{\,\s}(a,b,c;y)$ around $a=b=c=0$ are obtained by applying the de Rham projection and the single-valued period map to (motivic lifts of) the coefficients of the Laurent expansion of $F(a,b,c;y)$. The same statement is true for $G^{\,\s}$ and $G$ in view of \eqref{eq: G in terms of calG} (again because the Lefschetz element $\Lef^{\mathfrak{m}}$ of $\MT(k)$, which is a motivic lift of $2\pi i$, is sent to zero by the de Rham projection.) 
    There is no natural expression for $G^{\,\s}$ in terms of $F^{\,\s}$ for the same reasons as in Remark \ref{rem: Gs not in terms of Fs}.
    
    The next proposition expresses single-valued hypergeometric functions as a double copy of the classical hypergeometric functions. It shows in particular that $F^{\,\s}$ and $G^{\,\s}$ are holomorphic functions of generic complex numbers $a,b,c$ and (single-valued) analytic functions of $y\in\mathbb{C}\setminus \{0,1\}$.

    \begin{prop}   \label{prop: doublecopyfor2F1}
    For all generic $a,b,c$ we have the double copy formulae:
    \begin{equation*}
    F^{\,\s}(a,b,c;y) \, = \, F(a,b,c;y)\,F(a,b,c;\overline{y})  -  w_{a,c-a}\,w_{b,c-b} \, G(a,b,c;y)\,G(a,b,c;\overline{y})
    \end{equation*}
    and
    \begin{equation*} \begin{split} G^{\,\s}(a,b,c;y) =  \frac{a (c-a) b (c-b) \,  \overline{y}}{c^2(1+c)} \,
   \Big( F(a,b,c;y)\,  F(b+1,&a+1,c+2;\overline{y})  \\  - w_{a,c-a}& w_{b,c-b} G(a,b,c;y) \,G(b+1,a+1,c+2; \overline{y} \Big)\end{split}
    \end{equation*}
    where we have set
    $$w_{s,t}=\frac{\pi\sin(\pi(s+t))}{\sin(\pi s)\sin(\pi t)}\ \cdot$$
    \end{prop}
    
    \begin{proof}
    This follows from multiplying \eqref{eq: double copy Fabc} by $\beta^{\,\s}(b,c-b)^{-1}$ and using the identities \eqref{eq: identity beta use 2F1} and 
    $$\beta^{\,\s}(b,c-b) = \frac{\sin(\pi b)\sin(\pi(c-b))}{\pi\sin(\pi c)}\beta(b,c-b)^2\ .$$
    \end{proof}

    \begin{rem}
    It is obvious from the double copy formula  that both $F^{\,\s}$ and $G^{\, \s}$ satisfy the ``holomorphic part'' of the hypergeometric differential equation \eqref{eq: 2F1 differential equation}, namely:
    $$\left(y(1-y)\frac{\partial^2}{\partial y^2} + (c-(a+b+1)y)\frac{\partial}{\partial y} -ab \right)  f(y) =0\ ,$$
    for $f(y)  = F^{\,\s}(a,b,c;y)$ or $f(y)= G^{\, \s}(a,b,c;y)$. This  can also be derived from first principles. 
    \end{rem}

\section{Example: motivic coaction of \texorpdfstring{${}_2F_1$}{2F1}}
We derive formulae for the motivic coaction of the hypergeometric function $F={}_2F_1$ both in the global and in the local setting (Proposition \ref{prop: coaction 2F1} and Theorem \ref{thm: coaction local 2F1} respectively). As in the case of  Lauricella functions, these formulae turn out to be formally identical, even though the contexts in which they appear are very different.

\subsection{Motivic coaction of the hypergeometric function: the global point of view}
 As in the previous section, let $k\subset \mathbb{C}$ be a subfield and let us fix $y\in k\setminus \{0,1\}$. We consider the coefficient fields $k_{\mathrm{dR}}=k(a,b,c)$ and $\mathbb{Q}_{\mathrm{B}}=\mathbb{Q}(e^{2\pi ia}, e^{2\pi ib},e^{2\pi ic})$ and work in the Tannakian category $\mathcal{T}$ defined in  \S\ref{subsubsec: tannakian T minimalist}.
    We define the (global) motivic and canonical de Rham variants of $\mathcal{F}(a,b,c;y)$:
    $$\mathcal{F}^\mm(a,b,c;y) = [M_{a,b,c}(y),\varphi_1,\eta_1]^\mm \;\; \in \Pe_{\mathcal{T}}^\mm\ ,$$
    $$\mathcal{F}^\dR(a,b,c;y)=[M_{a,b,c}(y),-\eta_1,\eta_1]^\dR \;\; \in\Pe_{\mathcal{T}}^\dR\ .$$
    We will also need the (global) de Rham variant of $\mathcal{G}(a,b,c;y)$:
    $$\mathcal{G}^\dR(a,b,c;y)=[M_{a,b,c}(y),-y\eta_2,\eta_1]^\dR \;\; \in \Pe_{\mathcal{T}}^\dR\ , $$
   where $\eta_i$ and $\varphi_i$ were defined in \S\ref{subsec: cohomology with coefficients 2F1}.
        Applying the period map $\per:\Pe_{\mathcal{T}}^\mm\to \mathbb{C}$ to $\mathcal{F}^\mm(a,b,c;y)$ gives back $\mathcal{F}(a,b,c;y)$. If one works in the more refined Tannakian formalism of \S\ref{subsubsec: tannakian infinity} one can replace $\mathcal{F}^\dR(a,b,c;y)$ and $\mathcal{G}^\dR(a,b,c,;y)$ with elements in the ring $\mathcal{P}^{\mathrm{dR}^+,\mathrm{dR}^-}_{\mathcal{T}_\infty}$ whose single-valued periods are $\mathcal{F}^{\,\s}(a,b,c;y)$ and $\mathcal{G}^{\,\s}(a,b,c;y)$ respectively. We now compute the (global) motivic coaction $\Delta:\Pe^\mm_{\mathcal{T}}\to \Pe^\mm_{\mathcal{T}}\otimes_{k_{\mathrm{dR}}} \Pe^\dR_{\mathcal{T}}$.
    \begin{prop}\label{prop: coaction Fabc}
    We have the following (global) motivic coaction formula:
    \begin{equation*}\label{eq: coaction Fabc global}
    \begin{split}
    \Delta\,\mathcal{F}^\mm(a,b,c;y) = \frac{b(c-b)}{c}\, \mathcal{F}^\mm(a,b,c;y)\,\otimes\,  &\mathcal{F}^\dR(a,b,c;y) \\ & -\,  y\, \frac{a(c-a)}{c}\, \mathcal{F}^\mm(1+a,1+b,2+c;y)\otimes \mathcal{G}^\dR(a,b,c;y)\ . 
    \end{split}
    \end{equation*}
    \end{prop}

    \begin{proof}
    By the general formula for the motivic coaction we have, writing $M$ for $M_{a,b,c}(y)$:
    $$\Delta\,\mathcal{F}^\mm(a,b,c;y) = [M,\varphi_1,\eta_1]^\mm\otimes [M,\eta_1^\vee,\eta_1]^\dR + [M,\varphi_1, \eta_2]^\mm\otimes [M,\eta_2^\vee,\eta_1]^\dR\ .$$
    By Lemma \ref{lem: dR pairing 2F1}, the dual basis elements $\eta_1^\vee$ and $\eta_2^\vee$ are represented in $M_{-a,-b,-c}(y)_{\mathrm{dR}}$ by the elements 
    $$\eta_1^\vee = -\frac{b(c-b)}{c}\,\eta_1\quad \mbox{ and }\quad \eta_2^\vee = y^2\,\frac{a(c-a)}{c}\,\eta_2\ .$$
    Thus what remains is to prove the equality $[M,\varphi_1,\eta_2]^\mm=\mathcal{F}^\mm(1+a,1+b,2+c)$.
    We describe an isomorphism
    $$\alpha:M_{1+a,1+b,2+c}(y)\stackrel{\sim}{\longrightarrow} M$$
    in the category $\mathcal{T}$. At the level of de Rham components it is induced by multiplication of (smooth) differential forms by $x(1-x)(1-yx)^{-1}$. The equality 
    $$\nabla_{a,b,c}(x(1-x)(1-yx)^{-1}f)=x(1-x)(1-yx)^{-1}\nabla_{1+a,1+b,2+c}(f)$$
    proves that it induces an isomorphism of de Rham complexes. On the level of Betti components it is induced by multiplication of sections by $x(1-x)(1-yx)^{-1}$. On easily checks that this gives rise to an isomorphism $\alpha$ in the category $\mathcal{T}$. It satisfies $\alpha_{\mathrm{B}}^\vee(\varphi_1)=\varphi_1$ and $\alpha_{\mathrm{dR}}(\eta_1)=\eta_2$, and therefore induces the desired equality of motivic periods.
    \end{proof}

    We now define a motivic version of the function $F$ and compute its motivic coaction. We first need to record a few facts about the motivic and de Rham versions of the beta function. For complex numbers $s_0,s_1\in k_{\mathrm{dR}}$ such that $s_0,s_1,s_0+s_1\notin\mathbb{Z}$ we let $M_{s_0,s_1}=M_{\{0,1\}}(s_0,s_1)\in\mathcal{T}$ and let $\varphi$ denote the class of $(0,1)\otimes x^{s_0}(1-x)^{s_1}\in (M_{s_0,s_1})_{\mathrm{B}}^\vee$ and $\eta$ denote the class of $\frac{dx}{x(1-x)}$ in $(M_{s_0,s_1})_{\mathrm{dR}}$. We thus have the (global) motivic and de Rham beta functions $$\beta^\mm(s_0,s_1)=\left[M_{s_0,s_1},\varphi,\eta\right]^\mm \; \;\in\Pe^\mm_{\mathcal{T}} \qquad \mbox{ and } \qquad \beta^\dR(s_0,s_1)=\left[M_{s_0,s_1},-\eta,\eta\right]^\dR\;\;\in\Pe^\dR_{\mathcal{T}}\ .$$
    They are related to the Lauricella function via $$\beta^\bullet(s_0,s_1)=\frac{s_0+s_1}{s_0s_1}L_{\{0,1\}}^\bullet \qquad (\bullet\in\{\mm,\dR\})$$
    and thus we have the coaction formula:
    \begin{equation}\label{eq: coaction beta in 2F1}
    \Delta\,\beta^\mm(s_0,s_1)=\frac{s_0s_1}{s_0+s_1}\,\beta^\mm(s_0,s_1)\otimes \beta^\dR(s_0,s_1)\ .
    \end{equation}
    
    \begin{lem}\label{lem: beta invertible}
    The matrix coefficients $\beta^\mm(s_0,s_1)\in \mathcal{P}^\mm_{\mathcal{T}}$ and $\beta^\dR(s_0,s_1)\in \mathcal{P}^\dR_{\mathcal{T}}$ are invertible.
    \end{lem}
    
    \begin{proof}
    The object $M=M_{s_0,s_1}\in\mathcal{T}$ has rank one. Thus, the evaluation map $M^\vee\otimes M\stackrel{\sim}{\rightarrow} \mathbbm{1}$ is an isomorphism, where $\mathbbm{1}$ denotes the unit object $(\mathbb{Q}_{\mathrm{B}},k_{\mathrm{dR}},\mathrm{id}_{\mathbb{C}})$ in the Tannakian category $\mathcal{T}$. Since the matrix coefficients $\beta^\mm(s_0,s_1)\in \mathcal{P}^\mm_{\mathcal{T}}$ and $\beta^\dR(s_0,s_1)\in \mathcal{P}^\dR_{\mathcal{T}}$ are defined by non-zero classes, they are invertible and their inverses are matrix coefficients of $M^\vee$.
    \end{proof}
    
    We can thus mimic \eqref{eq: Fabcy} to define (global) motivic and de Rham lifts of $F(a,b,c;y)$.
     
    \begin{defn} \label{globalmot2F1} 
    We define the (global) motivic and de Rham hypergeometric functions
    $$F^\bullet(a,b,c;y) = \beta^\bullet(b,c-b)^{-1}\mathcal{F}^\bullet(a,b,c;y) \;\; \in \mathcal{P}^\bullet_{\mathcal{T}} \qquad (\bullet\in\{\mm,\dR\})\ .$$
    In view of Definition \ref{def: F G sv} we also define:
    $$G^\dR(a,b,c;y) = \frac{a(c-a)}{c}\,\beta^\dR(b,c-b)^{-1}\mathcal{G}^\dR(a,b,c;y) \;\; \in\mathcal{P}^\dR_{\mathcal{T}} .$$
    \end{defn}
    
    Before turning to the motivic coaction on  the motivic hypergeometric function, let us prove a motivic lift of the functional equation of the beta function.
    
    \begin{lem}\label{lem: functional equation beta}
    We have the identities:
    $$\beta^\mm(s_0+1,s_1)=\frac{s_0}{s_0+s_1}\,\beta^\mm(s_0,s_1) \quad \mbox{ and } \quad \beta^\mm(s_0,s_1+1)=\frac{s_1}{s_0+s_1}\,\beta^\mm(s_0,s_1)\ .$$
    \end{lem}
    
    \begin{proof}
    We describe an isomorphism
    $$\alpha:M_{s_0+1,s_1}\stackrel{\sim}{\longrightarrow} M_{s_0,s_1}\ .$$
    At the level of de Rham components it is induced by multiplication of (smooth) differential forms by $x$. The equality $\nabla_{s_0,s_1}(xf)=x\nabla_{s_0+1,s_1}(f)$ proves that this induces an isomorphism of de Rham complexes. At the level of Betti components, it is induced by multiplication by $x$. One easily checks that this gives rise to an isomorphism $\alpha$ in the category $\mathcal{T}$. One checks that it satisfies 
    $$\alpha_{\mathrm{B}}^\vee(\varphi)=\varphi \quad , \;\; \alpha_{\mathrm{dR}}(\eta)=\frac{s_0}{s_0+s_1}\eta \quad\ ,$$
    and the first identity follows. The second is proved in a similar way.
    \end{proof}

    \begin{prop}\label{prop: coaction 2F1}
    The (global) motivic coaction on the motivic hypergeometric function is:
    $$\Delta\,F^\mm(a,b,c;y) = F^\mm(a,b,c;y)\otimes F^\dR(a,b,c;y)  - \frac{y}{1+c}\,F^\mm(1+a,1+b,2+c;y)\otimes  \,  G^\dR(a,b,c;y)\ , 
$$
where the terms $F^{\bullet}(a,b,c;y)$ and $G^\dR(a,b,c;y)$ are as in definition \ref{globalmot2F1}. 
    \end{prop}
    
    \begin{proof}
    The coaction is multiplicative and thus satisfies
    $$\Delta \, F^\mm(a,b,c;y) = (\Delta\,\beta^\mm(b,c-b))^{-1} \, \Delta\,\mathcal{F}^\mm(a,b,c;y)\ .$$
    The result follows from Proposition \ref{prop: coaction Fabc} and \eqref{eq: coaction beta in 2F1} and the equality
    $$\beta^\mm(c-b,b)^{-1} = \frac{b(c-b)}{c(c+1)}\,\beta^\mm(c-b+1,b+1)^{-1}\ ,$$
    which follows from Lemma \ref{lem: functional equation beta}.
    \end{proof}

    \subsection{Motivic coaction for the hypergeometric function: the local point of view}
   
    We now study motivic lifts of the functions $\mathcal{F}(a,b,c;y)$ and $F(a,b,c;y)$ in the local setting, i.e., by lifting the coefficients of their Laurent expansions around $a=b=c=0$ to motivic periods of mixed Tate motives over $k$. Our main theorem is that the global coaction formulae of Proposition \ref{prop: coaction Fabc} and Proposition \ref{prop: coaction 2F1} admit local counterparts. We use the shorthand notation
    $$FL_{i,j}^{\bullet} = (FL_{\{0,1,y^{-1}\}}^\bullet(b,c-b,-a))_{i,j}$$
    for $1\leq i,j\leq 2$ and $\bullet\in \{\mm,\varpi\}$. Here we assume that the class $\gamma_1$ corresponds to the line segment $(0,1)$ and $\gamma_2$ is the line segment $(0,y^{-1})$, assuming that $y\notin \mathbb{R}_{>0}$. 
    We set
    \begin{equation} \label{Flocabcdefn}
    \mathcal{F}^\bullet_{\mathrm{loc}}(a,b,c;y) = \frac{c}{b(c-b)}\,FL^\bullet_{1,1} + \frac{1}{b} \,FL^\bullet_{1,2} \;\;\; \in\;\; \Pe^\bullet_{\MT(k)}[[a,b,c]][b^{-1}(c-b)^{-1}] 
    \end{equation}
    for $\bullet\in\{\mm,\varpi\}$. They are motivic lifts of the Laurent expansions of the functions $\mathcal{F}$ and $\mathcal{F}^{\,\s}$ around $a=b=c=0$, respectively; more precisely, we get from \eqref{eq: expression Fabcy Lauricella} and \eqref{eq: expression Fabcy sv Lauricella} and from Theorems \ref{thm: FL equals L} and \ref{thm: FL equals L sv} the equalities between Laurent series:
    $$\mathrm{per}\,\mathcal{F}^\mm_{\mathrm{loc}}(a,b,c;y) = \mathcal{F}(a,b,c;y) \quad \mbox{ and } \quad \s\,\mathcal{F}^\varpi_{\mathrm{loc}}(a,b,c;y) = \mathcal{F}^{\,\s}(a,b,c;y)\ ,$$
    where the period map $\mathrm{per}:\Pe^\mm_{\MT(k)}\to \mathbb{C}$ and the single-valued period map $\s:\Pe^\varpi_{\MT(k)}\to \mathbb{C}$ are applied term by term to the Laurent series. One can also obtain $\mathcal{F}_{\mathrm{loc}}^\varpi(a,b,c;y)$ by applying the projection $\pi^{\mm,+}_\varpi$ term by term to $\mathcal{F}_{\mathrm{loc}}^\mm(a,b,c;y)$:
    $$\pi^{\mm,+}_\varpi\,\mathcal{F}^\mm_{\mathrm{loc}}(a,b,c;y) = \mathcal{F}^\varpi_{\mathrm{loc}}(a,b,c;y)\ .$$
    We also introduce
    $$\mathcal{F}^\mm_{\mathrm{loc}}(1+a,1+b,2+c;y) = -\frac{1}{y a}\,FL^\mm_{1,2}\ .$$

    \begin{rem}\label{rem: rational structure y}
    Strictly speaking, it does not make sense to multiply a motivic period by $\frac{1}{y}$ since the ring $\Pe^\mm_{\MT(k)}$ is only $\Q$-linear. This is because  the Betti and (canonical) de Rham fiber functors take values in $\Q$-vector spaces. This prefactor comes about in the above definition because the form $\eta_2=\frac{dx}{1-yx}$ is not in fact in the $\Q$-structure $H^1_\varpi(X_\Sigma)$, having residue $-y^{-1}$ at $x=y^{-1}$. However, the form $y \eta_2$ is $\Q$-rational, and therefore the term $y\,\mathcal{F}^\mm_{\mathrm{loc}}(1+a,1+b,2+c;y)$ , which is the one which appears in the coaction formula \eqref{eq: coaction Fabc local} below, is a well-defined series of motivic periods in $\Pe^\mm_{\MT(k)}$, although $\mathcal{F}^\mm_{\mathrm{loc}}(1+a,1+b,2+c;y)$ itself is not.  Alternatively, one could extend scalars and work with a (canonical) de Rham fiber functor valued in $\Q(y)$-vector spaces.
    \end{rem}

    \begin{rem}\label{rem: warning notation shift formal variables}
    A warning about the notation is in order. Note that $\mathcal{F}^\mm_{\mathrm{loc}}(1+a,1+b,2+c;y)$ is \emph{not} obtained from $\mathcal{F}^\mm_{\mathrm{loc}}(a,b,c;y)$ by shifting the formal variables via $(a,b,c)\mapsto (1+a,1+b,2+c)$, which does not make sense in the setting of power series. Rather, the notation is justified by the fact that we have the equality between Laurent series:
    $$\mathrm{per}\,\mathcal{F}^\mm_{\mathrm{loc}}(1+a,1+b,2+c;y) = \mathcal{F}(1+a,1+b,2+c;y)\ ,$$
    which follows from Theorem \ref{thm: FL equals L} and the equality $\frac{dx}{1-yx}=-\frac{1}{ya}(a\omega_2)$. 
    Note that the function $\mathcal{F}(1+a,1+b,2+c;y)$ can also be expressed as a combination of  $\mathcal{F}(a,b,c;y)$ and its derivative  with respect to $y$.
    \end{rem}
   
    We will also need the following definition:
    $$\mathcal{G}^\varpi_{\mathrm{loc}}(a,b,c;y) = \frac{-c}{b(c-a)} \left(  FL^\varpi_{1,1} + \frac{c-b}{c} \, FL^\varpi_{1,2} -  \frac{c}{c-b}\, FL^\varpi_{2,1}  -  \,FL^\varpi_{2,2} \right)\ .$$
    By \eqref{eq: expression Gabcy sv Lauricella} and Theorem \ref{thm: FL equals L sv} we get the equality between Laurent series:
    $$\s\,\mathcal{G}^\varpi_{\mathrm{loc}}(a,b,c;y) = \mathcal{G}^{\,\s}(a,b,c;y)\ .$$

    We now turn to the local coaction formulae. As in \S\ref{subsec: coaction beta quotients} we consider the normalised motivic coaction
    $$\Delta_{\mathrm{nor}}: \Pe^\mm_{\MT(k)}[[a,b,c]][(bc(c-b))^{-1}] \longrightarrow (\Pe^\mm_{\MT(k)}\otimes \Pe^\varpi_{\MT(k)})[[a,b,c]][(bc(c-b))^{-1}]$$
    which consists in applying the coaction of $\MT(k)$ on each coefficient and coacting on the formal variables via
    $$\Delta_{\mathrm{nor}}(a)=a\,(1\otimes (\Lef^\varpi)^{-1}) \;\;\; ,\;\; \Delta_{\mathrm{nor}}(b)=b\,(1\otimes (\Lef^\varpi)^{-1})  \;\;\; ,\;\; \Delta_{\mathrm{nor}}(c)=c\,(1\otimes (\Lef^\varpi)^{-1})\ .$$

    \begin{thm}\label{thm: coaction Fcal}
    We have the following (local) coaction formula:
    \begin{equation}\label{eq: coaction Fabc local}
    \begin{split}
    \Delta_{\mathrm{nor}}\mathcal{F}^\mm_{\mathrm{loc}}&(a,b, c;y) \; = \; \frac{b(c-b)}{c}\; \mathcal{F}^\mm_{\mathrm{loc}}(a,b,c;y)\;\otimes\;  \mathcal{F}^\varpi_{\mathrm{loc}}((\Lef^\varpi)^{-1}a,(\Lef^\varpi)^{-1}b,(\Lef^\varpi)^{-1}c;y) \\ & - y\, \frac{a(c-a)}{c}\; \mathcal{F}^\mm_{\mathrm{loc}}(1+a,1+b,2+c;y)\;\otimes\;  \mathcal{G}^\varpi_{\mathrm{loc}}((\Lef^\varpi)^{-1}a,(\Lef^\varpi)^{-1}b,(\Lef^\varpi)^{-1}c;y)\ . 
    \end{split}
    \end{equation}
    \end{thm}
    
    \begin{proof}
    We drop the Lefschetz de Rham periods $\Lef^\varpi$ from the notation; they can be taken care of by weight considerations as in the proof of Theorem \ref{thmMainCoactionMmm}. On the one hand, by  \eqref{Flocabcdefn} and Theorem \ref{thmMainCoactionMmm} the left-hand side equals:
    $$\frac{c}{b(c-b)}\left(FL^\mm_{1,1}\otimes FL^\varpi_{1,1} + FL^\mm_{1,2}\otimes FL^\varpi_{2,1}\right) + \frac{1}{b} \left(FL^\mm_{1,2}\otimes FL^\varpi_{2,2} + FL^\mm_{1,1} \otimes FL^\varpi_{1,2}\right)\ .$$
    On the other hand the right-hand side is expressed as:
    \begin{equation*}
    \begin{split}
    \frac{b(c-b)}{c}\left(\frac{c}{b(c-b)}FL^\mm_{1,1}+\frac{1}{b}FL^\mm_{1,2}\right) &\otimes\left(\frac{c}{b(c-b)}FL^\varpi_{1,1}+\frac{1}{b}FL^\varpi_{1,2}\right) \\
    -  y\,\frac{a(c-a)}{c} \left(-\frac{1}{ya}FL^\mm_{1,2}\right)\otimes  \left(-\frac{c}{b(c-a)}\right) &\left(  FL^\varpi_{1,1} + \frac{c-b}{c} \, FL^\varpi_{1,2} -  \frac{c}{c-b}\, FL^\varpi_{2,1}  -  \,FL^\varpi_{2,2} \right)\ .
    \end{split}
    \end{equation*}
    By expanding it one sees that simplications occur and that both sides are equal.
    \end{proof}

    It is notable, and \emph{a priori} not at all obvious, that this \emph{local} motivic coaction formula is formally identical to the \emph{global} motivic coaction formula obtained in a different context in Proposition \ref{prop: coaction Fabc}. \medskip
    
    We now turn to the hypergeometric function $F$ and reinsert the beta factors. We recall the  definition of the local motivic beta functions
    $$\beta^\bullet_{\mathrm{loc}}(s_0,s_1) = \frac{s_0+s_1}{s_0s_1}\,\exp\Bigg(\sum_{n\geq 2}\frac{(-1)^{n-1}\zeta^\bullet(n)}{n} ((s_0+s_1)^n-s_0^n-s_1^n)\Bigg) \;\;\in\; \Pe^\bullet_{\MT(k)}[[s_0,s_1]][(s_0s_1)^{-1}]\ ,$$
    for $\bullet\in\{\mm,\varpi\}$. We have the local motivic coaction formula:
    \begin{equation}\label{eq: coaction beta loc nor} 
    \Delta_{\mathrm{nor}}\beta^\mm_{\mathrm{loc}}(s_0,s_1) = \frac{s_0s_1}{s_0+s_1}\beta^\mm_{\mathrm{loc}}(s_0,s_1)\otimes \beta^\varpi_{\mathrm{loc}}((\Lef^\varpi)^{-1}s_0,(\Lef^\varpi)^{-1}s_1)\ .
    \end{equation}
    In view of \eqref{eq: Fabcy} we define
    $$F_{\mathrm{loc}}^\bullet(a,b,c;y) = (\beta^\bullet_{\mathrm{loc}}(b,c-b))^{-1}\,\mathcal{F}^\bullet_{\mathrm{loc}}(a,b,c;y) \;\; \in \; \Pe^\bullet_{\MT(k)}[[a,b,c]][(bc(c-b))^{-1}]\ , $$
    for $\bullet\in\{\mm,\varpi\}$. (The term $\beta^\bullet_{\mathrm{loc}}(b,c-b)$ is invertible since we have inverted $c$ in the ring of power series.) They satisfy the equalities between Laurent series:
    $$\mathrm{per}\,F_{\mathrm{loc}}^\mm(a,b,c;y) = F(a,b,c;y) \qquad \mbox{ and } \qquad \s\,F_{\mathrm{loc}}^\varpi(a,b,c;y) = F^{\,\s}(a,b,c;y)\ ,$$
    where the period map $\mathrm{per}:\Pe^\mm_{\MT(k)}\to \mathbb{C}$ and the single-valued period map $\s:\Pe^\varpi_{\MT(k)}\to \mathbb{C}$ are applied term by term to the Laurent series. One can also obtain $F_{\mathrm{loc}}^\varpi(a,b,c;y)$ by applying the projection $\pi^{\mm,+}_\varpi$ term by term to $F_{\mathrm{loc}}^\mm(a,b,c;y)$:
    $$\pi^{\mm,+}_\varpi\,F_{\mathrm{loc}}^\mm(a,b,c;y) = F_{\mathrm{loc}}^\varpi(a,b,c;y)\ .$$
    
    We also introduce:
    $$F_{\mathrm{loc}}^\mm(1+a,1+b,2+c;y) = \frac{c(c+1)}{b(c-b)}\, (\beta^\mm_{\mathrm{loc}}(b,c-b))^{-1}\,\mathcal{F}^\mm_{\mathrm{loc}}(1+a,1+b,2+c;y)\ .$$
    The same warning as in Remark \ref{rem: warning notation shift formal variables} is in order: since we are working with formal power series, $F_{\mathrm{loc}}^\mm(1+a,1+b,2+c;y)$ is \emph{not} obtained from $F_{\mathrm{loc}}^\mm(a,b,c;y)$ by a shift of  variables, which would not make sense. Rather, the notation is justified by the equality between formal Laurent series:
    $$\mathrm{per}\,F_{\mathrm{loc}}^\mm(1+a,1+b,2+c;y) = F(1+a,1+b,2+c;y)\ ,$$
    which follows from the functional equation $\beta(b+1,c-b+1)=\frac{b(c-b)}{c(c+1)}\beta(b,c-b)$ for the beta function.
    
    Lastly, we introduce
    $$G_{\mathrm{loc}}^\varpi(a,b,c;y) = \frac{a(c-a)}{c}\,(\beta^\varpi_{\mathrm{loc}}(b,c-b))^{-1}\,\mathcal{G}^\varpi_{\mathrm{loc}}(a,b,c;y)\ .$$
    It satisfies
    $$\s\,G_{\mathrm{loc}}^\varpi(a,b,c;y) = G^{\,\s}(a,b,c;y)\ .$$ 
    where $G^{\,\s}(a,b,c;y)$ was defined in Definition \ref{def: F G sv}.

    \begin{thm}\label{thm: coaction local 2F1}
    We have the following (local) motivic coaction  formula for the (local) motivic hypergeometric function:
    \begin{equation}\label{eq: coaction local 2F1}
    \begin{split}
    \Delta_{\mathrm{nor}}\,F_{\mathrm{loc}}^\mm(&a,b,c;y) \;=  \;F_{\mathrm{loc}}^\mm(a,b,c;y) \; \otimes \; F_{\mathrm{loc}}^\varpi\left((\Lef^\varpi)^{-1}a,(\Lef^\varpi)^{-1}b,(\Lef^\varpi)^{-1}c;y\right)\\ & - \frac{y}{1+c}\; F_{\mathrm{loc}}^\mm(1+a,1+b,2+c;y) \; \otimes \;   G^\varpi_{\mathrm{loc}}\left((\Lef^\varpi)^{-1}a, (\Lef^\varpi)^{-1}b, (\Lef^\varpi)^{-1}c; y\right)\ . 
    \end{split}
    \end{equation}
    \end{thm}
    
    \begin{proof}
    It follows from Theorem \ref{thm: coaction Fcal} by using the multiplicativity of the coaction and \eqref{eq: coaction beta loc nor}.
    \end{proof}
    
    Again, it is \emph{a priori} not at all obvious, but certainly true, that this \emph{local} motivic coaction formula is comparable with the \emph{global} motivic coaction formula for the hypergeometric function obtained  in Proposition \ref{prop: coaction 2F1}.
    
    \begin{rem} Up until this point we have fixed a point $y\in k \backslash \{0,1\}$. It is also possible to view $y$ as a variable by working with families; this  will enable us to derive the hypergeometric differential equation from the coaction, which we presently explain. The local motivic and (canonical) de Rham hypergeometric functions can be viewed as power series with coefficients in a ring of families of motivic periods over the base scheme $S= \mathbb{P}^1 \backslash \{0,1,\infty\}$ with coordinate $y$. For this we need to work with a Tannakian category $\mathcal{M}(S)$ of `motives' over $S$, such as the category $\MT(S)$ of mixed Tate motives over $S$ \cite{esnaultlevine} or, for the following application,  the category $\mathcal{H}(S)$ considered in  \cite[\S 7]{brownnotesmot} is sufficient. We regard 
    $F^\bullet_{\mathrm{loc}}$,  and  $G^{\varpi}_{\mathrm{loc}}$ as Laurent series with coefficients in $\Pe^\bullet_{\mathcal{M}(S)}$, where $\bullet = \{\mm, \varpi\}$ relative to suitable fiber functors (whose definition is not important for the following discussion).
    To the lowest two orders, we find that
    $$F_{\mathrm{loc}}^\varpi(a,b,c;y) =\beta^{\varpi}_{\mathrm{loc}}(b,c-b)^{-1}\mathcal{F}^{\varpi}_{\mathrm{loc}}(a,b,c;y)  \ = \ 1  -  \frac{ab}{c} \log^{\varpi} (1-y)  + \cdots   $$ 
    $$G^{\varpi}_{\mathrm{loc}}(a,b,c;y) =\frac{a(c-a)}{c}(\beta^{\varpi}_{\mathrm{loc}}(b,c-b))^{-1}\mathcal{G}^{\varpi}_{\mathrm{loc}}(a,b,c;y) \ =  \ -  \frac{ab(c-a)(c-b)}{c^2} 
    \log^{\varpi}(1-y)   +\cdots  $$ 
    where the $\cdots$ denote iterated integrals of length $\geq 2$, and $\log^{\varpi}$ is the (canonical) de Rham logarithm \cite[\S 5.3, \S 7]{brownnotesmot}.  
    This is consistent with the Laurent expansions \eqref{eq: laurent expansion Fabc}, \eqref{eq: laurent expansion Gabc}. The rings of families of motivic periods $\Pe^\bullet_{\mathcal{M}(S)}$ are equipped with a canonical differential operator  $\nabla_{\partial/\partial y}$ which is compatible with the coaction $\Delta: \Pe^{\mm}_{\mathcal{M}(S)} \rightarrow  \Pe^{\mm}_{\mathcal{M}(S)}  \otimes \Pe^{\varpi}_{\mathcal{M}(S)} $ in the sense that the following equation holds:  
    $$\Delta \circ \nabla_{\partial/\partial y}   =   (\id \otimes \nabla_{\partial/\partial y})\circ \Delta\ . $$
    Furthermore, $\nabla_{\partial/\partial y}$ decreases the length of (canonical) de Rham iterated integrals  by one,   and one checks that 
    $$\nabla_{\partial/\partial y}\log^{\varpi}(1-y) =  -\frac{1}{1-y} \qquad \hbox{ and } \qquad \nabla_{\partial/\partial y}(1)=0 \ . $$ 
     The  ring of (canonical) de Rham periods $\Pe^{\varpi}_{\mathcal{M}(S)}$ is a Hopf algebra with counit $\mathrm{ev}$, 
     such that $(\id \otimes \mathrm{ev}) \Delta=\id$.
It   satisfies $\mathrm{ev}(\mathbb{L}^{\dR})=1$ and annihilates all (canonical) de Rham iterated integrals of length $\geq 1$ which occur in the expansions above. 
       Let us apply $\nabla_{\partial/\partial y}$ to the (unnormalised version of the)  coaction formula  \eqref{eq: coaction local 2F1}. 
        Therefore, by combining \eqref{eq: coaction local 2F1} with the  formulae stated above, we find that 
    \begin{equation*}\begin{split}
    \nabla_{\partial/\partial y} &F^\mm_{\mathrm{loc}}(a,b,c;y)  =  \left( ( \id \otimes \mathrm{ev})   \Delta \right)  \left(\nabla_{\partial/\partial y} F^{\mm}_{\mathrm{loc}}(a,b,c;y)\right) = (\id \otimes \mathrm{ev}\,  \nabla_{\partial/\partial y}) \Delta   F^\mm_{\mathrm{loc}}(a,b,c;y) \\
    &= \frac{ab}{c}\, F^{\mm}_{\mathrm{loc}}(a,b,c;y) \otimes  \frac{1}{1-y} - \frac{y}{1+c} \, F^{\mm}_{\mathrm{loc}}(1+a,1+b,2+c;y) \otimes \frac{ab(c-a)(c-b)}{c^2} \frac{1}{1-y}\ .   
    \end{split}
    \end{equation*}
 Since the connexion is compatible with the period homomorphism, we deduce that 
  $$\frac{\partial}{\partial y}F(a,b,c;y) = \frac{ab}{c}\frac{1}{1-y}\,F(a,b,c;y) - \frac{ab(c-a)(c-b)}{c^2(1+c)}\frac{y}{1-y} \, F(1+a,1+b,2+c;y)\ \cdot$$
as formal power series  near $a=b=c=0$. One verifies that this is equivalent, by contiguity relations, to
 $$\frac{\partial}{\partial y}F(a,b,c;y) = \frac{ab}{c}\,F(a+1,b+1,c+1;y)\ ,$$
 or to the hypergeometric differential equation \eqref{eq: 2F1 differential equation}. 
 
\end{rem}

\begin{rem}
It is also a general fact that the (local) motivic coaction is compatible with the monodromy of the (local) hypergeometric function. This, as well as the connection above, provides \emph{a priori} constraints on the shape of the motivic coaction.
\end{rem}

\bibliographystyle{alpha}

\bibliography{biblio}

\newcommand{\etalchar}[1]{$^{#1}$}
\begin{thebibliography}{ABDG17b}

\bibitem[ABD{\etalchar{+}}19]{abreuetaldiagrammatictwoloop}
S.~Abreu, R.~Britto, C.~Duhr, E.~Gardi, and J.~Matthew.
\newblock Diagrammatic coaction of two-loop {F}eynman integrals.
\newblock In {\em 14th International Symposium on Radiative Corrections}, 2019.

\bibitem[ABD{\etalchar{+}}20]{abreuetalpositive}
S.~Abreu, R.~Britto, C.~Duhr, E.~Gardi, and J.~Matthew.
\newblock From positive geometries to a coaction on hypergeometric functions.
\newblock {\em Journal of High Energy Physics}, 2020(2):1--45, 2020.

\bibitem[ABDG17a]{abreuetalalgebraicstructure}
S.~Abreu, R.~Britto, C.~Duhr, and E.~Gardi.
\newblock Algebraic structure of cut {F}eynman integrals and the diagrammatic
  coaction.
\newblock {\em Physical {R}eview {L}etters}, 119(5):051601, 2017.

\bibitem[ABDG17b]{abreuetaldiagrammatichopf}
S.~Abreu, R.~Britto, C.~Duhr, and E.~Gardi.
\newblock Diagrammatic {H}opf algebra of cut {F}eynman integrals: the one-loop
  case.
\newblock {\em Journal of High Energy Physics}, 2017(12):90, 2017.

\bibitem[Aom75]{aomotovanishing}
K.~Aomoto.
\newblock On vanishing of cohomology attached to certain many valued
  meromorphic functions.
\newblock {\em Journal of the Mathematical Society of Japan}, 27(2):248--255,
  1975.

\bibitem[Aom87]{aomotoselberg}
K.~Aomoto.
\newblock On the complex {S}elberg integral.
\newblock {\em The Quarterly Journal of Mathematics}, 38(4):385--399, 1987.

\bibitem[BD21a]{BD1}
F.~Brown and C.~Dupont.
\newblock Single-valued integration and double copy.
\newblock {\em J. Reine Angew. Math.}, 775:145--196, 2021.

\bibitem[BD21b]{BD2}
F.~Brown and C.~Dupont.
\newblock Single-valued integration and superstring amplitudes in genus zero.
\newblock {\em Comm. Math. Phys.}, 382(2):815--874, 2021.

\bibitem[BGF17]{burgosfresanbook}
J.~I. Burgos~Gil and J.~Fres{\'a}n.
\newblock {\em Multiple zeta values: from numbers to motives}.
\newblock Clay Mathematics Proceedings, to appear, 2017.

\bibitem[BPZ84]{BPZ}
A.~A. Belavin, A.~M. Polyakov, and A.~B. Zamolodchikov.
\newblock Infinite conformal symmetry in two-dimensional quantum field theory.
\newblock {\em Nuclear Phys. B}, 241(2):333--380, 1984.

\bibitem[Bro14a]{brownICM}
F.~Brown.
\newblock Motivic periods and {$\mathbb{P}^1\backslash \{0,1,\infty\}$}.
\newblock In {\em Proceedings of the {I}nternational {C}ongress of
  {M}athematicians -- {S}eoul 2014. {V}ol. {II}}, pages 295--318. Kyung Moon
  Sa, Seoul, 2014.

\bibitem[Bro14b]{brownSVMZV}
F.~Brown.
\newblock Single-valued motivic periods and multiple zeta values.
\newblock {\em Forum Math. Sigma}, 2:e25, 37, 2014.

\bibitem[Bro17]{brownnotesmot}
F.~Brown.
\newblock Notes on motivic periods.
\newblock {\em Commun. Number Theory Phys.}, 11(3):557--655, 2017.

\bibitem[BSST14]{BroedelSchlottererStiebergerTerasoma}
J.~Broedel, O.~Schlotterer, S.~Stieberger, and T.~Terasoma.
\newblock All order $\alpha'$-expansion of superstring trees from the
  {D}rinfeld associator.
\newblock {\em Phys. Rev. D}, 89:066014, Mar 2014.

\bibitem[CM95]{ChoMatsumoto}
K.~Cho and K.~Matsumoto.
\newblock Intersection theory for twisted cohomologies and twisted {R}iemann's
  period relations. {I}.
\newblock {\em Nagoya Math. J.}, 139:67--86, 1995.

\bibitem[Del70]{deligneEqDiff}
P.~Deligne.
\newblock {\em \'{E}quations diff\'{e}rentielles \`a points singuliers
  r\'{e}guliers}.
\newblock Lecture Notes in Mathematics, Vol. 163. Springer-Verlag, Berlin-New
  York, 1970.

\bibitem[DF85]{DotsenkoFateev}
Vl.~S. Dotsenko and V.~A. Fateev.
\newblock Four-point correlation functions and the operator algebra in {$2$}{D}
  conformal invariant theories with central charge {$C\leq 1$}.
\newblock {\em Nuclear Phys. B}, 251(5-6):691--734, 1985.

\bibitem[DG05]{delignegoncharov}
P.~Deligne and A.~B. Goncharov.
\newblock Groupes fondamentaux motiviques de {T}ate mixte.
\newblock {\em Ann. Sci. \'{E}cole Norm. Sup. (4)}, 38(1):1--56, 2005.

\bibitem[DM86]{DeligneMostow}
P.~Deligne and G.~D. Mostow.
\newblock Monodromy of hypergeometric functions and nonlattice integral
  monodromy.
\newblock {\em Inst. Hautes \'{E}tudes Sci. Publ. Math.}, 63:5--89, 1986.

\bibitem[Dri90]{Drinfeld}
V.~G. Drinfeld.
\newblock On quasitriangular quasi-{H}opf algebras and on a group that is
  closely connected with {$\mathrm{Gal}(\overline{\mathbb{Q}}/{\mathbb{ Q}})$}.
\newblock {\em Algebra i Analiz}, 2(4):149--181, 1990.

\bibitem[EL07]{esnaultlevine}
H.~Esnault and M.~Levine.
\newblock Tate motives and the fundamental group.
\newblock {\em arXiv preprint arXiv:0708.4034}, 2007.

\bibitem[Enr06]{EnriquezGamma}
B.~Enriquez.
\newblock On the {D}rinfeld generators of {$\mathfrak{grt}_1(k)$} and
  {$\Gamma$}-functions for associators.
\newblock {\em Math. Res. Lett.}, 13(2-3):231--243, 2006.

\bibitem[ESV92]{esnaultschechtmanviehweg}
H.~{Esnault}, V.~{Schechtman}, and E.~{Viehweg}.
\newblock {Cohomology of local systems on the complement of hyperplanes}.
\newblock {\em Inventiones Mathematicae}, 109:557--561, 1992.

\bibitem[Got15]{gotolauricella}
Y.~Goto.
\newblock Twisted period relations for {L}auricella's hypergeometric functions
  ${F}_{A}$.
\newblock {\em Osaka Journal of mathematics}, 52(3):861--879, 2015.

\bibitem[HK74]{hattorikimura}
A.~Hattori and T.~Kimura.
\newblock On the {E}uler integral representations of hypergeometric functions
  in several variables.
\newblock {\em Journal of the Mathematical Society of Japan}, 26(1):1--16,
  1974.

\bibitem[HY99]{hanamurayoshida}
M.~Hanamura and M.~Yoshida.
\newblock Hodge structure on twisted cohomologies and twisted {R}iemann
  inequalities {I}.
\newblock {\em Nagoya Mathematical Journal}, 154:123–139, 1999.

\bibitem[IKY87]{UniversalJacobi}
Y.~Ihara, M.~Kaneko, and A.~Yukinari.
\newblock On some properties of the universal power series for {J}acobi sums.
\newblock In {\em Galois representations and arithmetic algebraic geometry
  ({K}yoto, 1985/{T}okyo, 1986)}, volume~12 of {\em Adv. Stud. Pure Math.},
  pages 65--86. North-Holland, Amsterdam, 1987.

\bibitem[KLT86]{klt}
H.~Kawai, D.~C. Lewellen, and S.-H.~H. Tye.
\newblock A relation between tree amplitudes of closed and open strings.
\newblock {\em Nuclear Physics B}, 269(1):1 -- 23, 1986.

\bibitem[KN83]{kitanoumi}
M.~Kita and M.~Noumi.
\newblock On the structure of cohomology groups attached to the integral of
  certain many-valued analytic functions.
\newblock {\em Japanese journal of mathematics. New series}, 9(1):113--157,
  1983.

\bibitem[KY94a]{kitayoshida}
M.~Kita and M.~Yoshida.
\newblock Intersection {T}heory for {T}wisted {C}ycles.
\newblock {\em Mathematische Nachrichten}, 166(1):287--304, 1994.

\bibitem[KY94b]{kitayoshida2}
M.~Kita and M.~Yoshida.
\newblock Intersection {T}heory for {T}wisted {C}ycles {I}{I} -- {D}egenerate
  {A}rrangements.
\newblock {\em Mathematische Nachrichten}, 168(1):171--190, 1994.

\bibitem[KZ84]{KZ}
V.~G. Knizhnik and A.~B. Zamolodchikov.
\newblock Current algebra and {W}ess-{Z}umino model in two dimensions.
\newblock {\em Nuclear Phys. B}, 247(1):83--103, 1984.

\bibitem[KZ01]{kontsevichzagier}
M.~Kontsevich and D.~Zagier.
\newblock Periods.
\newblock In {\em Mathematics unlimited -- 2001 and beyond}, pages 771--808.
  Springer, 2001.

\bibitem[Li10]{LiGamma}
Z.~Li.
\newblock Gamma series associated to elements satisfying regularized double
  shuffle relations.
\newblock {\em J. Number Theory}, 130(2):213--231, 2010.

\bibitem[LS91]{loesersabbah}
F.~Loeser and C.~Sabbah.
\newblock Equations aux diff{\'e}rences finies et d{\'e}terminants
  d'int{\'e}grales de fonctions multiformes.
\newblock {\em Commentarii Mathematici Helvetici}, 66(1):458--503, 1991.

\bibitem[Mat98]{KMatsumoto}
K.~Matsumoto.
\newblock Intersection numbers for logarithmic $k$-forms.
\newblock {\em Osaka J. Math.}, 35(4):873--893, 1998.

\bibitem[Mat18]{Matthes}
N.~Matthes.
\newblock The meta-abelian elliptic {KZB} associator and periods of
  {E}isenstein series.
\newblock {\em Selecta Math. (N.S.)}, 24(4):3217--3239, 2018.

\bibitem[Mim18]{mimachicomplex}
K.~Mimachi.
\newblock Complex hypergeometric integrals.
\newblock In {\em Representation Theory, Special Functions and Painlev{\'e}
  Equations -- RIMS 2015}, pages 469--485. Mathematical Society of Japan, 2018.

\bibitem[Miz17]{mizera}
S.~Mizera.
\newblock Combinatorics and topology of {K}awai--{L}ewellen--{T}ye relations.
\newblock {\em Journal of High Energy Physics}, 2017(8):97, Aug 2017.

\bibitem[MY03]{mimachiyoshida}
K.~{Mimachi} and M.~{Yoshida}.
\newblock Intersection {N}umbers of {T}wisted {C}ycles and the {C}orrelation
  {F}unctions of the {C}onformal {F}ield {T}heory.
\newblock {\em Communications in Mathematical Physics}, 234:339--358, 2003.

\bibitem[Nak95]{Nakamura}
H.~Nakamura.
\newblock On exterior {G}alois representations associated with open elliptic
  curves.
\newblock {\em J. Math. Sci. Univ. Tokyo}, 2(1):197--231, 1995.

\bibitem[SS13]{schlottererstiebergermotivic}
O.~Schlotterer and S.~Stieberger.
\newblock Motivic multiple zeta values and superstring amplitudes.
\newblock {\em Journal of Physics A: Mathematical and Theoretical},
  46(47):475401, 2013.

\bibitem[SS18]{schlottererschnetz}
O.~Schlotterer and O.~Schnetz.
\newblock {C}losed strings as single-valued open strings: {A} genus-zero
  derivation.
\newblock {\em Preprint: arXiv:1808.00713}, 2018.

\bibitem[VZ18]{vanhovezerbini}
P.~Vanhove and F.~Zerbini.
\newblock Closed string amplitudes from single-valued correlation functions.
\newblock {\em Preprint: arXiv:1812.03018}, 2018.

\end{thebibliography}

\end{document}